\documentclass[11pt]{amsart}

\usepackage{enumitem}
\usepackage[
paper=letterpaper,
text={147mm,230mm},centering
]{geometry}
\usepackage{amssymb}
\usepackage{cancel}
\usepackage{csquotes}
\usepackage[dvipsnames]{xcolor}
\usepackage{marginnote}
\usepackage{color}
\usepackage{tabto}

\usepackage[color=yellow!90!black]{todonotes}
\usepackage{xparse}

\usepackage{placeins}
\usepackage{enumerate}

\usepackage{tikz}
\usetikzlibrary{decorations.markings}
\usetikzlibrary{positioning}
\usepackage[bookmarks]{hyperref}
\hypersetup{colorlinks=true,linkcolor=blue,citecolor=blue}

\iftrue
\makeatletter
\def\@settitle{%
  \vspace*{-20pt}
  \begin{flushleft}%
    \baselineskip14\p@\relax
    \normalfont\bfseries\LARGE
    \@title
  \end{flushleft}%
}
\def\@setauthors{%
  \begingroup
  \def\thanks{\protect\thanks@warning}%
  \trivlist
  \large \@topsep30\p@\relax
  \advance\@topsep by -\baselineskip
  \item\relax
  \author@andify\authors
  \def\\{\protect\linebreak}%
  \authors
  \ifx\@empty\contribs
  \else
    ,\penalty-3 \space \@setcontribs
    \@closetoccontribs
  \fi
  \normalfont
  \endtrivlist
  \endgroup
}
\def\@setaddresses{\par
  \nobreak \begingroup\raggedright
  \small
  \def\author##1{\nobreak\addvspace\smallskipamount}%
  \def\\{\unskip, \ignorespaces}%
  \interlinepenalty\@M
  \def\address##1##2{\begingroup
    \par\addvspace\bigskipamount\noindent
    \@ifnotempty{##1}{(\ignorespaces##1\unskip) }%
    {\ignorespaces##2}\par\endgroup}%
  \def\curraddr##1##2{\begingroup
    \@ifnotempty{##2}{\nobreak\noindent\curraddrname
      \@ifnotempty{##1}{, \ignorespaces##1\unskip}\/:\space
      ##2\par}\endgroup}%
  \def\email##1##2{\begingroup
    \@ifnotempty{##2}{\smallskip\nobreak\noindent E-mail address%
      \@ifnotempty{##1}{, \ignorespaces##1\unskip}\/:\space
      \ttfamily##2\par}\endgroup}%
  \def\urladdr##1##2{\begingroup
    \def~{\char`\~}%
    \@ifnotempty{##2}{\nobreak\noindent\urladdrname
      \@ifnotempty{##1}{, \ignorespaces##1\unskip}\/:\space
      \ttfamily##2\par}\endgroup}%
  \addresses
  \endgroup
  \global\let\addresses=\@empty
}
\def\@setabstracta{%
    \ifvoid\abstractbox
  \else
    \skip@25\p@ \advance\skip@-\lastskip
    \advance\skip@-\baselineskip \vskip\skip@
    \box\abstractbox
    \prevdepth\z@ 
    \vskip-10pt
  \fi
}

\let\oldtocsection=\tocsection
\let\oldtocsubsection=\tocsubsection
\let\oldtocsubsubsection=\tocsubsubsection
\renewcommand{\tocsection}[2]{\hspace{0em}\oldtocsection{#1}{#2}}
\renewcommand{\tocsubsection}[2]{\hspace{2em}\oldtocsubsection{#1}{#2}}
\renewcommand{\tocsubsubsection}[2]{\hspace{4em}\oldtocsubsubsection{#1}{#2}}

\def\section{\@startsection{section}{1}%
  \z@{-1.2\linespacing\@plus-.5\linespacing}{.8\linespacing}%
  {\normalfont\bfseries\large}}
\def\subsection{\@startsection{subsection}{2}%
  \z@{-.8\linespacing\@plus-.3\linespacing}{.3\linespacing\@plus.2\linespacing}%
  {\normalfont\bfseries}}
\def\subsubsection{\@startsection{subsubsection}{3}%
  \z@{.7\linespacing\@plus.1\linespacing}{-1.5ex}%
  {\normalfont\itshape}}
\def\@secnumfont{\bfseries}
\makeatother
\fi 

\usepackage{tikz}
\usetikzlibrary{decorations.markings}
\usetikzlibrary{positioning}
\usepackage[bookmarks]{hyperref}
\hypersetup{colorlinks=true,linkcolor=blue,citecolor=blue}

\usepackage{mathptmx}
\usepackage{amssymb}
\usepackage[all]{xy}
\usepackage{graphicx,color,float}
\usepackage[bookmarks]{hyperref}
\usepackage{pinlabel}
\usepackage{amsmath}
\usepackage{mathtools}
\usepackage{tikz-cd}
\usepackage{adjustbox}
\usepackage{pb-diagram,pb-xy}
\textwidth=\paperwidth \advance\textwidth by-2\oddsidemargin
\advance\oddsidemargin by-1in
\evensidemargin=\oddsidemargin
\calclayout
\def\A{\mathcal{A}}

\def\N{\mathbb{N}}
\def\Z{\mathbb{Z}}
\def\Q{\mathbb{Q}}
\def\R{\mathbb{R}}
\def\C{\mathbb{C}}

\def\d{\partial}

\def\sign{\operatorname{sign}}

\def\id{\mathrm{id}}

\def\int{\operatorname{int}}
\def\+{\oplus}
\def\hat{\widehat}

\def\proj{\operatorname{proj}}

\def\colim{\operatorname{colim}}
\def\k{\textbf{k}}
\def\sup{\textup{sup}}
\def\min{\textup{min}}

\theoremstyle{plain}
\newtheorem{theorem}{Theorem}[section]

\newtheorem{lemma}[theorem]{Lemma}
\newtheorem{corollary}[theorem]{Corollary}

\newtheorem{conjecture}[theorem]{Conjecture}
\theoremstyle{definition}
\newtheorem{definition}[theorem]{Definition}
\newtheorem{remark}[theorem]{Remark}

\newtheorem{example}[theorem]{Example}

\theoremstyle{definition}

\theoremstyle{remark}

\newcommand{\nonamethmname}{}

\NewDocumentEnvironment{genthm}{O{plain}m}
 {\renewcommand{\nonamethmname}{#2}\begin{nonamethm#1}}
 {\end{nonamethm#1}}
\NewDocumentEnvironment{genthm*}{O{plain}mo}
 {\renewcommand{\nonamethmname}{#2}%
  \IfNoValueTF{#3}
    {\begin{nonamethm#1}\relax}%
    {\begin{nonamethm#1}[#3]}%
  \mbox{}}
 {\end{nonamethm#1}}
\def\to{\mathchoice{\longrightarrow}{\rightarrow}{\rightarrow}{\rightarrow}}
\makeatletter
\newcommand{\shortxra}[2][]{\ext@arrow 0359\rightarrowfill@{#1}{#2}}
\def\longrightarrowfill@{\arrowfill@\relbar\relbar\longrightarrow}
\newcommand{\longxra}[2][]{\ext@arrow 0359\longrightarrowfill@{#1}{#2}}

\DeclareRobustCommand{\coprod}{\mathop{\text{\fakecoprod}}}
\newcommand{\fakecoprod}{%
  \sbox0{$\prod$}%
  \smash{\raisebox{\dimexpr.9625\depth-\dp0}{\scalebox{1}[-1]{$\prod$}}}%
  \vphantom{$\prod$}%
}

\renewcommand{\xrightarrow}[2][]{\mathchoice{\longxra[#1]{#2}}%
  {\shortxra[#1]{#2}}{\shortxra[#1]{#2}}{\shortxra[#1]{#2}}}
\makeatother

\makeatletter
\@namedef{subjclassname@2020}{%
  \textup{2020} Mathematics Subject Classification}
\makeatother

\begin{document}

\title
{Enhanced bounds for rho-invariants for both general and spherical 3-manifolds}


\author{Geunho Lim}
\address{Department of Mathematics, University of California, Santa Barbara, California, United States}
\curraddr{}
\email{limg@ucsb.edu}
\thanks{}

\maketitle

\section*{Abstract}

We establish enhanced bounds on Cheeger-Gromov $\rho$-invariants for general $3$-manifolds and yet stronger bounds for special classes of $3$-manifold. As key ingredients, we construct chain null-homotopies whose complexity is linearly bounded by its boundary's. This result can be regarded as an algebraic topological analogue of Gromov's conjecture for quantitative topology. The author hopes for applications to various fields including the smooth knot concordance group, quantitative topology, and complexity theory.


\tableofcontents

\section{Introduction}

\subsection{Background}

Cheeger and Gromov introduce the $L^2$ $\rho$-invariant (or Cheeger-Gromov $\rho$-invariant) in~\cite{CG}, defined on a Riemannian, closed, oriented $(4k-1)$-manifold $M$ endowed with an arbitrary representation $\varphi\colon \pi_1(M)\to G$. They analytically show the existence of universal bounds for $\rho$-invariants of Riemannian $3$-manifolds.

Following this work, Chang and Weinberger apply the $L^2$-Index theorem to {\em topologically} describe the Cheeger-Gromov $\rho$-invariants~\cite{CW}. In fact, they extend the definition of $\rho$-invariants to topological manifolds. Moreover, using their topological definition of $\rho$-invariants, they show that if the fundamental group of a $(4k-1)$-manifold $M$ is not torsion-free, then there are infinitely many manifolds which are simple homotopy equivalent to $M$, but not homeomorphic to it. 

In~\cite{Cha16}, using the Chang-Weinberger approach, Cha proves the existence of universal bounds for all $L^2$ $\rho$-invariants of any topological $(4k-1)$-manifold. He then proceeds to refine these bounds as functions of the simplicial complexity of $3$-manifolds. (Recall the simplicial complexity of a $3$-manifold $M$ is the minimal number of $3$-simplices in a triangulation of $M$.)

To determine these refined bounds on the $L^2$ $\rho$-invariant, Cha begins by following Chang-Weinberger,  embedding the $3$-manifold group $G$ in an acyclic group.  (Baumslag-Dyer-Heller constructed such a group which we will call {\em an acyclic container} or BDH-acyclic group~\cite{BDH}.) Cha finds a $4$-chain in the chain complex of this acyclic group with boundary representing the image of the fundamental class of the $3$-manifold $M$. Using this $4$-chain he constructs a null-bordism of $M$ over the BDH-acyclic group of the group of $M$. The number of $2$-handles of this $4$-manifold bounds the $L^2$ signature of the $4$-manifold. By the $L^2$-Index Theorem, this bounds the $L^2$ $\rho$-invariant of $M$.

To count the number of $2$-handles, Cha  constructs a small and uniformly controlled null-homotopy of the chain map induced by the inclusion of the $3$-manifold group into the its acyclic container. Thus, he obtains explicit universal bounds for the $L^2$-signature of that $4$-manifold as a function of the simplicial complexity of the bounding $3$-manifold.

\begin{theorem}\label{thm:Cha} 
\textup{(Cha~\cite[Theorem 1.5]{Cha16}).} Suppose $M$ is a closed, oriented $3$-manifold with simplicial complexity $n$. Then, 
\begin{center}
$\lvert$ $\rho^{(2)}(M,\varphi)$ $\rvert$ $\leq$ $363090$ $\cdot$ $n$  
\end{center}
for any homomorphism $\varphi\colon \pi_1(M)\to G$ to any group G.
\end{theorem}
\noindent Cha's construction of a controlled chain homotopy is motivated by the Baumslag-Dyer-Heller proof of the acyclicity of the acyclic container. 

In this paper, we establish stronger bounds on Cheeger-Gromov $\rho$-invariants as a function of the simplicial complexity of $M$ for both general $3$-manifolds and special classes of $3$-manifolds. As a key ingredient, we construct a new and more economic chain homotopies without computer assistance. This contrasts with constructions in Cha in~\cite{Cha16}.


\noindent {\bf Note:} {\em Throughout this paper, we  state our results in conjunction with those of Cha~\cite{Cha16} to aid the reader in gauging progress.}

\subsection{Main results : Stronger bounds on \texorpdfstring{$\rho$}{p}-invariants of \texorpdfstring{$3$}{3}-manifolds}

Our main theorem below, improves Theorem~\ref{thm:Cha} of Cha by roughly a factor of two.

The following theorem holds for all orientable closed $3$-manifolds.

\begin{theorem}\label{thm:general}
 Suppose $M$ is a closed, oriented $3$-manifold with simplicial complexity $n$. Then, 
\begin{center}
$\lvert$ $\rho^{(2)}(M,\varphi)$ $\rvert$ $\leq$ $189540$ $\cdot$ $n$ 
\end{center}
for any homomorphism $\varphi\colon \pi_1(M)\to G$ to any group G.
\end{theorem}

To achieve this, we construct a far more economic chain homotopy than in~\cite{Cha16}, guided directly by the relations of the acyclic container and no longer following the construction in~\cite{BDH}.

As a key secondary idea, we focus on non-degenerate simplices. Specifically, we observe that degenerate simplices play no role in the Cha's construction of a null-bordism of our $3$-manifold $M$.  This idea was omitted or overlooked in~\cite{Cha16}.  

The direct manner in which we constructed a chain null-homotopy allows one to easily count the number of degenerate simplices used.   These two ingredients, our economic chain null-homotopy together and the readily computable count of non-degenerate simplices, combine to produce the efficient bounds in Theorem~\ref{thm:general}.

Furthermore, Cha's explicit chain homotopy is computer dependent. The alternative chain homotopy we create has easily computed analogues in all dimensions, potentially leading to new results of bounds for $L^2$ $\rho$-invariants in high dimensional manifolds as well.

We also begin investigating Cheeger-Gromov bounds for classes of $3$-manifolds within a fixed geometric type.  Most notably, we have the following result for spherical $3$-manifolds (space forms), a $99.35\%$ reduction from Cha's Theorem~\ref{thm:Cha} given above.

\begin{theorem}\label{thm:space forms}
Let $M$ be a closed, oriented $3$-manifold with simplicial complexity $n$. Suppose $M$ is a spherical space form. Then,
\begin{center}
$\lvert$ $\rho^{(2)}(M,\varphi)$ $\rvert$ $\leq$ $2340$ $\cdot$ $n$ 
\end{center}
for any homomorphism $\varphi\colon \pi_1(M)\to G$ and any group G.
\end{theorem}

In a key step in the proof of Theorem~\ref{thm:general}, we construct a $4$-chain which bounds the fundamental class of $M$. To prove Theorem~\ref{thm:space forms}, we construct a rationalized $4$-chain which still allows us to compute $L^2$ bounds.

When we consider a representation induced by a simplicial-cellular map which is, roughly speaking, a cellular map sending simplices to simplices linearly (see Definition~\ref{def:simplicial-cellular}), our methods extend to prove the following, perhaps surprising, theorem. 

\begin{theorem}\label{thm:RelativeRes} 
Let $M$ and $N$ be closed, oriented, triangulated $3$-manifolds. Assume $M$ has the simplicial complexity $n$ and $N$ is a spherical space form. Suppose $f \colon M \to N$ is a simplicial-cellular map with degree 1. Then,
\begin{center}
$\lvert$ $\rho^{(2)}(M, f_\ast)$ $\rvert$ $\leq$ $2340$ $\cdot$ $n$
\end{center}
where $f_\ast \colon \pi_1(M) \to \pi_1(N)$ is the induced homomorphism by $f$.
\end{theorem}

Theorem~\ref{thm:RelativeRes} is an unexpected extension of Theorem~\ref{thm:space forms} because we make {\em no assumptions} on the sphericity of  $M$.  For instance, Luft and Sjerve~\cite{LS} construct homology spheres with infinite fundamental group which satisfy the hypothesis of Theorem~\ref{thm:RelativeRes} from any $2n \times 2n$ matrix $A$ with determinant $1$ and such that $A^2 - I$ is invertible.  They give examples for $n = 2, 3$, but presumably such examples exist for all $n \geq 2$. 

Theorem~\ref{thm:RelativeRes} does not provide universal bounds since the homomorphism $f_{\ast}$ is not necessarily an inclusion. However, Theorem~\ref{thm:RelativeRes} includes Theorem~\ref{thm:space forms} as a special case, by letting $M$ be spherical and $f_{\ast}$ the identity.

\subsection{Motivation}

An interesting implication of our new bounds on resent results can be found in the study of the knot concordance group in the smooth category. Recently the smooth concordance group of topologically slice knots is investigated via conjectural primary decomposition by Cha~\cite{Cha19}. Cha's explicit universal bounds in Theorem~\ref{thm:Cha} are used as an obstruction to construct a large subgroup of the smooth concordance group of topologically slice knots for which the prime decomposition conjectures is confirmed. Our stronger bounds on Cheeger-Gromov invariants tend to make the subgroup larger along Cha's construction. For details, we refer readers to~\cite[Section 2.1]{Cha19}.

A quantitative topological viewpoint can be another huge motivation of our research. In ~\cite{Gro99}, Gromov raises fundamental questions on quantitative topology. For a given null-cobordant Riemannian $n$-manifold with the geometric complexity, Gromov conjectured the minimal geometric complexity of a null-cobordism is linearly bounded by the geometric complexity of the given $n$-manifold. Chambers, Dotterrer, Manin, and Weinberger~\cite{CDMW17} shows the bounds is at most a polynomial whose degree depends on $n$. Theorem~\ref{thm:kill}, one of our fundamental theorems, says that for an $m$-chain with the simplicial complexity $n$ there exists a null-homotopy of the embedding of the chain into an acyclic container with the simplicial complexity at most $c(m) \cdot n$ where $c(m)$ is a constant which depends on only the dimension $m$ of the given chain. This result can be regarded as an algebraic topological (or homological algebraic) analogue of Gromov's conjecture. Furthermore, while the null-homotopy Cha constructed through dimension 4 is machine generated, our chain homotopy is easily computed in all dimensions by hand. This potentially leads to new results of bounds for $L^2$ $\rho$-invariants in high dimensional manifolds as well. Moreover our new null-homotopy provides an explicit recurrence formula of $c(m)$ so that we can study an estimate of the asymptotic growth rate of the constant $c(m)$ which depends on dimension (See Remark~\ref{rmk:growth}). For details about quantitative topology, we refer readers to~\cite{Man19}, ~\cite{CDMW17}, and ~\cite{Gut17}.




There is a direct application of our new bounds in the complexity theory of $3$-manifold. By following~\cite{Cha16}, we apply Theorem~\ref{thm:space forms} to $L^2$ $\rho$-invariants of lens spaces $L(n,1)$ and obtain lower bounds for $c(L(n,1))$, the pseudo-simplicial complexity of $L(n,1)$. (The integer $c(M)$ is defined to be the minimal number of $3$-simplices in a pseudo-simplicial triangulation of $M$. See Definition~\ref{def:pseudo-simplicial triangulation} and Definition~\ref{def:pseudo-simplicial complexity}. For the relation between the simplicial complexity and the pseudo-simplicial complexity, see Remark~\ref{rmk:pseudo-simplicial complexity}.)
\begin{theorem}\label{thm:complexity}
For each $n > 3$,
\[
\frac{1}{4043520} \cdot (n-3) \leq c(L(n,1)) \leq n-3.
\]
\end{theorem}

This lower bound is roughly $155$ times larger than the lower bound derived by Jae Choon Cha in~\cite{Cha16}.

Bounds on $\rho$-invariants play a role in a number of results already in the literature, including results in~\cite{Cha14a},~\cite{Cha14b}, \cite{Cha16}, \cite{CFP14}, \cite{CP14}, \cite{CHL08}, \cite{CHL09}, \cite{CHL11}, \cite{CT}, \cite{Fra}, and~\cite{Kim}, among others. The bounds give explicit examples in geometric topology including knot theory and the theory of $3$-manifolds. For more details we refer the reader to~\cite[Remark 6.6]{Cha16}. There is considerable room to explore the implication of our new bounds on past results.

\noindent {\bf Organization of the paper:} In Chapter 2, we review the topological definition of the $L^2$ $\rho$-invariants and recall the Moore complex of simplicial classifying spaces, controlled chain homotopy, simplicial-cellular complexes, and BDH-acyclic group, which we use to prove our main theorems. In Chapter 3, we outline the proof of Cha's Theorem~\ref{thm:Cha}. In Chapter 4, we introduce and prove our main theorems. To prove these theorems, we give a new chain homotopy, now guided primarily by the relations in the Baumslag-Dyer-Heller acyclic container. To construct the chain homotopy, we use edgewise subdivisions and introduce what we call {\em simplicial cylinders}. In Chapter 5, we construct a rationalized $4$-chain for spherical $3$-manifolds to compute way stronger bounds for the space forms. In Chapter 6, we investigate the impact of our results by revising results of Cha accordingly.

\noindent {\bf Acknowledgements:} This paper is based on my doctoral thesis. I would like to express my deepest gratitude to my advisor Professor Kent Orr for his thoughtful guidance. I also wish to thank Professor Jae Choon Cha for his support. I am indebted to Professor Fedor Manin, Professor Min Hoon Kim, and Homin Lee for helpful conversations.

\section{Preliminaries}
In this chapter we review Chang-Weinberger's topological definition of the $L^2$ $\rho$-invariant~\cite{CW}, the Moore complex of a simplicial classifying space, controlled chain homotopies, simplicial-cellular complexes, and the work of Gilbert Baumslag, Eldon Dyer, and Alex Heller on constructing an acyclic container for a group~\cite{BDH}.

\subsection{Chang-Weinberger's topological definition of the \texorpdfstring{$L^2$}{L2} \texorpdfstring{$\rho$}{p}-invariant}

In this section, we briefly recall the $L^2$-signature and the Chang-Weinberger's topological definition of the $L^2$ $\rho$-invariant.

Let $W$ be a compact $2k$-manifold endowed with a homomorphism $\pi_1(W) \to \Gamma$. For a group $\Gamma$, one can obtain the group von Neumann algebra $\mathcal{N}\Gamma$ (See~\cite[Definition 1.1]{L}). Since $\C\Gamma \subset \mathcal{N}\Gamma$, the given homomorphism $\pi_1(W) \to \Gamma$ induces a representation of $\Z \pi_1(W)$ into $\mathcal{N}\Gamma$ via the composition $\Z \pi_1(W) \xhookrightarrow{} \C \pi_1(W) \to \C \Gamma \subset \mathcal{N}\Gamma$. This makes $\mathcal{N}\Gamma$ a $\Z \pi_1(W)$-module. Thus we have an intersection form of the homology of $W$ with local coefficients in $\mathcal{N}\Gamma$
\[
\lambda\colon H_k(W;\mathcal{N}\Gamma) \times H_k(W;\mathcal{N}\Gamma) \to \mathcal{N}\Gamma.
\]
Since any finitely generated submodule of a finitely generated projective module over $\mathcal{N}\Gamma$ is projective~\cite[Theorems~6.7]{L},  we know that $H_k(W;\mathcal{N}\Gamma)$ is a finitely generated $\mathcal{N}\Gamma$-module. By spectral theory for a Hermitian form over a finitely generated $\mathcal{N}\Gamma$-module, there is an orthogonal direct sum decomposition for the intersection form $\lambda$
\[
H_k(W;\mathcal{N}\Gamma)=V_+ \oplus V_- \oplus V_0
\]
such that $\lambda$ is positive definite, negative definite, and zero on $V_+$, $V_-$, and $V_0$ respectively.

We define the $L^2$-signature of $W$ over $\Gamma$ using the von Neumann dimension for $\mathcal{N}G$-modules (See \linebreak \cite[Definition 6.6]{L}).
\[
\dim_{\mathcal{N}\Gamma} \colon \{\mathcal{N}G\text{-modules}\} \to [0,\infty]
\]
For details about $L^2$-dimension theory, we refer readers to excellent references~\cite{L} and~\cite{Shu}.

\begin{definition}
The $L^2$-signature of $W$ over $\Gamma$ is defined by
\[
\sign_{\Gamma}^{(2)}W=\dim_{\mathcal{N}\Gamma}(V_+)-\dim_{\mathcal{N}\Gamma}(V_-) \in \R.
\]
\end{definition}

We recall the topological definition of the $L^2$ $\rho$-invariant for $(4k-1)$-manifolds.

\begin{definition}\label{def:rho}
For a closed oriented topological $(4k-1)$-manifold $M$ and a homomorphism $\varphi\colon \pi_1(M)\to G$, suppose there is a compact oriented $4k$-manifold $W$ with $\partial W = \coprod\limits ^r M$, a group $\Gamma$, a monomorphism $G \xhookrightarrow{} \Gamma$, and a homomorphism $\pi_1(W)\to \Gamma$ which make the following diagram commute:
\[
\adjustbox{scale=1.3,center}{%
\begin{tikzcd}
\pi_1(\coprod\limits^r M) \arrow[r, "\coprod \varphi"] \arrow[d, "i_{\ast}"]
& G \arrow[d, hook, dashed] \\
\pi_1(W) \arrow[r, dashrightarrow, "\overline{\varphi}"]
& \Gamma
\end{tikzcd}
}
\]
Then, the $L^2$ $\rho$-invariant is defined as the signature defect.
\begin{equation}\label{equation:rho inv}
\rho^{(2)}(M,\varphi) := \frac{1}{r}(\sign_{\Gamma}^{(2)}W - \sign W)
\end{equation}
\end{definition}

Note that a well-known result of D. Kan and W. Thurston implies that a diagram as above always exists, and in fact, we can assume $r = 1$ if $M$ is a $3$-manifold. However, we will find this diagram useful for $r > 1$ in later Chapters.

More specifically, recall the Kan-Thurston's theorem~\cite{KT} that any group can be embedded into an acyclic group $\Gamma$. (See Definition~\ref{def:acyclic}.)  The oriented topological bordism groups over $X$, $\Omega_{\ast}^{STOP}(X)$, is a generalized homology theory by Kirby-Siebenmann~\cite{KS} and Freedman-Quinn~\cite{FQ}. Thus, one readily computes that ~$\Omega_{\ast}^{STOP}(\Gamma) \cong \Omega_{\ast}^{STOP}$, for any acyclic group $\Gamma$.  This follows from the Atiyah-Hirzebruch spectral sequence and the acyclicity of $\Gamma$. Furthermore, $\Omega_{\ast}^{STOP} \otimes \Q \cong \Omega_{\ast}^{SO} \otimes \Q$. By the pioneering work of Thom~\cite{Tho}, $\Omega_{4k-1}^{SO} \otimes \Q = 0$. Thus we have the desired $4k$-manifold $W$ over $B\Gamma$ for some $r > 0$.

The diagram above computes the $L^2$ $\rho$-invariant for $\coprod\limits ^r M$, implying formula~\ref{equation:rho inv}.

We show that the definition is well-defined. Suppose there are $W_i$, $r_i$, and $\Gamma _i$ which satisfy the assumption of Definition~\ref{def:rho}, for $i = 1, 2$. 

First, notice we can embed $\Gamma_1$ and $\Gamma_2$ into the amalgamation of $\Gamma_1$ and $\Gamma_2$ over $G$,~$\Gamma_1*_G\Gamma_2$. Let $\Gamma$ be the Kan-Thurston's acyclic container of $\Gamma_1*_G\Gamma_2$. Then, we obtain a commuting diagram of embeddings:

\begin{center}
\begin{tikzcd}
                                    & \Gamma_1 \arrow[rd, hook] &                                     &        \\
G \arrow[rd, hook] \arrow[ru, hook] &                           & \Gamma_1*_G\Gamma_2 \arrow[r, hook] & \Gamma \\
                                    & \Gamma_2 \arrow[ru, hook] &                                     &       
\end{tikzcd}
\end{center}
By the naturality of $L^2$-signatures under inclusion of groups~\cite[Proposition 5.13]{COT}, for $i=1,2$, we obtain:
\[
\sign_{\Gamma_i}^{(2)}W_i=\sign_{\Gamma}^{(2)}W_i.
\]
Thus, we can replace $\Gamma_1$ and $\Gamma_2$ with $\Gamma$.

Define a $4k$-manifold $V \colon= r'W \cup_{rr'M} W$. Then, $V$ is a closed $4k$-manifold over $\Gamma$. By acyclicity of $\Gamma$, $\Omega_{4k}^{STOP}(B\Gamma)=\Omega_{4k}^{STOP}$. Thus, $V$ is bordant over $\Gamma$ to $V'$ endowed with a constant map. Since this $L^2$-signature (with the constant map) equals the classical  signature, we have $\sign_{\Gamma}^{(2)}V'=\sign V'$. Thus, by Novikov additivity, using that signatures are bordism invariants, we conclude:
\begin{align*}
\frac{1}{r}(\sign_{\Gamma}^{(2)}W_1 - \sign W_1)-\frac{1}{r'}(\sign_{\Gamma}^{(2)}W_2 - \sign W_2) &= \frac{1}{rr'}(\sign_{\Gamma}^{(2)}V - \sign V)\\
               &=\frac{1}{rr'}(\sign_{\Gamma}^{(2)}V' - \sign V')\\
               &=0.
\end{align*}
Thus, $\rho^{(2)}(M,\varphi)$ is independent of the choice of $W$, $r$, and $\Gamma$. In other words, $\rho^{(2)}(M,\varphi)$ is well-defined.

\noindent {\bf Note:}  The Chang-Weinberger topological definition of the $L^2$ $\rho$-invariant uses the Kan-Thurston acyclic container.  Other acyclic container functors exist for groups, and any acyclic container serves to define the $\rho$-invariant.  Following  Cha~\cite{Cha16}, we use the BDH-acyclic group~\cite{BDH} which is an acyclic container of a group for the remainder of this paper.

\subsection{The Moore complex of a simplicial classifying space}

In this section, we recall the Moore complex of a simplicial classifying space, a chain complex arising from the bar construction of $G$ which is used to compute group homology. For details, we refer readers to excellent references~\cite{May},~\cite{B}, and~\cite[Appendix]{Cha16}.

\begin{definition}\label{def:simplicial sets}
A simplicial set $X$ is a graded set $\{X_0, X_1, X_2, \cdots \}$ together with functions $d_i\colon X_n \to X_{n-1}$ and $s_i\colon X_n \to X_{n+1}$ where $n \in \N \cup \{0\}$ and $i \in \{0,1,2,\cdots,n\}$ which satisfy the following identities:
\begin{align*}
d_id_j &=d_{j-1}d_i \quad\text{if } i < j,\\
s_is_j &=s_{j+1}s_i \quad\text{ if } i \leq j,\\
d_is_j &=s_{j-1}d_i \quad\text{ if } i < j,\\
d_is_j &=s_jd_{i-1} \quad\text{ if } i > j+1,\\
d_js_j &=\text{identity}=d_{j+1}s_j.
\end{align*}
The elements of $X_n$ are called $n$-simplices. The $d_i$ and $s_i$ are called face functions and degeneracy functions. A simplex $\sigma$ is degenerate if $\sigma=s_i\tau$ for some simplex $\tau$ and degeneracy function $s_i$. Otherwise $\sigma$ is non-degenerate.
\end{definition}

We define the Moore complex $\Z X_{\ast}$ of a simplicial set $X$.

\begin{definition}\label{def:Moore}
The Moore complex $\Z X_{\ast}$ of a simplicial set $X$ is a chain complex of the abelian groups $\Z X_n$ together with the boundary operators $\d\colon\Z X_n \to \Z X_{n-1}$, where $n \in \N \cup \{0\}$. The group $\Z X_n$ is defined to be the free abelian group generated by the $n$-simplices of $X_n$. The boundary operator $\d\colon\Z X_n \to \Z X_{n-1}$ is defined by $\d:=\Sigma_{i=0}^{n}(-1)^{i}d_i$.
\end{definition}

\begin{remark}
Readers are warned that the Moore complex $\Z X_{\ast}$ of a simplicial set $X$ is not the same as the cellular chain complex $C_{\ast}(\lvert X  \rvert)$ of the geometric realization $\lvert X \rvert$ of $X$. 
\end{remark}

However there is a relation between $\Z X_{\ast}$ and $C_{\ast}(\lvert X  \rvert)$. Abusing notation, define $C_{\ast}(X) := C_{\ast}(\lvert X  \rvert)$. Define $D_{\ast}(X)$ by the subgroup of $\Z X_{\ast}$ generated by degenerate simplices of X. 

\begin{theorem}\label{thm:projection from the Moore complex}
\textup{(Mac Lane~\cite[p. 236]{ML}).} 
For a simplicial set $X$, there exists a short exact sequence 
\[
0 \to D_{\ast}(X) \to \Z X_{\ast} \overset{p} \to C_{\ast}(X) \to 0 
\]
where the projection $p$ is a chain homotopy equivalence.
\end{theorem}
Notice
\[
C_{\ast}(X) \cong \frac{\Z X_{\ast}}{D_{\ast}(X)}.
\]

\begin{remark}\label{rmk:inclusion into the Moore complex}
We note that if $X$ is an simplicial complex which is viewed as a simplicial set, then there is an injective chain map $C_{\ast}(X) \to \Z X_{\ast}$. Readers are warned that this does not hold for an arbitrary simplicial set $X$. (See, for instance~\cite[Appendix A.2]{Cha16}.)
\end{remark}

The classifying space $BG$ of a discrete group $G$ can be chosen to have a simplicial structure.  We  give a standard functorial construction of $BG$ in the definition which follows.

\begin{definition}
For a group $G$, the {\em simplicial classifying space $BG$} is a simplicial set with $BG_n = \{ [g_1, \ldots , g_n ] \mid  g_i \in G \}$ where $n \in \N \cup \{0\}$ together with  face functions $d_i : BG_n \to BG_{n-1}$ and  degeneracy functions $s_i : BG_n \to BG_{n+1}$  defined by
\begin{align*}
d_i[g_1,\ldots,g_n]&=
    \begin{cases}
      [g_2,\ldots,g_n] & i=0 \\
      [g_1,\ldots,g_i g_{i+1},\ldots,g_n] & 1 \leq i \leq n-1 \\
      [g_1,\ldots,g_{n-1}] & i=n \\
    \end{cases}\\
s_i[g_1, \ldots, g_n] &= [g_1, \ldots, g_i , e, g_{i+1}, \ldots, g_n]
\end{align*}
where $i = 0, 1, 2, \ldots, n$.
\end{definition}

\begin{remark}\label{rmk:component of simplex}
For an $n$-simplex $\sigma=[g_1,\ldots,g_n]$, we call $g_1, \ldots, g_n$ {\em components of $\sigma$}. Define a projection function $\proj_i\colon BG_n \to G$ ruled by $\proj_i([g_1,\ldots,g_n])=g_i$.
\end{remark}

Readers are warned that common alternative notations are $(g_1,\ldots,g_n)$ or $[g_1 \vert \ldots \vert g_n]$ in place of $[g_1,\ldots,g_n]$.

By applying Definition~\ref{def:Moore} to the simplicial classifying space $BG$, we obtain the Moore complex of the simplicial classifying space $BG$.  As this plays a large role in this paper, we state this clearly below.

\begin{definition}
For a group $G$, the Moore complex $\Z BG_{\ast}$ of the simplicial classifying space $BG$ is a chain complex of free abelian groups $\Z BG_n$ which is generated by $n$-tuples $[g_1, \ldots , g_n ]$ of group elements $g_1, \ldots , g_n \in G$, together with the boundary operator which is defined as the alternating sum of face functions $\d = \Sigma_{i=0}^{n}(-1)^i d_i$, where
\begin{center}
$
d_i[g_1,\ldots,g_n]=
    \begin{cases}
      [g_2,\ldots,g_n] & i=0 \\
      [g_1,\ldots,g_i g_{i+1},\ldots,g_n] & 1 \leq i \leq n-1 \\
      [g_1,\ldots,g_{n-1}] & i=n.
    \end{cases}
$
\end{center}
\end{definition}

The Moore space provides a chain complex for computing group homology, as stated in the theorem below.

\begin{theorem}
For a group $G$,
\[
H_n(G)=H_n(\Z BG_{\ast})
\]
for all $n \in \N \cup \{0\}$.
\end{theorem}

For details and related discussions we refer readers to~\cite[p. 35-41]{B}.

We briefly recall the product of simplicial classifying spaces of groups.

\begin{definition}
For two groups $G$ and $H$, the product of simplicial classifying spaces $BG$ and $BH$ is a simplicial set $BG \times BH$ defined by $(BG \times BH)_n := BG_n \times BH_n$, together with the face and degeneracy functions given by $d_i(\sigma \times \tau)=d_i\sigma \times d_i\tau$ ~and ~$s_i(\sigma \times \tau)=s_i\sigma \times s_i\tau$.
\end{definition}

A benefit of simplicial sets which allow degenerate simplices is that the construction for a product is easy. We obtain the Moore complex of the product of simplicial classifying spaces by applying Definition~\ref{def:Moore}.

\begin{definition}
For a product of simplicial classifying spaces $BG$ and $BH$, the Moore complex $\Z(BG \times BH)_{\ast}$ consists of ~$\Z(BG \times BH)_{n}$ where $n \in \N \cup \{0\}$ which are free abelian groups generated by $[g_1,\ldots,g_n] \times [h_1,\ldots,h_n]$ ~for $g_i \in G$ and $h_i \in H$. The boundary operators are given by the alternating sum of face functions $\d = \Sigma_{i=0}^{n}(-1)^i d_i$.
\end{definition}

\subsection{Controlled chain homotopy}

In this section, we recall basic definitions of {\em controlled chain homotopy} introduced by Cha~\cite{Cha16}.

\begin{definition}
For a positive based chain complex over $\Z$, the {\em diameter} $d(u)$ of a chain $u=\Sigma_{\alpha} n_\alpha e_\alpha$ is defined by
\[
d(u):=\Sigma_{\alpha} \vert n_\alpha \vert
\]
which is the $L^1$-norm.
\end{definition}

\begin{definition}
Suppose $C_\ast$ and $D_\ast$ are based chain complexes. For a chain map $f \colon C_\ast \to D_\ast$, the {\em diameter function $d_f$ of the chain map} is defined by
\begin{center}
$d_f(k) := max \{d(f(c)) \mid c \in C_i$ is a basis element, $i \leq k \}$.
\end{center}
\end{definition}

\begin{definition}
Suppose $C_\ast$ and $D_\ast$ are based chain complexes. For a chain homotopy $P \colon C_\ast \to D_{\ast +1}$, the {\em diameter function $d_P$ of the chain homotopy} is defined by
\begin{center}
$d_P(k) := max \{d(P(c)) \mid c \in C_i$ is a basis element, $i \leq k \}$.
\end{center}
\end{definition}

\begin{remark}
In general, the diameter function can be infinity. If a chain map or a chain homotopy is defined on a finitely generated chain complex, then its diameter function is finite.
\end{remark}

\begin{remark}
For a partial chain homotopy $P$ of dimension $n$ which is defined on $C_i$ for $i \leq n$ only, the diameter function $d_P(k)$ is defined for $k \leq n$ only.
\end{remark}

\begin{definition}
We say a function $\delta \colon \N \cup \{0\} \to \N \cup \{0\}$ {\em controls} a chain homotopy $P$ if 
\[
d_P(k) \leq \delta(k).
\]
\end{definition}

\begin{definition}\label{def:uniformly controlled}
For a collection of chain homotopies $S=\{P_A : C_{\ast}^{A} \to D_{\ast+1}^{A} \}_{A\in I}$, we say that $S$ is {\em uniformly controlled by $\delta$} if each $P_A$ is controlled by $\delta$.
\end{definition}

\subsection{Simplicial-cellular complexes}

We recall the simplicial-cellular complex and important properties, following the approach introduced by Cha~\cite[p.1173-1174]{Cha16}.

\begin{definition}\label{def:simplicial-cellular}
A CW complex $X$ is called {\em pre-simplicial-cellular} if each $n$-cell is endowed with a characteristic map of the standard $n$-simplex $\Delta^n$ to $X$. 

For pre-simplicial-cellular complexes $X$ and $Y$, a cellular map $X \to Y$ is {\em simplicial-cellular} if its restriction on an open $k$-simplex of $X$ is surjective onto an open $m$-simplex of $Y$ $(m\leq k)$ which extends to a linear surjection $\Delta^k \to \Delta^m$, sending vertices to vertices.

A pre-simplicial-cellular complex $X$ is {\em simplicial-cellular} if the attaching map $\d\Delta^k \to X^{(k-1)}$ of every $k$-cell is simplicial-cellular.
\end{definition}

\begin{remark}\label{rmk:simplicial-cellular}
We remark important properties of simplicial-cellular complexes. Notice an open $n$-cell is identified with the interior of $\Delta^n$. A simplicial complex is a pre-simplicial-cellular in a canonical way. Note that the composition of simplicial-cellular maps is simplicial-cellular. A simplicial complex is a simplicial-cellular complex. A simplicial map is simplicial-cellular. A triangulation is a simplicial-cellular complex. It is known that the geometric realization of a simplicial set is a simplicial-cellular complex~\cite{Mil57}. It is well known that, for a discrete group $G$, the geometric realization of the simplicial classifying space $BG$ is a $K(G, 1)$ space.  (See, for instance~\cite[p.6]{GJ99}.)
\end{remark}

We recall the simplicial-cellular approximation of maps to $BG$.

\begin{theorem}\label{thm:simplicial approximation}
\textup{(Cha~\cite[Theorem 3.7]{Cha16}).}
For a geometric realization of a simplicial set $X$, any map $X \to BG$ is homotopic to a simplicial-cellular map.
\end{theorem}

For the proof we refer readers to~\cite[Proposition A.1]{Cha16}.

\subsection{Mitosis, the BDH-acyclic group, and the Baumslag-Dyer-Heller functor}

We briefly review the work of Gilbert Baumslag, Eldon Dyer, and Alex Heller on constructing an acyclic container for a group~\cite{BDH}.  In particular, all results discussed in this section are due to the above authors.

\begin{definition}\label{def:mitosis}

The mitosis embeddings are defined by using below sequence of groups and injective homomorphisms:
\begin{center}
$G \overset{k_G} \hookrightarrow \A(G) \overset{k_{\A(G)}} \hookrightarrow \A^2(G) \overset{k_{\A^2(G)}} \hookrightarrow \A^3(G)\hookrightarrow \cdots$ 
\end{center}
where $\A^0(G) = G$ and
\begin{center}
$\A^{n+1}(G):=\langle \A^n(G), u_{n+1}, t_{n+1} \mid a^{t_{n+1}}=a \cdot a^{u_{n+1}}\text{, } [a^{u_{n+1}},b] \quad \text{for every} \quad a, b \in \A^n(G) \rangle$.
\end{center}

We call $\A^n(G)$ the {\em ($n$-th) mitosis of $G$} and define $\A(G)$ to be $\A^1(G)$.

Since $k_G, k_{\A(G)}, k_{\A^2(G)}, \cdots$ are injective, so are their compositions. We denote by 
\[
i_G^n\colon G \to \A^n(G)
\]
the composition $k_{\A^n(G)} \circ \cdots \circ k_{\A(G)} \circ k_G$ and by \[
k_{ij}\colon\A^i(G) \to \A^j(G)
\]
where $i \leq j$ the composition $k_{\A^j(G)} \circ k_{\A^{j-1}(G)} \circ \cdots \circ k_{\A^i(G)}$.

For a $n \in \N$, we call $i_G^n$ the {\em ($n$-th) mitosis embedding of $G$}.
\end{definition}
\noindent {\bf Note:} Throughout this paper, $a^b$ is the conjugation $b\cdot a \cdot \overline{b}$.

\begin{theorem}
The function $\A \colon \textbf{Gp} \to \textbf{Gp}$ is a functor on the category $\textbf{Gp}$ of groups with a natural transformation $k \colon id_{\textbf{Gp}} \to \A$ such that $k_G \colon G \to \A(G)$ is injective for any group $G$.
\end{theorem}
\begin{proof}
For a homomorphism $f \colon G \to H$, define the homomorphism $\A(f) : \A(G) \to \A(H)$ determined by
\begin{center}
$
\A(f)(a):=
    \begin{cases}
      f(a) & \text{if   } a \in G \subset \A (G)\\
      u \in \A(H) & \text{if   } a=u \in \A (G) \\
      t \in \A(H) & \text{if   } a=t \in \A (G) .
    \end{cases}
$
\end{center}
Then, $k_{\A(H)}\circ f = \A(f)\circ k_{\A(H)}$.
\end{proof}

\begin{remark}
Similarly one can check that, for any $n$, $\A^n\colon \textbf{Gp} \to \textbf{Gp}$ is a functor with a natural transformation $i^n \colon id_{\textbf{Gp}} \to \A^n$.
\end{remark}

We recall one of the most significant properties of mitosis embeddings.

\begin{theorem}\label{thm:mitosis}
\textup{(Baumslag-Dyer-Heller~\cite[Proposition 4.1]{BDH}).}
Let $\k$ be a field. Suppose $f\colon A \to B$ be a homomorphism of groups such that $f_{\ast}\colon H_i(A;\k) \to H_i(B;\k)$ is a zero homomorphism for $i=1,2,\cdots,n-1$. Then $i_B^1 \circ f \colon A \to \A(B)$ induces a zero homomorphism $(i_B^1 \circ f)_{\ast}\colon H_{n}(A;\k) \to H_n(\A(B);\k)$.
\end{theorem}

\begin{corollary}\label{cor:mitosis}
For any $n \in \N$, $(i_G^n)_{\ast}\colon H_{i}(G;\k) \to H_i(\A^n(G);\k)$ is zero for all $i \in \{1,2,\cdots,n\}$.
\end{corollary}

We recall the definition of acyclicity.

\begin{definition}\label{def:acyclic}
A group $G$ is called {\em acyclic} if $H_i(G)=0$ for any $i \in \N$.
\end{definition}

Corollary~\ref{cor:mitosis} follows from Theorem~\ref{thm:mitosis} and Definition~\ref{def:mitosis}. Corollary~\ref{cor:mitosis} plays a key role to prove the following theorem.

\begin{theorem}
\textup{(Baumslag-Dyer-Heller~\cite[Chapter 5]{BDH}).}
For any group $G$, the colimit
\[
\colim \A^i(G)
\]
of the direct system $\langle \A^{i}(G),k_{ij}\rangle$ is acyclic. We call the colimit the BDH-acyclic group of $G$ or the acyclic container of $G$. We denote by $\mathcal{A}(G)$ the acyclic container of $G$.
\end{theorem}

We end this section, recalling the functoriality of the acyclic container.

\begin{theorem}\label{thm:BDH}
\textup{(Baumslag-Dyer-Heller~\cite[Theorem 5.5]{BDH}).}
There exists a functor $\mathcal{A} \colon \textbf{Gp} \to \textbf{Gp}$ on the category $\textbf{Gp}$ of groups with a natural transformation $\iota \colon id_{\textbf{Gp}} \to \mathcal{A}$ such that $\mathcal{A}(G)$ is acyclic and $\iota_{G} \colon G \to \mathcal{A}(G)$ is injective for any group $G$.
\end{theorem}

\section{An outline of Cha's proof of Theorem~\ref{thm:Cha}}

In this chapter, we outline the proof of Cha's Theorem~\ref{thm:Cha} which gives universal upper bounds on the Cheeger-Gromov $\rho$-invariants of $M$. We conclude by briefly discussing how these bounds might be improved.

\subsection{Existence of universal bounds of \texorpdfstring{$L^2$}{L2} \texorpdfstring{$\rho$}{p}-invariants}
In this section, we briefly sketch Cha's topological proof of existence of universal bounds. For brevity, we focus on $3$-manifolds, that is, $(4k-1)$-manifolds where $k=1$. For a general proof we refer the reader to \cite{Cha16}.

We construct a $4$-manifold $W$ which satisfies the hypothesis of Definition~\ref{def:rho}, being independent of the given homomorphism $\varphi \colon  \pi_1(M) \to G$. This independence implies that the resulting bounds obtained hold for all $L^2$ $\rho$-invariants independent of the choice of homomorphism $\varphi$. To do so we use Theorem~\ref{thm:BDH} of Baumslag, Dyer, and Heller.

We prove the existence of the universal bounds.

\begin{theorem}\label{thm:ChaExist}
\textup{(Cha~\cite[Theorem 1.3]{Cha16}).} 
For any closed oriented topological $3$-manifold M, there is a constant $C_M$ such that $\lvert \rho^{(2)}(M, \varphi)\rvert \leq C_M$ for any homomorphism $\varphi \colon  \pi_1(M) \to G$ to any group $G$.
\end{theorem}

\begin{proof}
Assume a $3$-manifold $M$ is given. Using acyclicity and functoriality, for any homomorphism $\varphi \colon  \pi_1(M) \to G$ we then have a commutative diagram as follows.
\[
\adjustbox{scale=1.3,center}{%
\begin{tikzcd}
\pi_1(M) \arrow[r, "\varphi"] \arrow[d, "i_{\ast}"] \arrow[rd, hook, "\iota_{\pi_1(M)}"] & G \arrow[rd, hook, "\iota_{G}"] \\
\pi_1(W) \arrow[r] & \mathcal{A}(\pi_1(M)) \arrow[r, "\mathcal{A}(\varphi)"] & \mathcal{A}(G) = \Gamma
\end{tikzcd}
}
\]
Since $\Omega_3^{STOP}(\mathcal{A}(\pi_1(M))) \cong \Omega_3^{STOP} \cong \{0\}$,  there is a $4$-manifold $W$ which makes the left triangle commute. Baumslag-Dyer-Heller's functor together with the injective natural transformation allows us to construct the right parallelogram such that $\mathcal{A}(\pi_1(M))$ is acyclic. 

Notice that $W$ is independent of the given homomorphism $\varphi$. $W$ depends on only $M$ and the acyclic functor, $\mathcal{A}$. By the topological definition of the $L^2$ $\rho$-invariant, for $N$, the number of $2$-handles in a handle decomposition of $W$,
\[
\lvert \sign_{\Gamma}^{(2)}W \rvert \leq \textup{dim}_{\Gamma}^{(2)} H_{2}(W;\N \Gamma) \leq \textup{dim}_{\Gamma}^{(2)} C_{2}(W;\N \Gamma) \leq N.
\]
Similarly, $\lvert \sign W \rvert \leq N$. Thus $\lvert \rho^{(2)}(M,\varphi) \rvert \leq 2N$ for any homomorphism $\varphi\colon \pi_1(M)\to G$ to any group $G$.
\end{proof}

\subsection{Cha's universal bounds for general \texorpdfstring{$3$}{3}-manifolds}

In this section, we give a brief outline of Cha's proof of Theorem~\ref{thm:Cha}.

As seen in Definition~\ref{def:rho}, Chang and Weinberger defined their topological definition of $\rho$-invariants by using the idea of embedding a group into an acyclic group. Using the Chang-Weinberger approach, Cha proved the existence of universal bounds for all rho-invariants of any topological (4k-1)-manifold in Theorem~\ref{thm:ChaExist}. In the proof, we observed the important two facts. The first is that the injective natural transformation given by~\cite{BDH} is needed to construct $W$ which is independent to the given homomorphism $\varphi$. The second is that $\lvert \rho^{(2)}(M,\varphi) \rvert$ is bounded by $2N$ where $N$ is the number of $2$-handles in a handle decomposition of $W$. 

Cha first finds a $4$-chain $u$ in the chain complex of this acyclic group with boundary representing the image of the fundamental class of the $3$-manifold $M$. Using this $4$-chain he constructs a null-bordism $W$ of $M$ over the BDH-acyclic container and then counts the $2$-handle complexity of $W$.

As discussed in Theorem~\ref{thm:ChaExist}, Cha's null-bordism $W$ of $M$ over the BDH-acyclic container has the $2$-handle complexity which depends on the complexity of $M$ and the $4$-chain $u$.

\begin{theorem}\label{thm:975} 
\textup{(Cha~\cite[Theorem 3.9]{Cha16}).} Suppose $M$ is a closed triangulated $3$-manifold with complexity $d(\zeta_M)$ where $\zeta_M$ is the fundamental class of $M \in C_{\ast}(M)$ associated the given triangulation. Suppose $M$ is over a simplicial-cellular complex $K$ via a simplicial-cellular map $\varphi\colon M\to K$. If there is a $4$-chain $u$ $\in$ $C_4(K)$ satisfying $\partial u = \varphi_{\ast}(\zeta_M)$, then there exists a smooth bordism $W$ over $K$ between $M$ and a trivial end whose $2$-handle complexity is at most 195 $\cdot$ $d(\zeta_M)$+975 $\cdot$ $d(u)$.
\end{theorem}

\begin{remark}\label{rmk:trivial end}
A trivial end is a $3$-manifold over $K$ with a constant map. By the Lickorish–Wallace Theorem, any closed $3$-manifold is null-bordant. Notice that we can obtain a null-bordism of a trivial end over $K$ via a constant map. Since the 
$L^2$-signature via a constant map is just the classical signature of $M$, the $L^2$ $\rho$-invariant of the trivial end is zero. By Novikov additivity, signatures are additive under connected sum. Thus, a trivial end has no effect on the $L^2$ $\rho$-invariant $\rho(M,\varphi)$.
\end{remark}

To determine the complexity of the $4$-chain in terms of the complexity of $M$, Cha constructs controlled chain homotopies.

\begin{theorem}\label{thm:ChaPartial} 
\textup{(Cha~\cite[Theorem 5.2]{Cha16}).} For each $n$, there is a family 
\begin{center}
    $\{\Phi^n_G : e \backsimeq i^n_G \mid G$ is a group$\}$ 
\end{center}
of partial simplicial chain homotopies $\Phi^n_G$ of dimension $n$, between the chain maps induced by the mitosis embedding $i_G^3 \colon  G \to \A^3(G)$ and the trivial homomorphism $e \colon  G \to \A^3(G)$.
\[
i^n_G, e \colon \Z BG_\ast \to \Z B\A^n(G)_\ast
\]
(We abuse notation and denote by $e$ and $i_G^n$ the pushforward $e_\ast$ and ${i_G^n}_\ast$ respectively.)
\\
These partial simplicial chain homotopies are uniformly controlled by a function $\delta_{BDH}$. (Recall Definition~\ref{def:uniformly controlled}.) For $k \leq 4$, the value of $\delta_{BDH} (k)$ is as follows:
\begin{center}
 \begin{tabular}{c c c c c c} 
 \hline
 k & 0 & 1 & 2 & 3 & 4 \\ [0.5ex]
 $\delta_{BDH} (k)$ & 0 & 6 & 26 & 186 & 3410 \\ 
 \hline
 \end{tabular}
\end{center}
\end{theorem}

By using Theorem~\ref{thm:975} and Theorem~\ref{thm:ChaPartial}, we prove Theorem~\ref{thm:Cha}.

\begin{proof}[Proof of Theorem~\ref{thm:Cha}]
We denote by $\pi$ the fundamental group of $M$. For brevity, we do not distinguish the $3$-manifold $M$ endowed with a given triangulation and the geometric realization $\lvert M \rvert$ of the simplicial set induced by the triangulation. Notice $M$ is a simplicial-cellular complex by Remark~\ref{rmk:simplicial-cellular}. Furthermore, abusing notation, we denote by $B\A^3(\pi)$ the geometric realization of the simplicial classifying space $B\A^3(\pi)$. As mentioned in Remark~\ref{rmk:simplicial-cellular}, the geometric realization $B\A^3(\pi)$ is a simplicial-complex~\cite{Mil57}.

By Theorem~\ref{thm:simplicial approximation} there is a simplicial-cellular map $j\colon M \to B\pi$ induced by the identity homomorphism $\pi_1(M) \to \pi_1(B\pi)$. Recall $i_\pi^3 \colon  \pi \to \A^3(\pi)$ is the mitosis embedding. Again, by Theorem~\ref{thm:simplicial approximation}, abusing notation, there is a simplicial-cellular map $i_\pi^3 \colon  B\pi \to B\A^3(\pi)$. Define $\varphi:= i_\pi^3 \circ j$. Then, $\varphi \colon M \to B\A^3(\pi)$ is a simplicial-cellular map. We show $\varphi$ satisfies the hypothesis of Theorem~\ref{thm:975}.
 
We now discuss the composition of chain maps given below and the associated notation.
\[
C_\ast(M) \overset{i} \to \Z_\ast(M) \overset{j} \to \Z B\pi_{\ast} \overset{i^3_G} \to \Z B\A^3(\pi)_\ast \overset{p} \to C_\ast(B\A^3(\pi))
\]

Since a triangulation can be regarded as a cellular complex, we obtain the cellular chain complex $C_\ast(M)$ induced by the given triangulation of $M$. Notice that a triangulation is associated to a simplicial complex for $M$. Since a simplicial complex induces a simplicial set, we obtain the Moore complex $\Z_\ast(M)$ induced by the simplicial structure of the given triangulation of $M$. By Remark~\ref{rmk:inclusion into the Moore complex}, there is the inclusion chain map $i\colon C_\ast(M) \to \Z_\ast(M)$.

Abusing notation, we denote by $j\colon \Z_\ast(M) \to \Z B\pi_\ast$ the pushforward induced by $j\colon M \to B\pi$. 

As seen in Definition~\ref{def:mitosis}, there is a mitosis embedding $i_G^3 \colon  \pi \to \A^3(\pi)$ which is an injective homomorphism of groups. For brevity, denote by $i_G^3 \colon  \Z B\pi_{\ast} \to \Z B\A^3(\pi)_\ast$ the chain map induced by the monomorphism.

As seen in Theorem~\ref{thm:projection from the Moore complex}, there is a projection 
\[
p\colon \Z B\A^3(\pi)_{\ast} \to C_{\ast}(B\A^3(\pi)) \cong \Z B\A^3(\pi)_{\ast}/D_{\ast}(B\A^3(\pi)).
\]

Since $i$ is an inclusion and $p$ is a projection, for the simplicial-cellular map $\varphi\colon M \to B\A^3(\pi)$,
\[
\varphi_{\ast} = p \circ i_G^3 \circ j \circ i \colon  C_\ast(M) \to C_\ast(B\A^3(\pi)).
\]

Using Theorem~\ref{thm:ChaPartial}, we find a desired $4$-chain $u$. Define $u:=p(\Phi^n_G(j(i(\zeta_M))))$ where 
\[
C_\ast(M) \overset{i} \to \Z_\ast(M) \overset{j} \to \Z B\pi_{\ast} \overset{\Phi^3_G} \to \Z B\A^3(\pi)_{\ast + 1} \overset{p} \to C_{\ast + 1}(B\A^3(\pi)).
\]
Notice $\zeta_M \in C_3(M)$ is a boundary. Since $\Phi^3_G$ is a partial simplicial chain homotopy of dimension $3$ between chain maps induced by the mitosis embedding $i_G^3$ and the trivial homomorphism $e$, one can check $\d p(\Phi^n_G(j(i(\zeta_M))))=\varphi_{\ast}(\zeta_M)$.

Since $p$ is a projection, $d(p(\sigma)) \leq 1$ for any simplex $\sigma$. Notice $d(j(\sigma))=1$ for any simplex $\sigma$ because $j$ is induced by a simplicial-cellular map. Since $i$ is an inclusion, $d(i(\sigma)) = 1$ for any simplex $\sigma$. Thus, we obtain 
\[
d(u)=d(p(\Phi^n_G(j(i(\zeta_M))))) \leq d_p(4) \cdot d_{\Phi^n_G}(3) \cdot d_j(3) \cdot d_i(3) \cdot n \leq 1 \cdot 186 \cdot 1 \cdot 1 \cdot n = 186 \cdot n
\]
where $n$ is the simplicial complexity $d(\zeta_M)$ of $M$.

In other words, $d(u) \leq 186 \cdot d(\zeta_M)$. By combining Theorem~\ref{thm:ChaExist}, Theorem~\ref{thm:975}, and Remark~\ref{rmk:trivial end}, Cha concludes 
\[
\lvert \rho^{(2)}(M,\varphi) \rvert \leq 
2N \leq
2 \cdot (195 \cdot n+975 \cdot d(u)) \leq
2 \cdot (195 \cdot n+975 \cdot (186 \cdot n)) \leq
363090 \cdot n.
\]
\end{proof}

In an essential part of the proof of Theorem~\ref{thm:general}, we find a chain null-homotopy with a smaller upper bound for $d(u)$ than Cha's $186 \cdot d(\zeta_M)$.  We then follow the proof of Cha's theorem~\ref{thm:Cha}. In the next chapter we construct a new and better chain null-homotopy.

\section{Proof of Theorem~\ref{thm:general}}\label{chapter:general}

In this chapter we prove our main Theorem~\ref{thm:general}.

We prove Theorem~\ref{thm:general} in Section~\ref{section:chain homotopy} after first stating three fundamental theorems that we will use in our proof of Theorem~\ref{thm:general}.  

We warn the reader that proving the `fundamental theorems' mentioned above will require extensive computation (without computer aid). To help the reader weave their way through the necessary computations, we will provide the machinery to make clean inductive arguments and supply models which we hope will communicate the underlying geometric foundations for these computations.

The remainder of Chapter~\ref{chapter:general} is then devoted to proving the following three theorems.

\noindent {\bf Note:} Throughout the rest of this paper, we denote by $e$ the trivial group homomorphism. Recall $i_G^n$ is a mitosis embedding defined in Definition~\ref{def:mitosis}. Abusing notation, we denote by $f$ the chain maps induced by a group homomorphism $f$.

\subsection{Constructing chain homotopies}\label{section:chain homotopy}

\begin{theorem}\label{thm:family}
For each $n$, there is a family 

\begin{center}
    $\{\Psi^n_G : e \backsimeq i^n_G \mid G$ is a group$\}$
\end{center}
of partial simplicial chain homotopies $\Psi^n_G$ of dimension $n$, between the chain maps 
\[
i^n_G, e \colon  \Z BG_\ast \to \Z B\A^n(G)_{\ast},\]
which is uniformly controlled (recall Definition~\ref{def:uniformly controlled}) by a function $\gamma$. (As mentioned in the above note, we denote by $e$ the trivial group homomorphism and $i_G^n$ is a mitosis embedding defined in Definition~\ref{def:mitosis}. For simplicity, we denote by $e$ and $i_G^n$ the chain maps induced by $e$ and $i_G^n$, respectively.) For $m \leq 7$, the value of $\gamma (m)$ follows.

\begin{center}
 \begin{tabular}{c c c c c c c c c} 
 \hline
m & 0 & 1 & 2 & 3 & 4 & 5 & 6 & 7\\ [0.5ex]
 $\gamma (m)$ & 0 & 4 & 24 & 152 & 1120 & 9732 & 98336 & 1135024\\ 
 \hline
 \end{tabular}
\end{center}

Furthermore, for $m \geq 1$, there is a recurrence formula of $\gamma$ as follows:
\[
\gamma (m)=2^{m} \cdot (m+1) + \Sigma_{k=1}^{m-1} \gamma (k) \cdot {\binom{m+1}{m-k}}.
\]
\end{theorem}
Next, we will construct a specific simplicial chain homotopy in the family, $\{\Psi^n_G\}$ in Theorem~\ref{thm:family}, counting the degenerate simplices in the simplicial chain homotopy.

\begin{theorem}\label{thm:new}
For any group $G$ and each $n$, there is a partial simplicial chain homotopy $\Psi_G^n$ between the chain maps $i^n_G, e \colon  \Z BG_\ast \to \Z B\A^n(G)$, whose diameter function is exactly $\gamma$ in Theorem~\ref{thm:family} and, for an $m$-simplex $\sigma$, $\Psi_G^n(\sigma)$ has at least $q(m)$ degenerate $(m+1)$-simplices. For $m \leq 7$, the value of $q(m)$ follows.

\begin{center}
 \begin{tabular}{c c c c c c c c c} 
 \hline
 m & 0 & 1 & 2 & 3 & 4 & 5 & 6 & 7\\ [0.5ex]
 $q(m)$ & 0 & 1 & 8 & 55 & 414 & 3613 & 36532 & 421699\\ 
 \hline
 \end{tabular}
\end{center}

Actually, for $m \geq 1$, there is the recurrence formula of $q(m)$:
\[
q(m)=2^{m} \cdot (m-1) + 1 + \Sigma_{k=1}^{m-1} q(k) \cdot \binom{m+1}{m-k}.
\]
\end{theorem}

As we discussed in Theorem~\ref{thm:projection from the Moore complex}, $C_{\ast}(X) \cong \Z X_{\ast} / D_{\ast}(X)$, where $D_{\ast}(X)$ is the subgroup of $\Z X_{\ast}$ generated by degenerate simplices of $X$. There is the projection $p \colon  \Z X_{\ast} \to C_{\ast}(X)$. Notice the projection is a chain map. Thus we obtain Theorem~\ref{thm:kill} by subtracting $q(m)$ from  $\gamma (m)$.

\begin{theorem}\label{thm:kill}
For any group $G$ and each $n$, there is a partial chain homotopy $\Phi_G^n$ of dimension $n$ between the chain maps 
\[
i^n_G, e \colon  \Z BG_\ast \to C_{\ast}( B\A^n(G) ),
\]
which is controlled by $c(m) := \gamma (m) - q(m)$. For $m \leq 7$, the value of $c(m)$ is as follows.

\begin{center}
 \begin{tabular}{c c c c c c c c c} 
 \hline
 m & 0 & 1 & 2 & 3 & 4 & 5 & 6 & 7\\ [0.5ex]
 $c(m)$ & 0 & 3 & 16 & 97 & 706 & 6119 & 61804 & 713325\\ 
 \hline
 \end{tabular}
\end{center}

For $m \geq 1$, there is the recurrence formula of $c(m)$ as follows:
\[
c(m)=2 \cdot 2^{m} - 1 + \Sigma_{k=1}^{m-1}  c(k) \cdot \binom{m+1}{m-k}.
\]
\end{theorem}

\begin{proof}[Proof of Theorem~\ref{thm:general}]
For a $3$-manifold $M$, define $\varphi := i_\pi^3 \circ j$ and $u := \Phi_{\pi}^3(j(i(\zeta_M)))$ where $\pi=\pi_1(M)$, $i:C_\ast(M) \to \Z_\ast(M)$ is the inclusion, $j:\Z_\ast(M) \to \Z B\pi_\ast$ is induced by the identity group homomorphism $\pi_1(M) \to \pi_1(B\pi)$, and $\zeta_M \in C_\ast(M)$ is the fundamental class. Proceeding as in the proof of Theorem~\ref{thm:Cha} with observe the universal bound for $\lvert \rho^{(2)}(M, \varphi)\rvert$ is $2$ times the $2$-handle complexity $N$ of a desired bordism in Theorem~\ref{thm:975}. Notice, using Theorem~\ref{thm:kill}, we see that $d(u) \leq 97 \cdot d(\zeta_M)$. In particular,  we obtain a stronger universal bound
\[
\lvert \rho^{(2)}(M,\varphi) \rvert \leq 
2N \leq
2 \cdot (195 \cdot n+975 \cdot d(u)) \leq
2 \cdot (195 \cdot n+975 \cdot (97 \cdot n)) =
189540\cdot n
\]
where $n$ is the simplicial complexity of $M$, $d(\zeta_M)$.
\end{proof}

\begin{remark}\label{rmk:growth}

Explicit recurrence formulas for  $\gamma(m)$, $q(m)$, and $c(m)$ in the above three fundamental theorems provide an interesting link with enumerative combinatorics and number theory. From the recurrence formulas, we can obtain general terms like below:
\begin{align*}
\gamma(m) &= \Sigma_{k=0}^{m-1} 2^{m-k} \cdot (m-k+1) \cdot \binom{m+1}{k} \cdot B_{k}\\
q(m) &= \Sigma_{k=0}^{m-1} (2^{m-k} \cdot (m-k-1) + 1) \cdot \binom{m+1}{k} \cdot B_{k}\\
c(m) &= \Sigma_{k=0}^{m-1} (2\cdot2^{m-k} - 1) \cdot \binom{m+1}{k} \cdot B_{k}
\end{align*}
where $B_k$ is the $k$th ordered Bell numbers (or Fubini numbers) which are given by
\[
F_k = \begin{cases}
      1 & k=0 \\
      \displaystyle{\frac{Li_{-k}(\frac{1}{2})}{2}} & k \geq 1
    \end{cases}\\
\]
where $\displaystyle{Li_{s}(z):=\Sigma_{i=1}^{\infty}\frac{z^i}{i^S}}$ is the polylogarithm function. 

It is well known that an approximation of the ordered Bell numbers is given by $\displaystyle{B_{k} \approx \frac{k!}{2(log2)^{k+1}}}$. (See, for instance~\cite{Skl} and ~\cite{Bar}.) Thus, we may study an estimate of the asymptotic growth rate of each diameter function. Using a rough estimate, we can show $\gamma (m)$, $q(m)$, and $\displaystyle{c(m) = O((\frac{2}{e})^m \cdot m^{m+\frac{3}{2}})}$. One can numerically check $\displaystyle{\lim_{m\to\infty} \frac{q(m)}{\gamma(m)}=0.3715...}$. In other words, degenerate simplices asymptotically account for about 37.15$\%$ of the whole chain null-homotopy  $\Psi_G^n$.
\end{remark}


\subsection{Edgewise subdivision}

In many literature, \emph{edgewise subdivision} is used as a terminology which contains geometric idea decomposing an $n$-simplex into smaller $n$-simplices while each edge is subdivided into an equal number of equilateral edges. Since this construction leaves overall shape of the given $n$-simplex, there can be room that one can explore with smaller $n$-simplices. In this section, we define the edgewise subdivision as a chain map of Moore complexes by associating two homomorphisms of groups. We will investigate its properties and see how this concept is geometrically interpreted. We use the terminology of \emph{edgewise subdivision} in the narrow sense like ~\cite{Dup} rather than Segal's edgewise subdivision functor~\cite{Seg} in the category theoretic sense. Readers can find further information about \emph{edgewise subdivision} in various senses, for example, ~\cite{EG}, ~\cite{BOORS}. ,~\cite{Vel}, ~\cite{Wal}, ~\cite{Rie}, and ~\cite{DGM}.


\begin{definition}\label{def:edgewise}
Let $G$ and $H$ be groups. For two group homomorphisms $f, g \colon  G \to H$, the {\em edgewise subdivision $Ed_{(f,g)}$ with respect to $f$ and $g$} is the chain map defined as the composite of chain maps:
\begin{center}
$Ed_{(f,g)} : 
\Z BG_\ast \overset{D_\ast} \to 
\Z (BG \times BG)_\ast \overset{(g,f)_\ast} \to 
\Z (BH \times BH)_\ast \overset{\bigtriangleup} \to 
\Z BH_\ast \otimes \Z BH_\ast \overset{\bigtriangledown} \to \Z (BH \times BH)_\ast \overset{T_\ast} \to 
\Z BH_\ast$
\end{center}
In the above,  $D_\ast$, $(g,f)_\ast$, and $T_\ast$ are chain maps induced by the following homomorphisms of groups, respectively:

\begin{center}
$D \colon  G \to G \times G$ is defined by $D(g_1)=(g_1,g_1)$,
\\
$(g,f) \colon  G \times G \to H \times H$ is defined by $(g,f)((g_1, g_2))=(g(g_1), f(g_2))$,
\\
$T \colon H \times H \to H$ is defined by $T((h_1, h_2))=h_1 \cdot h_2.$
\end{center}
\[
\bigtriangleup \colon  \Z (BH \times BH)_\ast \to 
\Z BH_\ast \otimes \Z BH_\ast
\]
is the Alexander-Whitney homomorphism and 
\[
\bigtriangledown \colon  \Z BH_\ast \otimes \Z BH_\ast \to \Z (BH \times BH)_\ast
\]
is the Eilenberg-Zilber shuffle homomorphism given by the formulas below:
\[
\bigtriangleup(\sigma \times \tau)=\Sigma_{i=0}^{n}d_{i+1} \ldots d_n \sigma \otimes (d_0)^i \tau \mbox{\quad for \quad} \sigma \times \tau \in BH_n \times BH_n.
\] 
\[
\bigtriangledown(\sigma \otimes \tau)=\Sigma_{(\mu,\nu) \in S_{p,g}} {\epsilon(\mu,\nu)}(s_{\nu_q} \ldots s_{\nu_1} \sigma) \times (s_{\mu_p} \ldots s_{\mu_1} \tau) \mbox{\quad for \quad} \sigma \otimes \tau \in BH_p \otimes BH_q.
\]
Here, $d_i$ and $s_i$ are the face function and degeneracy function in the simplicial set $BH$. $(\mu,\nu)$ is a $(p,q)$-shuffle and $S_{p,q}$ is the set of $(p,q)$-shuffle for natural numbers $p$ and $q$. $\epsilon(\mu,\nu)$ is the sign of the corresponding permutation with respect to the $(p,q)$-shuffle $(\mu,\nu)$.
\end{definition}

Since each homomorphism in the above composition is a chain map, the edgewise subdivision is a chain map as well.

The edgewise subdivision has a geometric interpretation. One can easily compute
\[Ed_{(f,g)}([g_1])=[f(g_1)] + [g(g_1)]\] 
from Definition~\ref{def:edgewise}. Thus we can interpret $Ed_{(f,g)}([g_1])$ geometrically as the $2$-edgewise subdivision of a $1$-simplex as illustrated in the diagram below:


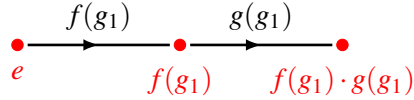
\begin{center}
\tikzset{->-/.style n args={2}{decoration={
  markings,
  mark=at position #1 with {\arrow[line width=1pt]{#2}}},postaction={decorate}}}
\begin{tikzpicture}[scale=1.08, bullet/.style={circle,inner sep=1.5pt,fill}]
 
 \begin{scope}  
  \path
   (0,0) node[red,label=above:](A){}
   (2,0) node[red,label=above:](B){}
   (4,0) node[red,label=below:](C){}
   (0,0) node[red,below=.4em](G){$e$}
   (2,0) node[red,below=.4em](H){$f(g_1)$}
   (4,0) node[red,below=.4em](I){$f(g_1) \cdot g(g_1)$}
   ;
   {\draw [red,fill] (0,0) circle [radius=0.07];}
   {\draw [red,fill] (2,0) circle [radius=0.07];}
   {\draw [red,fill] (4,0) circle [radius=0.07];}
   
   {\draw[line width=1pt,->-={0.5}{latex}] (A) -- 
   node [text width=,midway,above=0.1em,align=center ] {$f(g_1)$} (B);}
   {\draw[line width=1pt,->-={0.5}{latex}] (B) -- 
   node [text width=,midway,above=0.1em,align=center ] {$g(g_1)$} (C);}
  \end{scope} 
  
\end{tikzpicture} 
\end{center}
Similarly we can calculate and geometrically interpret $Ed_{(f,g)}([g_1, g_2])$.
\[Ed_{(f,g)}([g_1, g_2]) = [f(g_1), f(g_2)]{\color{green!60!black} - [f(g_2), g(g_1)]}{\color{red!70!black} +[g(g_1), f(g_2)]}{\color{blue!70!black} + [g(g_1), g(g_2)]}\]
Before interpreting $Ed_{(f,g)}([g_1, g_2])$ geometrically, we define commuting homomorphisms of groups.




\begin{definition}\label{def:commuting homomorphisms}
Suppose $f, g \colon  G \to H$ are group homomorphisms. We say $f$ and $g$  {\em commute} if  $f(g_1)$ and $g(g_2)$ commute for any $g_1$, $g_2$ $\in G$.
\end{definition}

If $f, g \colon  G \to H$ commute then $f(g_2) \cdot g(g_1) = g(g_1) \cdot f(g_2)$. Then we observe that
\begin{align*}
\d Ed_{(f,g)}([g_1, g_2])
&= {\cancel{[f(g_2)]}}-[f(g_1)f(g_2)]+[f(g_1)] {\color{green!60!black} {+\bcancel{[g(g_1)]}}+{\xcancel{[f(g_2)g(g_1)]}}-{\cancel{[f(g_2)]}}}\\
&\phantom{other stuff}{\color{red!70!black} +[f(g_2)]-{\xcancel{[g(g_1)f(g_2)]}}+[g(g_1)]}{\color{blue!70!black} + [g(g_2)]-[g(g_1)g(g_2)]+{\bcancel{[g(g_1)]}}}).
\end{align*}
The following diagram models $Ed_{(f,g)}([g_1, g_2])$ When $f$ and $g$ commute.

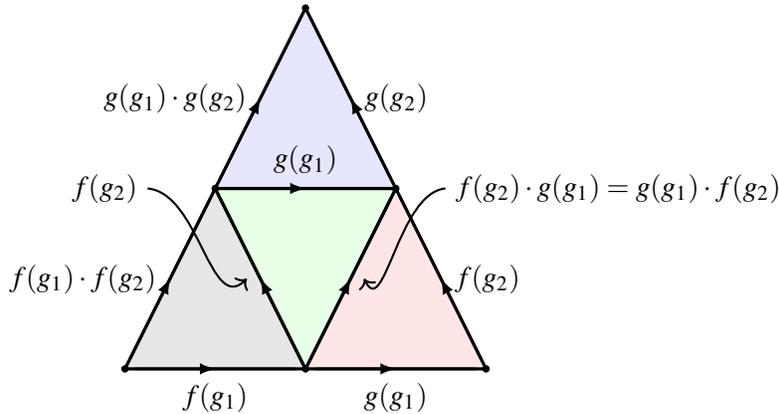
\begin{center}\label{figure:2-simplex}
\tikzset{->-/.style n args={2}{decoration={
  markings,
  mark=at position #1 with {\arrow[line width=1pt]{#2}}},postaction={decorate}}}
\begin{tikzpicture}[scale=0.6, bullet/.style={circle,inner sep=1.5pt,fill}]
 
 \begin{scope}  
  \path
   (0,0) node[red,label=above:](A){}
   (4,0) node[red,label=above:](B){}
   (8,0) node[red,label=below:](C){}

   (6,4) node[red,label=above:](D){}
   (4,8) node[red,label=above:](E){}
   (2,4) node[red,label=below:](F){}
   
   ;
   {\draw[fill=black!10!,very thick] (0,0) -- (4,0) -- (2,4) -- cycle;}
   {\draw[fill=red!10!,very thick] (4,0) -- (8,0) -- (6,4) -- cycle;}
   {\draw[fill=green!10!,very thick] (4,0) -- (6,4) -- (2,4) -- cycle;}
   {\draw[fill=blue!10!,very thick] (2,4) -- (6,4) -- (4,8) -- cycle;}
   
   {\draw [black,fill] (0,0) circle [radius=0.07];}
   {\draw [black,fill] (4,0) circle [radius=0.07];}
   {\draw [black,fill] (8,0) circle [radius=0.07];}
   
   {\draw [black,fill] (6,4) circle [radius=0.07];}
   {\draw [black,fill] (4,8) circle [radius=0.07];}
   {\draw [black,fill] (2,4) circle [radius=0.07];}
   ;
   {\draw[line width=1pt,->-={0.5}{latex}] (A) -- 
   node [text width=,midway,below=0.1em,align=center ] {$f(g_1)$} (B);}
   {\draw[line width=1pt,->-={0.5}{latex}] (B) -- 
   node [text width=,midway,below=0.1em,align=center ] {$g(g_1)$} (C);}

   {\draw[line width=1pt,->-={0.5}{latex}] (C) -- 
   node [text width=,midway,right=0.1em,align=center ] {$f(g_2)$} (D);}
   {\draw[line width=1pt,->-={0.5}{latex}] (D) -- 
   node [text width=,midway,right=0.1em,align=center ] {$g(g_2)$} (E);}

   {\draw[line width=1pt,->-={0.5}{latex}] (A) -- 
   node [text width=,midway,left=0.1em,align=center ] {$f(g_1) \cdot f(g_2)$} (F);}
   {\draw[line width=1pt,->-={0.5}{latex}] (F) -- 
   node [text width=,midway,left=0.1em,align=center ] {$g(g_1) \cdot g(g_2)$} (E);}

   {\draw[line width=1pt,->-={0.5}{latex}] (B) -- 
   node [text width=,midway,above=,align=center ] {}
    (F);}
   {\draw[line width=1pt,->-={0.5}{latex}] (F) -- 
   node [text width=,midway,above=0.1em,align=center ] {$g(g_1)$} (D);}
   {\draw[line width=1pt,->-={0.5}{latex}] (B) -- 
   node [text width=,midway,above=0.1em,align=center ] {} (D);}
  \end{scope}

\draw [->,thick] (.5,4) to [out=-5,in=-135] (2.7,2); 
\node [left] at (.5,4) {$f(g_2)$};

\draw [<-,thick] (5.3,2) to [out=-65,in=-175] (7.1,4); 
\node [right] at (7.1,4) {$f(g_2) \cdot g(g_1) = g(g_1) \cdot f(g_2)$};

\end{tikzpicture} 
\end{center}

\begin{remark}
If $f$ and $g$ do not commute, the green simplex and the red simplex cannot be adjacent.
\end{remark}

One can calculate and model $Ed_{(f,g)}([g_1, g_2, g_3])$ if $f$ and $g$ commute in the same manner.  A model appears on the following  page in Figure~\ref{figure:3simplex}.
\begin{equation*}
\begin{split}
\text{Ed}_{(f,g)}([g_1, g_2, g_3]) &= [f(g_1), f(g_2), f(g_3)] \\
&{\color{green!60!black} + [g(g_1), f(g_2), f(g_3)]-[f(g_2), g(g_1), f(g_3)]+[f(g_2), f(g_3), g(g_1)]}\\
&{\color{red!70!black} + [g(g_1), g(g_2), f(g_3)]-[g(g_1), f(g_3), g(g_2)]+[f(g_3), g(g_1), g(g_2)]}\\
&{\color{blue!70!black} + [g(g_1), g(g_2), g(g_3)]}
\end{split}
\end{equation*}

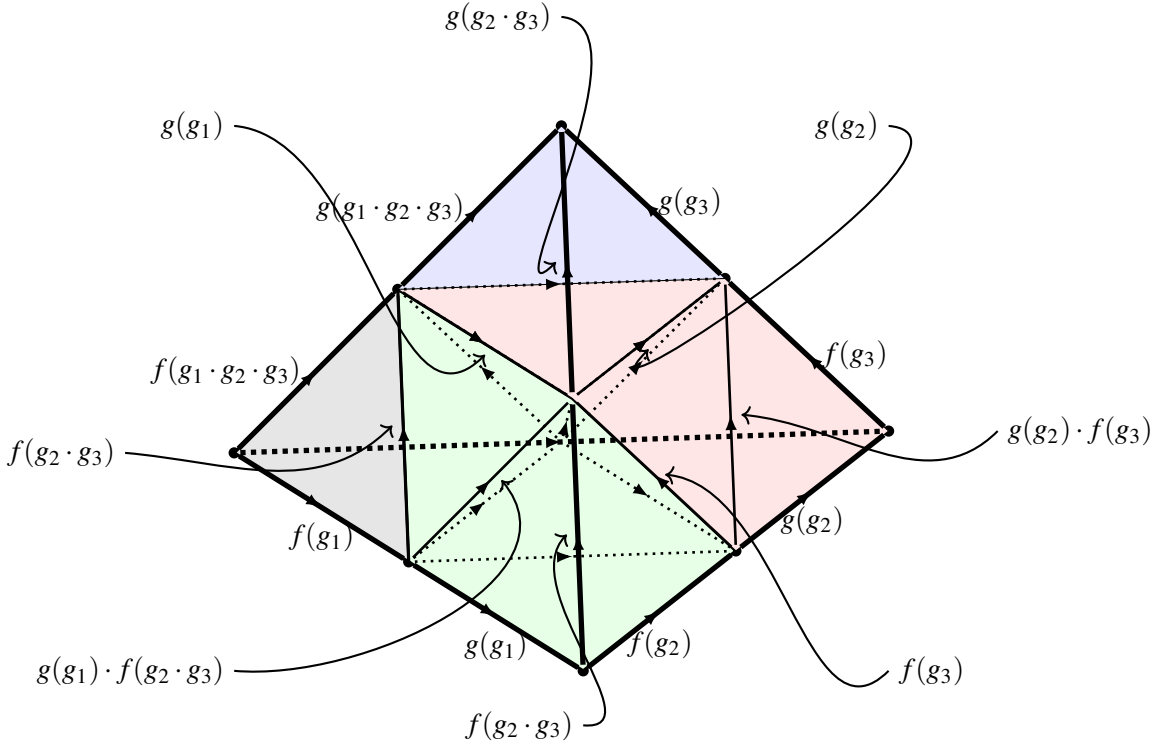
\begin{figure*}
\begin{center}
\tikzset{->-/.style n args={2}{decoration={
  markings,
  mark=at position #1 with {\arrow[line width=1pt]{#2}}},postaction={decorate}}}
\begin{tikzpicture}[scale=1.45, bullet/.style={circle,inner sep=1.5pt,fill}]
 
 \begin{scope}
  \path
   (0,0) node[bullet,label=above:](A){}
   (3.2,-2) node[bullet,label=below:](B){}
   (6,0.2) node[bullet,label=below:](C){}
   (3,3) node[bullet,label=below:](D){}
   
   (1.6,-1) node[bullet,label=above:](E){}
   (4.6,-0.9) node[bullet,label=below:](F){}
   (3,0.1) node[bullet,label=below:](G){} 
   
   (1.5,1.5) node[bullet,label=above:](H){}
   (3.1,0.5) node[bullet,label=below:](I){}
   (4.5,1.6) node[bullet,label=below:](J){}   
   
   (3.05,0.3) node[bullet,label=below:](K){}   
   ;
   
   {\draw[fill=black!10!,very thick] (0,0) -- (1.6,-1) -- (3,0.1) -- (1.5,1.5) -- cycle;}
   {\draw[fill=green!10!,very thick] (1.5,1.5) -- (1.6,-1) -- (3.2,-2) -- (4.6,-0.9) -- (3.1,0.5) -- cycle;}
   {\draw[fill=blue!10!] (1.5,1.5) -- (3.1,0.5) -- (4.5,1.6) -- (3,3) -- cycle;}
   {\draw[fill=red!10!] (1.5,1.5) -- (3.1,0.5) -- (4.6,-0.9) -- (6,0.2) -- (4.5,1.6) -- cycle;}
   ;
   
   {\draw[line width=2pt,->-={0.5}{latex}] (A) -- 
   node [text width=,midway,below=0.1em,align=center ] {$f(g_1)$} (E);}
   {\draw[line width=2pt,->-={0.5}{latex}] (E) -- 
   node [text width=,midway,below=0.1em,align=center ] {$g(g_1)$} (B);}
 
   {\draw[line width=2pt,->-={0.5}{latex}] (B) -- 
   node [text width=,midway,below=0.2em,align=center ] {$f(g_2)$} (F);}
   {\draw[line width=2pt,->-={0.5}{latex}] (F) -- 
   node [text width=,midway,below=0.1em,align=center ] {$g(g_2)$} (C);}
 
   {\draw[line width=2pt,->-={0.5}{latex}] (C) -- 
   node [text width=,midway,right=0.1em,align=center ] {$f(g_3)$} (J);}
   {\draw[line width=2pt,->-={0.5}{latex}] (J) -- 
   node [text width=,midway,right=0.1em,align=center ] {$g(g_3)$} (D);}
 
   {\draw[line width=2pt,->-={0.5}{latex}] (A) -- 
   node [text width=,midway,left=0.1em,align=center ] {$f(g_1 \cdot g_2 \cdot g_3)$} (H);}
   {\draw[line width=2pt,->-={0.5}{latex}] (H) -- 
   node [text width=,midway,left=0.1em,align=center ] {$g(g_1 \cdot g_2 \cdot g_3)$} (D);}

   {\draw[line width=2pt,->-={0.5}{latex}] (B) -- 
   node [text width=1cm,midway,above=0.1em,align=center ] {} (I);}
   {\draw[line width=2pt,->-={0.5}{latex}] (I) -- 
   node [text width=1cm,midway,above=0.1em,align=center ] {} (D);}
   
   {\draw[line width=1pt,->-={0.5}{latex}] (E) -- 
   node [text width=1.5cm,midway,above=0.1em,align=center ] {} (H);}
   {\draw[line width=1pt,->-={0.5}{latex}] (E) -- 
   node [text width=1cm,midway,above=0.1em,align=center ] {} (I);}
   {\draw[line width=1pt,->-={0.5}{latex}] (H) -- 
   node [text width=1cm,midway,above=0.1em,align=center ] {} (I);}

   {\draw[line width=1pt,->-={0.5}{latex}] (F) -- 
   node [text width=1cm,midway,above=0.1em,align=center ] {} (I);}
   {\draw[line width=1pt,->-={0.5}{latex}] (F) -- 
   node [text width=1cm,midway,above=0.1em,align=center ] {} (J);}
   {\draw[line width=1pt,->-={0.5}{latex}] (I) -- 
   node [text width=1cm,midway,above=0.1em,align=center ] {} (J);}

   {\draw[dotted,line width=2pt,->-={0.5}{latex}] (A) --
   node [text width=2.5cm,midway,above=0.1em,align=center ] {} (C);}

   {\draw[dotted,line width=1pt,->-={0.5}{latex}] (E) --
   node [text width=2.5cm,midway,above=0.1em,align=center ] {} (F);}
   {\draw[dotted,line width=1pt,->-={0.5}{latex}] (E) --
   node [text width=2.5cm,midway,above=0.1em,align=center ] {} (G);}
   {\draw[dotted,line width=1pt,->-={0.5}{latex}] (G) --
   node [text width=2.5cm,midway,above=0.1em,align=center ] {} (F);}

   {\draw[dotted,line width=1pt,->-={0.5}{latex}] (H) --
   node [text width=2.5cm,midway,above=0.1em,align=center ] {} (I);}
   {\draw[dotted,line width=1pt,->-={0.5}{latex}] (H) --
   node [text width=2.5cm,midway,above=0.1em,align=center ] {} (J);}
   {\draw[dotted,line width=1pt,->-={0.5}{latex}] (I) --
   node [text width=2.5cm,midway,above=0.1em,align=center ] {} (J);}

   {\draw[dotted,line width=1pt,->-={0.5}{latex}] (G) --
   node [text width=2.5cm,midway,above=0.1em,align=center ] {} (H);}
   {\draw[dotted,line width=1pt,->-={0.5}{latex}] (G) --
   node [text width=2.5cm,midway,above=0.1em,align=center ] {} (J);}

   {\draw[dotted,line width=1pt,->-={0.5}{latex}] (G) --
   node [text width=2.5cm,midway,above=0.1em,align=center ] {} (I);}
  \end{scope}

\draw [->,thick] (0,3) to [out=-5,in=-135] (2.3,.9); 
\node [left] at (0,3) {$g(g_1)$};

\draw [->,thick] (6,3) to [out=-5,in=-135] (3.8,.95); 
\node [left] at (6,3) {$g(g_2)$};

\draw [->,thick] (3,4) to [out=-5,in=-135] (2.95,1.75); 
\node [left] at (3,4) {$g(g_2 \cdot g_3)$};

\draw [->,thick] (3.2,-2.5) to [out=-5,in=-135] (3.05,-0.75); 
\node [left] at (3.2,-2.5) {$f(g_2 \cdot g_3)$};

\draw [->,thick] (0,-2) to [out=-5,in=-50] (2.45,-0.25); 
\node [left] at (0,-2) {$g(g_1) \cdot f(g_2 \cdot g_3)$};

\draw [->,thick] (-1,0) to [out=-5,in=-135] (1.45,0.25); 
\node [left] at (-1,0) {$f(g_2 \cdot g_3)$};

\draw [->,thick] (6,-2) to [out=-135,in=-5] (3.95,-0.2); 
\node [right] at (6,-2) {$f(g_3)$};

\draw [->,thick] (7,0.2) to [out=-135,in=-5] (4.65,0.35); 
\node [right] at (7,0.2) {$g(g_2) \cdot f(g_3)$};
\end{tikzpicture}
\end{center}
\caption{This picture subdivides a standard $3$-simplex into $9$ different $3$-simplices, one colored black, one blue, three red, and three green.  Each of the four vertices of the large $3$-simplex is contained in a distinctly colored $3$-dimensional sub-simplex.  Furthermore, the red simplices, as well as the green simplices form a connected sub-polyhedron.  The reader who struggles to see this picture may benefit from fitting the $3$-simplices of a fixed color together along a common faces.
}\label{figure:3simplex}
\end{figure*}

In general, we can describe $Ed_{(f,g)}$ using {\em $(p,q)$-shuffles}. First, we remind the reader of the  definition of a $(p,q)$-shuffle.

\begin{definition}\label{def:pq shuffle}
Let $S(n)$ be the group of permutations of the set $\{1, 2, \ldots, n\}.$ For $p, q \in \N \cup \{0\}$ a permutation $\mu \in S(p+q)$ is a $(p,q)$-shuffle if
\[
\begin{split}
\mu(1) &< \mu(2) < \cdots < \mu(p),\\
\mu(p+1) &< \mu(p+2) < \cdots < \mu(p+q).
\end{split}
\]
We denote the set of $(p,q)$-shuffles by $S_{p,q}$.  

The set $S_{p,q}$ is an ordered set.  Given two $(p,q)$-shuffles, $\mu$ and $\nu$, we say {\em $\mu < \nu$} if
\[
(\mu(1), \ldots, \mu(p), \mu(p+1), \ldots, \mu(p+q))< (\nu(1), \ldots, \nu(p), \nu(p+1), \ldots, \nu(p+q))
\]
in the dictionary order.
\end{definition}

The cardinality of $S_{p,q}$ is $\binom{p+q}{p}$ since a $(p,q)$-shuffle $\mu$ is determined by the values $\mu(j)$ for $j = 1 \ldots, p$.

We denote the $j$-th $(p,q)$-shuffle in the ordered set of $(p,q)$-shuffles by $S_j^{p,q}$. In other words, as an ordered set, we have
\[
S_{p,q}=\{S_1^{p,q}, S_2^{p,q}, \ldots, S_{C(p+q,p)}^{p,q}\}.
\]
For $p,q$ such that $p+q=n$, $\mu = S_j^{p,q}$, and 
$\sigma=[g_1,\cdots,g_n]$, define a formal function \[
{S_{(f,g)}}_j^{p,q}\sigma:=T_\ast \circ (s_{\mu(n)} \ldots s_{\mu(p+1)}[g(g_1),\ldots,g(g_p)] \times s_{\mu(p)} \ldots s_{\mu(1)}[f(g_{p+1}),\ldots,f(g_{n})]).
\]
One easily checks that 
\begin{align}
Ed_{f,g}(\sigma)=\Sigma_{i=0}^{n} \Sigma_{j=1}^{\binom{n}{i}} \sign(S_{j}^{i,n-i}) {S_{(f,g)}}_{j}^{i,n-i} \sigma\label{eq:order}
\end{align}
where $\sign(S_{j}^{i,n-i})$ is the sign of the permutation  $S_{j}^{i,n-i}$.

\noindent {\bf Note:} Throughout the rest of this paper, we regard $Ed_{f,g}(\sigma)$ as the sum in the order of the equation~(\ref{eq:order}).

\begin{example} To help the reader we check the following example:
\begin{align*}
Ed_{(f,g)}([g_1, g_2, g_3])&=[f(g_1), f(g_2), f(g_3)]\\
&\textcolor{green!60!black}{+[g(g_1), f(g_2), f(g_3)]-[f(g_2), g(g_1), f(g_3)]+[f(g_2), f(g_3), g(g_1)]}\\
&\textcolor{red!70!black}{+[g(g_1), g(g_2), f(g_3)]-[g(g_1), f(g_3), g(g_2)]+[f(g_3), g(g_1), g(g_2)]}\\
&\textcolor{blue!70!black}{+[g(g_1), g(g_2), g(g_3)]}\\ 
& = {S_{(f,g)}}_{1}^{0,3} \sigma \\
&\textcolor{green!60!black}{+{S_{(f,g)}}_{1}^{1,2} \sigma-{S_{(f,g)}}_{2}^{1,2} \sigma+{S_{(f,g)}}_{3}^{1,2} \sigma} \\
&\textcolor{red!70!black}{+{S_{(f,g)}}_{1}^{2,1} \sigma-{S_{(f,g)}}_{2}^{2,1} \sigma+{S_{(f,g)}}_{3}^{2,1} \sigma}\\
&\textcolor{blue!70!black}{+{S_{(f,g)}}_{1}^{3,0} \sigma}.
\end{align*}
\end{example}

\subsection{Simplicial cylinders}

We define {\em a simplicial cylinder} and investigate a property of this object by proving a lemma.

\begin{definition}\label{def:simplicial cylinder}
Let $G$ be a group and $a_i, b_i, t_0, t_i \in G$ where $i \in \{1,2,\ldots,n\}$. Suppose there are relations ~$t_i \cdot a_{(i+1)}=b_{(i+1)} \cdot t_{(i+1)}$ for $i \in \{0,1,2,\ldots,(n-1)\}$.

For two $n$-simplices $\sigma = [a_1, \ldots, a_n]$ and $\tau = [b_1, \dots, b_n]$, the {\em simplicial cylinder $Cyl(\sigma,\tau, T)$  between $\sigma$ and $\tau$ and related by the ordered set $T=\{t_0, t_1, \ldots, t_n\}$} is an $(n+1)$-chain defined by:
\begin{align*}
Cyl(\sigma,\tau,T) & = [t_0,a_1,a_2,\ldots, a_n]\\
                   & -[b_1,t_1,a_2,a_3,\ldots, a_n]\\
                   & + \qquad \cdots \\
                   & +(-1)^{n}[b_1,\ldots, b_n, t_n].
\end{align*}
\end{definition}

We can model the simplicial cylinder as a simplicial subdivision of a product of an $n$-simplex and a $1$-simplex. For example, for $\sigma = [a_1,a_2]$ and $\tau = [b_1,b_2]$ with ordered set $T = \{t_0, t_1, t_2\}$, the associated simplicial cylinder $Cyl(\sigma,\tau,T)$ is modeled below:

\begin{center}
\tikzset{->-/.style n args={2}{decoration={
  markings,
  mark=at position #1 with {\arrow[line width=1pt]{#2}}},postaction={decorate}}}
\begin{tikzpicture}[scale=1.05, bullet/.style={circle,inner sep=1.5pt,fill}]
 
 \begin{scope}
  \path
   (-3,4) node[label=below:](B){}
   (0,3) node[label=below:](C){}
   (-3,1) node[label=below:](E){}
   (0,0) node[label=below:](F){}
   (3,4) node[label=below:](G){}
   (3,1) node[label=below:](I){}
   ;
   {\draw[fill=red!10!,very thick] (-3,1) -- (-3,4) -- (3,4) -- (0,3) -- cycle;}
   {\draw[fill=blue!10!,very thick] (-3,1) -- (0,0) -- (3,1) -- (3,4) -- cycle;}
   {\draw[fill=green!10!,very thick] (-3,1) -- (0,3) -- (3,4) -- (0,0) -- cycle;}   

   {\draw [black,fill] (-3,4) circle [radius=0.07];}
   {\draw [black,fill] (0,3) circle [radius=0.07];}
   {\draw [black,fill] (-3,1) circle [radius=0.07];}
   {\draw [black,fill] (0,0) circle [radius=0.07];}
   {\draw [black,fill] (3,4) circle [radius=0.07];}
   {\draw [black,fill] (3,1) circle [radius=0.07];}
   
   {\draw[line width=1pt,->-={0.5}{latex}] (B) -- 
   node [text width=2.5cm,midway,below=0.1em,align=center ] {$a_1$} (C);}
   {\draw[line width=1pt,->-={0.5}{latex}] (E) --
   node [text width=2.5cm,midway,below=0.1em,align=center ] {$b_1$} (F);}
   {\draw[line width=1pt,->-={0.5}{latex}] (E) --
   node [text width=0.3cm,midway,below=0.5em,left=0.1em,align=center ] {$t_0$} (B);}
   {\draw[line width=1pt,->-={0.5}{latex}] (E) --
   node [text width=2.5cm,midway,above=0.1em,align=center ] {} (C);}
   {\draw[line width=1pt,->-={0.5}{latex}] (C) -- 
   node [text width=2.5cm,midway,below=0.1em,align=center ] {$a_2$} (G);}
   {\draw[line width=1pt,->-={0.5}{latex}] (F) --
   node [text width=0.4cm,midway,left=0.1em,align=center ] {$t_1$} (C);}
   {\draw[line width=1pt,->-={0.5}{latex}] (F) --
   node [text width=2.5cm,midway,below=0.1em,align=center ] {$b_2$} (I);}
   {\draw[line width=1pt,->-={0.5}{latex}] (F) --
   node [text width=2.5cm,midway,above=0.1em,align=center ] {} (G);}
   {\draw[line width=1pt,->-={0.5}{latex}] (I) --
   node [text width=0.3cm,midway,below=0.5em,left=0.1em,align=center ] {$t_2$} (G);}
   {\draw[line width=1pt,->-={0.5}{latex}] (B) --
   node [text width=0.3cm,midway,below=0.5em,above=0.1em,align=center ] {} (G);}
   {\draw[dotted,line width=1pt,->-={0.5}{latex}] (E) --
   node [text width=2.5cm,midway,above=0.1em,align=center ] {} (G);}
   {\draw[dotted,line width=1pt,->-={0.5}{latex}] (E) --
   node [text width=2.5cm,midway,above=0.1em,align=center ] {} (I);}     
  \end{scope} 

 \begin{scope}[xshift=5.5cm]
  \path
   (0,3.5) node[label=above:](A){$Cyl(\sigma,\tau,T)=$}
   (0,3) node[red!70!black,label=above:](A){$+[t_0,a_1,a_2]$}
   (0,2.5) node[green!60!black,label=above:](A){$-[b_1,t_1,a_2]$}
   (0,2) node[blue!70!black,label=above:](A){$+[b_1,b_2,t_2]$}
      ;
  \end{scope}
\end{tikzpicture} 
\end{center}

\begin{definition}
For an ordered set $T=\{t_0,t_1,\ldots,t_n\}$ and each $i \in \{0,1,2,\cdots,n\}$, define an ordered set $d_i(T)$ by
\[
d_i(T):=\{t_0,\ldots,t_{(i-1)},\hat{t_i},t_{(i+1)},\ldots,t_n\}.
\]
\end{definition}

\begin{lemma}\label{lem:cylinder}

$\d Cyl(\sigma,\tau,T)= (\sigma - \tau) - \Sigma_{i=0}^{n}(-1)^{i} Cyl(d_i \sigma, d_i \tau, d_i (T))$

\end{lemma}

\begin{proof}[Proof of Lemma~\ref{lem:cylinder}]
For any $n$,
\begin{align*}
\d Cyl(\sigma,\tau,T) &=\d [t_0,a_1,a_2,\ldots,a_n]\\
                      &-\d [b_1,t_1,a_2,a_3,\ldots,a_n]\\
                      &+\d [b_1,b_2,t_2,a_3,a_4,\ldots,a_n]\\
                      &-\d [b_1,b_2,b_3,t_3,a_4,\ldots,a_n]\\
                      &+ \qquad \ldots\\
                      &+(-1)^{n}\d [b_1,\ldots,b_n,t_n].
\end{align*}
We expand the right hand side. We use a smaller font to accommodate a long computation.
{\tiny
\begin{alignat*}{4}
=[a_1,\ldots, a_n]-{\bcancel{[t_0 \cdot a_1,a_2,\ldots,a_n]}}\quad &+&\quad \textcolor{green!60!black}{[t_0,a_1 \cdot a_2,a_3,\ldots,a_n]}-\textcolor{blue}{[t_0,a_1,a_2 \cdot a_3,\ldots,a_n]}\quad &+\ldots +&\quad \textcolor{brown}{(-1)^{(n+2)+1}[t_0,a_1,\ldots,a_{(n-1)}]}\\
-\textcolor{red!70!black}{[t_1,a_1,\ldots,a_n]}+{\bcancel{[b_1 \cdot t_1,a_2,\ldots,a_n]}}\quad &-&\quad{\cancel{[b_1,t_1 \cdot a_2,a_3,\ldots,a_n]}}+\textcolor{blue}{[b_1,t_1,a_2 \cdot a_3,a_4,\ldots,a_n]}\quad &-\cdots+&\quad \textcolor{brown}{(-1)^{(n+2)}[b_1,t_1,a_2,\ldots,a_{(n-1)}]}\\
+\textcolor{red!70!black}{[b_2,t_2,a_3,\ldots,a_n]}-\textcolor{green!60!black}{[b_1 \cdot b_2,t_2,a_3,\ldots,a_n]}\quad &+&\quad {\cancel{[b_1,b_2 \cdot t_2,a_3,\ldots,a_n]}}-{\xcancel{[b_1,b_2,t_2 \cdot a_3,a_4,\ldots,a_n]}}\quad &+\cdots+ &\quad\textcolor{brown}{(-1)^{(n+2)+1}[b_1,b_2,t_2,a_3,\ldots,a_{(n-1)}]}\\
-\textcolor{red!70!black}{[b_2,b_3,t_3,a_4,\ldots,a_n]}+\textcolor{green!60!black}{[b_1 \cdot b_2,b_3,t_3,a_4,\ldots,a_n]}\quad&-&\quad\textcolor{blue}{[b_1,b_2 \cdot b_3,t_3,a_4,\ldots,a_n]}+{\xcancel{[b_1,b_2,b_3 \cdot t_3,a_4,\ldots,a_n]}}\quad &-\cdots+&\quad\textcolor{brown}{(-1)^{(n+2)}[b_1,b_2,b_3,t_3,a_4,\ldots,a_{(n-1)}]}\\
&\vdots & &\vdots &\\ 
+\textcolor{red!70!black}{(-1)^{n}[b_2,\ldots,b_n,t_n]}\textcolor{green!60!black}{+(-1)^{n+1}[b_1 \cdot b_2,b_3,\ldots,b_n,t_n]}\quad &+&\quad\textcolor{blue}{(-1)^{n+2}[b_1,b_2\cdot b_3,b_4,\ldots,t_n]} +\textcolor{magenta}{(-1)^{n+3}[b_1,b_2,b_3\cdot b_4\ldots,,t_n]}\quad &-\cdots + & (-1)^{n}(-1)^{(n+2)+1}[b_1,\ldots,b_n]
\end{alignat*}
}
Notice that the first term and the last one are $\sigma$ and $\tau$ respectively. After canceling, where possible, we combine terms of the same color to obtain:
\begin{align*}
    & \d Cyl(\sigma,\tau,T) &\\
    &=\sigma\textcolor{red!70!black}{-Cyl(d_0 \sigma, d_0 \tau, d_0 T)}\textcolor{green!60!black}{+Cyl(d_1 \sigma, d_1 \tau, d_1 T)}\textcolor{blue!70!black}{-Cyl(d_2 \sigma, d_2 \tau, d_2 T)}\\
    & \textcolor{magenta}{+Cyl(d_3 \sigma, d_3 \tau, d_3 T)}-\ldots \textcolor{brown}{+(-1)^{(n+1)}Cyl(d_n \sigma, d_n \tau, d_n T)}-\tau\\
    &=\sigma-(\Sigma_{i=0}^{n}(-1)^{i}Cyl(d_i \sigma, d_i \tau, d_i T))-\tau\\
    &=(\sigma-\tau)-(\Sigma_{i=0}^{n}(-1)^{i}Cyl(d_i \sigma, d_i \tau, d_i T)).
\end{align*}
\end{proof}

The simplicial cylinder $Cyl(d_i \sigma, d_i \tau, d_i T)$ models the $i$-th side of the cylinder $Cyl(\sigma,\tau,T)$ for any $i \in \{0,1,2,\ldots,n\}$.

We observe a cancellation property of simplicial cylinders.

\begin{lemma}\label{lem:canceling}
$Cyl(\sigma,\tau,T)$ is a simplicial cylinder between $\sigma = [a_1,\cdots,a_n]$ and $\tau = [b_1,\cdots,b_n]$ with respect to $T=\{t_0,t_1,\cdots,t_n\}$. $Cyl(\mu,\nu,U)$ is a simplicial cylinder between $\mu = [c_1,\cdots,c_n]$ and $\nu = [d_1,\cdots,d_n]$ with respect to $U=\{u_0,u_1,\cdots,u_n\}$. Assume there are some $s, m \in \{0,1,\cdots\}$ such that 
\[
(-1)^{s} d_s \sigma + (-1)^{m} d_m \mu = 0,
\]
\[
(-1)^{s} d_s \tau + (-1)^{m} d_m \nu =0,
\]
and 
\[
d_s T = d_m U.
\]
Then 
\begin{equation*}
\begin{split}
\d (Cyl(\sigma,\tau,T)+Cyl(\mu,\nu,U)) =
&(\sigma+\mu)-(\tau+\nu)-(\Sigma_{i=0,i\neq s}^{n}(-1)^{i}Cyl(d_i \sigma, d_i \tau, d_i T)\\
&+\Sigma_{j=0,j\neq m}^{n}(-1)^{j}Cyl(d_j \mu, d_j \nu, d_j U)).
\end{split}
\end{equation*}
\end{lemma}

\begin{proof}[Proof of Lemma~\ref{lem:canceling}]
This follows immediately from  lemma~\ref{lem:cylinder}.
\end{proof}

The diagram below models the result in lemma~\ref{lem:canceling}.

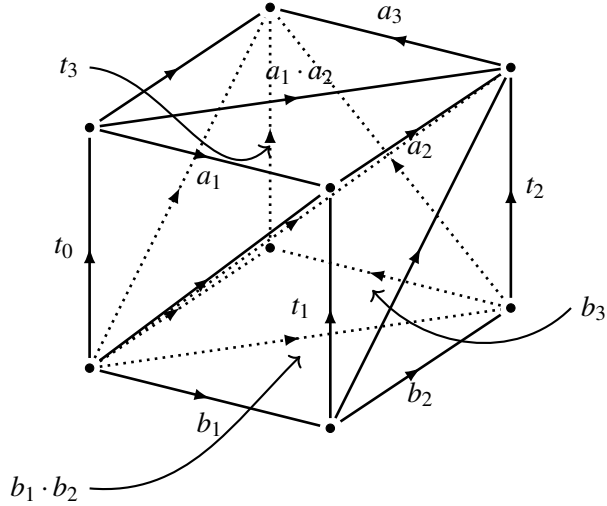
\begin{center}
\tikzset{->-/.style n args={2}{decoration={
  markings,
  mark=at position #1 with {\arrow[line width=1pt]{#2}}},postaction={decorate}}}
\begin{tikzpicture}[scale=0.8, bullet/.style={circle,inner sep=1.5pt,fill}]
 
 \begin{scope}
  \path
   (0,5) node[label=below:](B){}
   (4,4) node[label=below:](C){}
   (0,1) node[label=below:](E){}
   (4,0) node[label=below:](F){}
   (7,6) node[label=below:](G){}
   (7,2) node[label=below:](I){}

   (3,7) node[label=below:](A){}
   (3,3) node[label=below:](D){}
   ;
   {\draw [black,fill] (0,5) circle [radius=0.07];}
   {\draw [black,fill] (4,4) circle [radius=0.07];}
   {\draw [black,fill] (0,1) circle [radius=0.07];}
   {\draw [black,fill] (4,0) circle [radius=0.07];}
   {\draw [black,fill] (7,6) circle [radius=0.07];}
   {\draw [black,fill] (7,2) circle [radius=0.07];}

   {\draw [black,fill] (3,7) circle [radius=0.07];}
   {\draw [black,fill] (3,3) circle [radius=0.07];}

   {\draw[line width=1pt,->-={0.5}{latex}] (B) -- 
   node [text width=2.5cm,midway,below=0.1em,align=center ] {$a_1$} (C);}
   {\draw[line width=1pt,->-={0.5}{latex}] (E) --
   node [text width=2.5cm,midway,below=0.1em,align=center ] {$b_1$} (F);}
   {\draw[line width=1pt,->-={0.5}{latex}] (E) --
   node [text width=0.3cm,midway,below=0.5em,left=0.1em,align=center ] {$t_0$} (B);}
   {\draw[line width=1pt,->-={0.5}{latex}] (E) --
   node [text width=2.5cm,midway,above=0.1em,align=center ] {} (C);}
   {\draw[line width=1pt,->-={0.5}{latex}] (C) -- 
   node [text width=2.5cm,midway,below=0.1em,align=center ] {$a_2$} (G);}
   {\draw[line width=1pt,->-={0.5}{latex}] (F) --
   node [text width=0.4cm,midway,left=0.1em,align=center ] {$t_1$} (C);}
   {\draw[line width=1pt,->-={0.5}{latex}] (F) --
   node [text width=2.5cm,midway,below=0.1em,align=center ] {$b_2$} (I);}
   {\draw[line width=1pt,->-={0.5}{latex}] (F) --
   node [text width=2.5cm,midway,above=0.1em,align=center ] {} (G);}
   {\draw[line width=1pt,->-={0.5}{latex}] (I) --
   node [text width=0.3cm,midway,below=0.5em,right=0.1em,align=center ] {$t_2$} (G);}
   {\draw[line width=1pt,->-={0.5}{latex}] (B) --
   node [text width=,midway,below=0.5em,above=0.1em,align=center ] {$a_1 \cdot a_2$} (G);}
   {\draw[dotted,line width=1pt,->-={0.5}{latex}] (E) --
   node [text width=2.5cm,midway,above=0.1em,align=center ] {} (G);}
   {\draw[dotted,line width=1pt,->-={0.5}{latex}] (E) --
   node [text width=2.5cm,midway,below=0.1em,align=center ] {} (I);}     

   {\draw[line width=1pt,->-={0.5}{latex}] (G) --
   node [text width=2.5cm,midway,above=0.1em,align=center ] {$a_3$} (A);}     
   {\draw[line width=1pt,->-={0.5}{latex}] (B) --
   node [text width=2.5cm,midway,above=0.1em,align=center ] {} (A);}     
   {\draw[dotted,line width=1pt,->-={0.5}{latex}] (D) --
   node [text width=2.5cm,midway,below=0.5em,align=center ] {} (A);}     
   {\draw[dotted,line width=1pt,->-={0.6}{latex}] (I) --
   node [text width=2.5cm,midway,above=0.1em,align=center ] {} (D);}     
   {\draw[dotted,line width=1pt,->-={0.5}{latex}] (E) --
   node [text width=2.5cm,midway,above=0.1em,align=center ] {} (D);}     
   {\draw[dotted,line width=1pt,->-={0.5}{latex}] (E) --
   node [text width=2.5cm,midway,above=0.1em,align=center ] {} (A);}     
   {\draw[dotted,line width=1pt,->-={0.5}{latex}] (I) --
   node [text width=2.5cm,midway,above=0.1em,align=center ] {} (A);}     
  \end{scope} 

\draw [->,thick] (0,-1) to [out=-5,in=-135] (3.5,1.2); 
\node [left] at (0,-1) {$b_1 \cdot b_2$};

\draw [->,thick] (8,2) to [out=-135,in=-50] (4.7,2.3); 
\node [right] at (8,2) {$b_3$};

\draw [->,thick] (0,6) to [out=-5,in=-135] (3,4.7); 
\node [left] at (0,6) {$t_3$};

\end{tikzpicture} 
\end{center}
For example, suppose $\sigma = [a_1,a_2]$, $\tau = [b_1,b_2]$, $\mu = [a_1 \cdot a_2, a_3]$, $\nu = [b_1 \cdot b_2, b_3]$, $T = \{t_0,t_1,t_2\}$, and $U = \{t_0,t_2,t_3\}$. Then, we have the condition $d_1 \sigma = d_2 \mu$, $d_1 \tau = d_2 \nu$, and $d_1 T = d_2 U$. Geometrically, the condition means $Cyl(\sigma,\tau,T)$ and $Cyl(\mu,\nu,U)$ share a side of a $2$-dimensional simplicial cylinder in the above picture. By lemma~\ref{lem:canceling}, $\d ( Cyl(\sigma,\tau,T)+Cyl(\sigma,\tau,T))$ is the top and bottom simplices and four sides of $2$-dimensional simplicial cylinders. Any two sides that face each other will cancel in the computation. Inductively, we can apply this observation to a sum of simplicial cylinders.

We introduce a concept that we call a {\em system of pillars}.

\begin{definition}\label{def:systemofpillars}
Let $\sigma_i$ and $\tau_i$ be $n$-simplices where $\sigma_i := [a_1^i, \cdots, a_n^i]$ and $\tau_i := [b_1^i, \cdots, b_n^i]$. For $n$-chains $\Sigma_{i=1}^m \sigma_i$ and $\Sigma_{i=1}^m \tau_i$, {\em the system of pillars with respect to $\Sigma_{i=1}^m \sigma_i$ and $\Sigma_{i=1}^m \tau_i$} is a set 
\[
T_{\tau}^{\sigma} := \{T_1, \cdots, T_m \}
\]
where $T_i = \{t_0^i, t_1^i, \cdots, t_n^i \}$ such that, for any $i$ and $j$, $t_j^i \cdot a_{j+1}^i = b_{j+1}^i \cdot t_{j+1}^i$. 
\end{definition}

Since $\sigma_i$, $\tau_i$, and $T_i$ satisfy the condition in Definition~\ref{lem:cylinder}, we can define a {simplicial cylinder between two chains with a system of pillars} as follows.

\begin{definition}\label{def:simplicial cylinder with pillars}
For two $n$-chains $\Sigma_{i=1}^m \sigma_i$, $\Sigma_{i=1}^m \tau_i$, suppose we have a system of pillars $T_{\tau}^{\sigma}$. Define {\em a simplicial cylinder between $\Sigma_{i=1}^m \sigma_i$ and $\Sigma_{i=1}^m \tau_i$} related with $T_{\tau}^{\sigma}$ as 
\[
Cyl(\Sigma_{i=1}^m \sigma_i,  \Sigma_{i=1}^m \tau_i, T_{\tau}^{\sigma}) := \Sigma_{i=1}^{m} Cyl(\sigma_i,\tau_i,T_i).
\]
\end{definition}

Notice $\d \Sigma_{i=1}^m \sigma_i = \Sigma_{i=1}^m \d \sigma_i = \Sigma_{i=1}^m \Sigma_{k=0}^{n} d_k \sigma_i$.

\begin{definition}\label{def:bdsys}
For two $n$-chains $\Sigma_{i=1}^m \sigma_i$, $\Sigma_{i=1}^m \tau_i$, suppose we have a system of pillars $T_{\tau}^{\sigma}$. Define $\d T_{\tau}^{\sigma} := \{d_k T_i : i=1, \ldots, m \text{ and } k = 0, \cdots, n\}$.
\end{definition}

Then, one can easily check that $\d T_{\tau}^{\sigma}$ is a system of pillars between $\d \Sigma_{i=1}^m \sigma_i$ and $\d \Sigma_{i=1}^m \tau_i$ canonically. Thus, we obtain the below lemma.

\begin{lemma}\label{lem:general canceling}
For $\Sigma \sigma$, $\Sigma \tau$, and the system of pillars $T_{\tau}^{\sigma}$, 
\[
\d Cyl(\Sigma \sigma, \Sigma \tau, T_{\tau}^{\sigma}) = \Sigma \sigma - \Sigma \tau - Cyl(\d \Sigma \sigma, \d \Sigma \tau, \d T_{\tau}^{\sigma}).
\]
\end{lemma}

\begin{proof}
This follows readily from Definition~\ref{def:bdsys} and Lemma~\ref{lem:cylinder}.
\end{proof}

\subsection{Controlled simplicial chain homotopy between edgewise subdivisions}

We show a key ingredient used to prove Theorem~\ref{thm:family}.

\begin{theorem}\label{thm:edgewise chain homotopy}
For groups $G$ and $H$, let $f, g, h, k$ be homomorphisms from $G$ to $H$ such that $f$ and $g$  commute, as do $h$ and $k$. Suppose there is an element $\ell \in H$ such that 
\begin{equation}\label{eq:group relation}
\ell \cdot f(x) \cdot g(x) = h(x) \cdot k(x) \cdot \ell
\end{equation} for any $x \in G$. Then $Ed_{(f,g)}$ and $Ed_{(h,k)}$ are chain homotopic. 

Moreover, there is a simplicial chain homotopy $P_{h,k}^{f,g}$ between the chain maps \[
Ed_{(f,g)}, Ed_{(h,k)} \colon  \Z BG_\ast \to \Z BH_\ast,
\]
whose diameter function is exactly $d(n)$. For $n \leq 7$, the value of $d(n)$ is given in the chart below:

\begin{center}
 \begin{tabular}{c c c c c c c c c} 
 \hline
 n & 0 & 1 & 2 & 3 & 4 & 5 & 6 & 7\\ [0.5ex]
 $d(n)$ & 0 & 4 & 12 & 32 & 80 & 192 & 448 & 1024\\ 
 \hline
 \end{tabular}
\end{center}

In fact, the general term for $d(n)$ is given as follows:

\begin{center}
    $d(n)=2^{n} \cdot (n+1)$.
\end{center}

\end{theorem}

\begin{proof}

In this proof, we use the abbreviated notation $P$ to denote the simplicial chain homotopy  $P_{h,k}^{f,g}$. Our goal is to construct the simplicial chain homotopy $P$. We first introduce explicit construction of $P_0, P_1, P_2,$ and $P_3$. We then explain how to obtain general formula of $P_n$ for any $n \in \N \cup \{0\}$. 

First we define $P_0([\quad])=0$. In other words, since $[\quad] = 0 \in ZBG_{\ast}$, we define $P_0(0) = 0$. Thus, $d(0)=0$.

Notice that the hypothesized displayed relation~(\ref{eq:group relation}) in $G$, from Theorem~\ref{thm:edgewise chain homotopy}, provides a $2$-chain in the Moore complex as modeled below. We have added an additional edge, subdividing the rectangle into two squares.

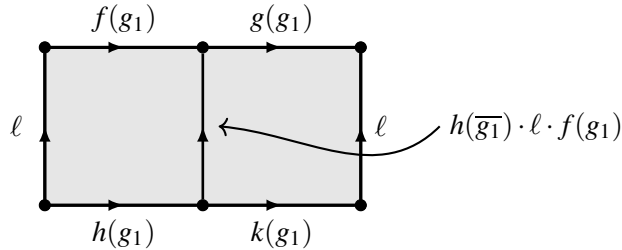
\begin{center}
\tikzset{->-/.style n args={2}{decoration={
  markings,
  mark=at position #1 with {\arrow[line width=1pt]{#2}}},postaction={decorate}}}
\begin{tikzpicture}[scale=1.05, bullet/.style={circle,inner sep=1.5pt,fill}]
 
 \begin{scope}
  \path
   (0,2) node[bullet,label=above:](A){}
   (2,2) node[bullet,label=below:](B){}
   (4,2) node[bullet,label=below:](C){}
   (0,0) node[bullet,label=below:](D){}
   (2,0) node[bullet,label=below:](E){}
   (4,0) node[bullet,label=below:](F){}
   
   ;
   {\draw[fill=black!10!,very thick] (0,0) rectangle (4,2);}
   {\draw [black,fill] (0,0) circle [radius=0.07];}
   {\draw [black,fill] (2,0) circle [radius=0.07];}
   {\draw [black,fill] (0,2) circle [radius=0.07];}
   {\draw [black,fill] (2,2) circle [radius=0.07];}
   {\draw [black,fill] (4,0) circle [radius=0.07];}
   {\draw [black,fill] (4,2) circle [radius=0.07];}
   {\draw[line width=1pt,->-={0.5}{latex}] (A) -- 
   node [text width=2.5cm,midway,above=0.1em,align=center ] {$f(g_1)$} (B);}
   {\draw[line width=1pt,->-={0.5}{latex}] (B) -- 
   node [text width=2.5cm,midway,above=0.1em,align=center ] {$g(g_1)$} (C);}
   {\draw[line width=1pt,->-={0.5}{latex}] (D) --
   node [text width=0.4cm,midway,left=0.1em,align=center ] {$\ell$} (A);}
   {\draw[line width=1pt,->-={0.5}{latex}] (F) -- 
   node [text width=0.1cm,midway,right=0.1em,align=center ] {$\ell$} (C);}
   {\draw[line width=1pt,->-={0.5}{latex}] (D) --
   node [text width=2.5cm,midway,below=0.1em,align=center ] {$h(g_1)$} (E);}
   {\draw[line width=1pt,->-={0.5}{latex}] (E) --
   node [text width=2.5cm,midway,below=0.1em,align=center ] {$k(g_1)$} (F);}
   
   {\draw[line width=1pt,->-={0.5}{latex}] (E) --
   node [text width=,midway,above=,align=center ] {} (B);}
  \end{scope}
  
\draw [->,thick] (5,1) to [out=-135,in=-5] (2.2,1); 
\node [right] at (5,1) {$h(\overline{g_1}) \cdot \ell \cdot f(g_1)$};
  
\end{tikzpicture} 
\end{center}

Notice we have added an extra edge dividing the original rectangle into two squares.  Also notice $h(\overline{g_1}) \cdot \ell \cdot f(g_1) = k(g_1) \cdot \ell \cdot g(\overline{g_1})$ because of the given relation.

Similarly one can divide the relation into four $2$-simplices as modeled below. 

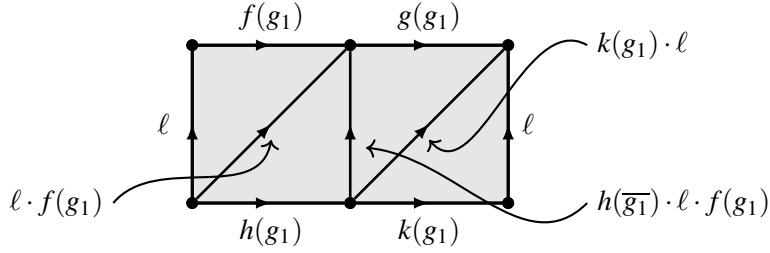
\begin{center}
\tikzset{->-/.style n args={2}{decoration={
  markings,
  mark=at position #1 with {\arrow[line width=1pt]{#2}}},postaction={decorate}}}
\begin{tikzpicture}[scale=1.05, bullet/.style={circle,inner sep=1.5pt,fill}]
 
 \begin{scope}
  \path
   (0,2) node[bullet,label=above:](A){}
   (2,2) node[bullet,label=below:](B){}
   (4,2) node[bullet,label=below:](C){}
   (0,0) node[bullet,label=below:](D){}
   (2,0) node[bullet,label=below:](E){}
   (4,0) node[bullet,label=below:](F){}
   
   ;
   {\draw[fill=black!10!,very thick] (0,0) rectangle (4,2);}
   {\draw [black,fill] (0,0) circle [radius=0.07];}
   {\draw [black,fill] (2,0) circle [radius=0.07];}
   {\draw [black,fill] (0,2) circle [radius=0.07];}
   {\draw [black,fill] (2,2) circle [radius=0.07];}
   {\draw [black,fill] (4,0) circle [radius=0.07];}
   {\draw [black,fill] (4,2) circle [radius=0.07];}
   {\draw[line width=1pt,->-={0.5}{latex}] (A) -- 
   node [text width=2.5cm,midway,above=0.1em,align=center ] {$f(g_1)$} (B);}
   {\draw[line width=1pt,->-={0.5}{latex}] (B) -- 
   node [text width=2.5cm,midway,above=0.1em,align=center ] {$g(g_1)$} (C);}
   {\draw[line width=1pt,->-={0.5}{latex}] (D) --
   node [text width=0.4cm,midway,left=0.1em,align=center ] {$\ell$} (A);}
   {\draw[line width=1pt,->-={0.5}{latex}] (F) -- 
   node [text width=0.1cm,midway,right=0.1em,align=center ] {$\ell$} (C);}
   {\draw[line width=1pt,->-={0.5}{latex}] (D) --
   node [text width=2.5cm,midway,below=0.1em,align=center ] {$h(g_1)$} (E);}
   {\draw[line width=1pt,->-={0.5}{latex}] (E) --
   node [text width=2.5cm,midway,below=0.1em,align=center ] {$k(g_1)$} (F);}
   
   {\draw[line width=1pt,->-={0.5}{latex}] (D) -- 
   node [text width=,midway,below=0.1em,align=center ] {} (B);}
   {\draw[line width=1pt,->-={0.5}{latex}] (E) --
   node [text width=,midway,above=,align=center ] {} (B);}
   {\draw[line width=1pt,->-={0.5}{latex}] (E) --
   node [text width=,midway,below=0.1em,align=center ] {} (C);}
   
  \end{scope}

\draw [->,thick] (-1,0) to [out=50,in=-100] (1,.8); 
\node [left] at (-1,0) {$\ell \cdot f(g_1)$};
  
\draw [->,thick] (5,0) to [out=-135,in=-5] (2.2,.7); 
\node [right] at (5,0) {$h(\overline{g_1}) \cdot \ell \cdot f(g_1)$};

\draw [->,thick] (5,2) to [out=150,in=-50] (3,0.8); 
\node [right] at (5,2) {$k(g_1) \cdot \ell$};

\end{tikzpicture} 
\end{center}
Considering orientation, this model provides a partial simplicial chain homotopy of dimension $1$ between the chain maps $Ed_{(f,g)}, Ed_{(h,k)} \colon  \Z BG_\ast \to \Z BH_\ast$. Recognizing each $2$-simplex in this model as a $2$-chain, We define $P_1$. For $g_1 \in G$, define 
\[
P([g_1]) := [\ell, f(g_1)] - [h(g_1), m(g_1)] + [m(g_1), g(g_1)] - [k(g_1), \ell]
\]
where $m \colon  G \to H$ is a function defined by $m(g_1) := h(\overline{g_1}) \cdot \ell \cdot f(g_1)$. Observe that $m$ is not a homomorphism and $m(g_1) = k(g_1) \cdot \ell \cdot g(\overline{g_1})$ because of the given relation. Then, we can easily check the equation below. \[(\d P_1 + P_0 \d)([g_1])=Ed_{(f,g)}([g_1]) - Ed_{(h,k)}([g_1]).\]
Thus, $\{P_0, P_1\}$ is a partial simplicial chain homotopy of dimension $1$ between the chain maps 
\[
Ed_{(f,g)}, Ed_{(h,k)} \colon  \Z BG_\ast \to \Z BH_\ast.
\]
Moreover, $d(0)=0$ and $d(1)=4$.



\begin{remark}\label{rmk:P1 as simplicial cylinder}
We can describe $P_1([g_1])$ in terms of a simplicial cylinder between $Ed_{(f,g)}([g_1])$ and $Ed_{(h,k)}([g_1])$ related with a system of pillars $T([g_1])=T_{Ed_{(h,k)}([g_1]}^{Ed_{(f,g)}([g_1])}$.

Denote by $\sigma$ and $\tau$ the chains $Ed_{(f,g)}([g_1])$ and $Ed_{(h,k)}([g_1])$ respectively. Let $\sigma_1$ and $\sigma_2$ be $[f(g_1)]$ and $[f(g_2)]$ respectively. Then, $Ed_{(f,g)}([g_1])=\Sigma_{i=1}^{2}\sigma_i$. Similarly, define $\tau_1$ and $\tau_2$ by $[h(g_1)]$ and $[k(g_2)]$ respectively. Then, $Ed_{(h,k)}([g_1])=\Sigma_{i=1}^{2}\tau_i$. Define a system of pillars $T([g_1])$ by $\{T_1, T_2 \}$ where the ordered sets $T_1=\{\ell, m(g_1)\}$ and $T_2=\{m(g_1), \ell\}$. Recall Definition~\ref{def:simplicial cylinder with pillars} with the model of $P([g_1])$. Then, one can check
\begin{align*}
P([g_1])&=Cyl(\Sigma_{i=1}^{2}\sigma_i,\tau_{i=1}^{2}\tau_i,T([g_1]))\\
        &=Cyl(Ed_{(f,g)}([g_1]),Ed_{(h,k)}([g_1]),T([g_1])).
\end{align*}
\end{remark}
We construct $P_2$ by assembling $P_1([g_1])$, $P_1([g_2])$, and $P_1([g_1 \cdot g_2])$.
Notice $P_1([g_1])$, $P_1([g_2])$, and $P_1([g_1 \cdot g_2])$ are the $2$-chains below:

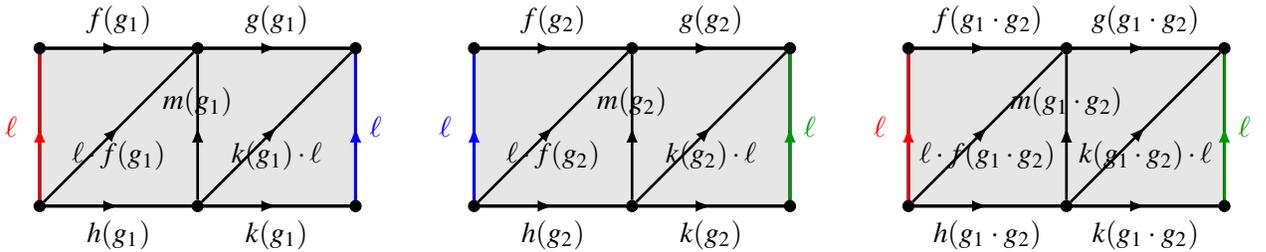
\begin{center}
\tikzset{->-/.style n args={2}{decoration={
  markings,
  mark=at position #1 with {\arrow[line width=1pt]{#2}}},postaction={decorate}}}
\begin{tikzpicture}[scale=1.05, bullet/.style={circle,inner sep=1.5pt,fill}]
 
 \begin{scope}
  \path
   (0,2) node[bullet,label=above:](A){}
   (2,2) node[bullet,label=below:](B){}
   (4,2) node[bullet,label=below:](C){}
   (0,0) node[bullet,label=below:](D){}
   (2,0) node[bullet,label=below:](E){}
   (4,0) node[bullet,label=below:](F){}
   
   ;
   {\draw[fill=black!10!,very thick] (0,0) rectangle (4,2);}
   {\draw [black,fill] (0,0) circle [radius=0.07];}
   {\draw [black,fill] (2,0) circle [radius=0.07];}
   {\draw [black,fill] (0,2) circle [radius=0.07];}
   {\draw [black,fill] (2,2) circle [radius=0.07];}
   {\draw [black,fill] (4,0) circle [radius=0.07];}
   {\draw [black,fill] (4,2) circle [radius=0.07];}
   {\draw[line width=1pt,->-={0.5}{latex}] (A) -- 
   node [text width=2.5cm,midway,above=0.1em,align=center ] {$f(g_1)$} (B);}
   {\draw[line width=1pt,->-={0.5}{latex}] (B) -- 
   node [text width=2.5cm,midway,above=0.1em,align=center ] {$g(g_1)$} (C);}
   {\draw[red, line width=1pt,->-={0.5}{latex}] (D) --
   node [text width=0.4cm,midway,left=0.1em,align=center ] {$\ell$} (A);}
   {\draw[blue, line width=1pt,->-={0.5}{latex}] (F) -- 
   node [text width=0.1cm,midway,right=0.1em,align=center ] {$\ell$} (C);}
   {\draw[line width=1pt,->-={0.5}{latex}] (D) --
   node [text width=2.5cm,midway,below=0.1em,align=center ] {$h(g_1)$} (E);}
   {\draw[line width=1pt,->-={0.5}{latex}] (E) --
   node [text width=2.5cm,midway,below=0.1em,align=center ] {$k(g_1)$} (F);}
   
   {\draw[line width=1pt,->-={0.5}{latex}] (D) -- 
   node [text width=,midway,below=0.1em,align=center ] {$\ell \cdot f(g_1)$} (B);}
   {\draw[line width=1pt,->-={0.5}{latex}] (E) --
   node [text width=,midway,above=,align=center ] {$m(g_1)$} (B);}
   {\draw[line width=1pt,->-={0.5}{latex}] (E) --
   node [text width=,midway,below=0.1em,align=center ] {$k(g_1) \cdot \ell$} (C);}
   
  \end{scope}

 \begin{scope}[xshift=5.5cm]
  \path
   (0,2) node[bullet,label=above:](A){}
   (2,2) node[bullet,label=below:](B){}
   (4,2) node[bullet,label=below:](C){}
   (0,0) node[bullet,label=below:](D){}
   (2,0) node[bullet,label=below:](E){}
   (4,0) node[bullet,label=below:](F){}
   
   ;
   {\draw[fill=black!10!,very thick] (0,0) rectangle (4,2);}
   {\draw [black,fill] (0,0) circle [radius=0.07];}
   {\draw [black,fill] (2,0) circle [radius=0.07];}
   {\draw [black,fill] (0,2) circle [radius=0.07];}
   {\draw [black,fill] (2,2) circle [radius=0.07];}
   {\draw [black,fill] (4,0) circle [radius=0.07];}
   {\draw [black,fill] (4,2) circle [radius=0.07];}
   {\draw[line width=1pt,->-={0.5}{latex}] (A) -- 
   node [text width=2.5cm,midway,above=0.1em,align=center ] {$f(g_2)$} (B);}
   {\draw[line width=1pt,->-={0.5}{latex}] (B) -- 
   node [text width=2.5cm,midway,above=0.1em,align=center ] {$g(g_2)$} (C);}
   {\draw[blue, line width=1pt,->-={0.5}{latex}] (D) --
   node [text width=0.4cm,midway,left=0.1em,align=center ] {$\ell$} (A);}
   {\draw[green!60!black, line width=1pt,->-={0.5}{latex}] (F) -- 
   node [text width=0.1cm,midway,right=0.1em,align=center ] {$\ell$} (C);}
   {\draw[line width=1pt,->-={0.5}{latex}] (D) --
   node [text width=2.5cm,midway,below=0.1em,align=center ] {$h(g_2)$} (E);}
   {\draw[line width=1pt,->-={0.5}{latex}] (E) --
   node [text width=2.5cm,midway,below=0.1em,align=center ] {$k(g_2)$} (F);}
   
   {\draw[line width=1pt,->-={0.5}{latex}] (D) -- 
   node [text width=,midway,below=0.1em,align=center ] {$\ell \cdot f(g_2)$} (B);}
   {\draw[line width=1pt,->-={0.5}{latex}] (E) --
   node [text width=,midway,above=,align=center ] {$m(g_2)$} (B);}
   {\draw[line width=1pt,->-={0.5}{latex}] (E) --
   node [text width=,midway,below=0.1em,align=center ] {$k(g_2) \cdot \ell$} (C);}
   
  \end{scope}

 \begin{scope}[xshift=11cm]  
  \path
   (0,2) node[bullet,label=above:](A){}
   (2,2) node[bullet,label=below:](B){}
   (4,2) node[bullet,label=below:](C){}
   (0,0) node[bullet,label=below:](D){}
   (2,0) node[bullet,label=below:](E){}
   (4,0) node[bullet,label=below:](F){}
   
   ;
   {\draw[fill=black!10!,very thick] (0,0) rectangle (4,2);}
   {\draw [black,fill] (0,0) circle [radius=0.07];}
   {\draw [black,fill] (2,0) circle [radius=0.07];}
   {\draw [black,fill] (0,2) circle [radius=0.07];}
   {\draw [black,fill] (2,2) circle [radius=0.07];}
   {\draw [black,fill] (4,0) circle [radius=0.07];}
   {\draw [black,fill] (4,2) circle [radius=0.07];}
   {\draw[line width=1pt,->-={0.5}{latex}] (A) -- 
   node [text width=2.5cm,midway,above=0.1em,align=center ] {$f(g_1 \cdot g_2)$} (B);}
   {\draw[line width=1pt,->-={0.5}{latex}] (B) -- 
   node [text width=2.5cm,midway,above=0.1em,align=center ] {$g(g_1 \cdot g_2)$} (C);}
   {\draw[red, line width=1pt,->-={0.5}{latex}] (D) --
   node [text width=0.4cm,midway,left=0.1em,align=center ] {$\ell$} (A);}
   {\draw[green!60!black, line width=1pt,->-={0.5}{latex}] (F) -- 
   node [text width=0.1cm,midway,right=0.1em,align=center ] {$\ell$} (C);}
   {\draw[line width=1pt,->-={0.5}{latex}] (D) --
   node [text width=2.5cm,midway,below=0.1em,align=center ] {$h(g_1 \cdot g_2)$} (E);}
   {\draw[line width=1pt,->-={0.5}{latex}] (E) --
   node [text width=2.5cm,midway,below=0.1em,align=center ] {$k(g_1 \cdot g_2)$} (F);}
   
   {\draw[line width=1pt,->-={0.5}{latex}] (D) -- 
   node [text width=,midway,below=0.1em,align=center ] {$\ell \cdot f(g_1 \cdot g_2)$} (B);}
   {\draw[line width=1pt,->-={0.5}{latex}] (E) --
   node [text width=,midway,above=,align=center ] {$m(g_1 \cdot g_2)$} (B);}
   {\draw[line width=1pt,->-={0.5}{latex}] (E) --
   node [text width=,midway,below=0.1em,align=center ] {$k(g_1 \cdot g_2) \cdot \ell$} (C);}
   
  \end{scope}
  
\end{tikzpicture} 
\end{center}
Since each chain has the same \emph{pillars of $\ell$}, we can assemble these chains along \emph{pillars}, carefully considering orientation. Then, we have the below $3$-chain in $\Z BH_\ast$.

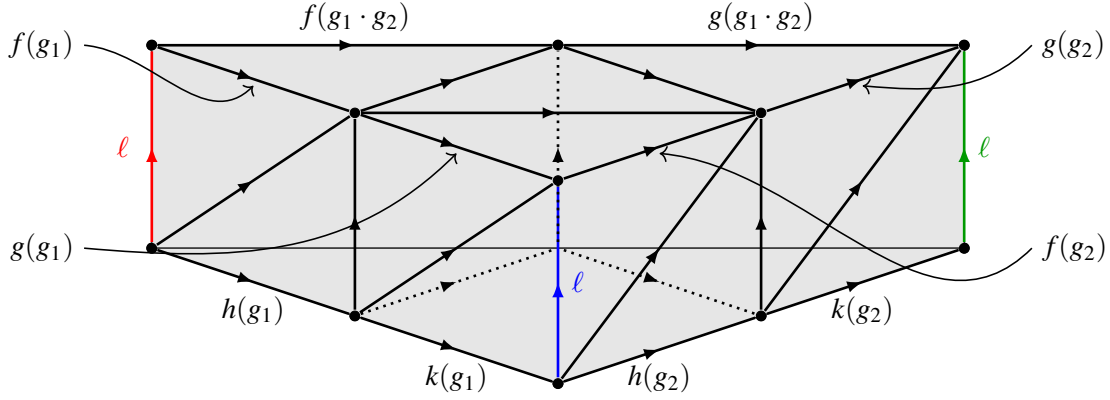
\begin{center}
\tikzset{->-/.style n args={2}{decoration={
  markings,
  mark=at position #1 with {\arrow[line width=1pt]{#2}}},postaction={decorate}}}
\begin{tikzpicture}[scale=0.9, bullet/.style={circle,inner sep=1.5pt,fill}]
 
 \begin{scope}
  \path
   (-6,5) node[bullet,label=above:](A){}
   (-3,4) node[bullet,label=below:](B){}
   (0,3) node[bullet,label=below:](C){}
   (-6,2) node[bullet,label=below:](D){}
   (-3,1) node[bullet,label=below:](E){}
   (0,0) node[bullet,label=below:](F){}
   
   (3,4) node[bullet,label=below:](G){}
   (6,5) node[bullet,label=below:](H){}
   (3,1) node[bullet,label=below:](I){}
   (6,2) node[bullet,label=below:](J){}
   
   (0,5) node[bullet,label=below:](K){}
   (0,2) node[bullet,label=below:](L){}
   ;
   
   {\draw[fill=black!10!] (-6,2) rectangle (6,5);}
   {\draw[fill=black!10!] (-6,2) -- (6,2) -- (0,0) -- cycle;}
   
   {\draw [black,fill] (-6,5) circle [radius=0.07];}
   {\draw [black,fill] (-3,4) circle [radius=0.07];}
   {\draw [black,fill] (0,3) circle [radius=0.07];}
   {\draw [black,fill] (-6,2) circle [radius=0.07];}
   {\draw [black,fill] (-3,1) circle [radius=0.07];}
   {\draw [black,fill] (0,0) circle [radius=0.07];}
   
   {\draw [black,fill] (3,4) circle [radius=0.07];}
   {\draw [black,fill] (6,5) circle [radius=0.07];}
   {\draw [black,fill] (3,1) circle [radius=0.07];}
   {\draw [black,fill] (6,2) circle [radius=0.07];}
   
   {\draw [black,fill] (0,5) circle [radius=0.07];}
   
   {\draw[line width=1pt,->-={0.5}{latex}] (A) -- 
   node [text width=2.5cm,midway,above=0.1em,align=center ] {} (B);}
   {\draw[line width=1pt,->-={0.5}{latex}] (B) -- 
   node [text width=2.5cm,midway,above=0.1em,align=center ] {} (C);}
   {\draw[red,line width=1pt,->-={0.5}{latex}] (D) --
   node [text width=0.4cm,midway,left=0.1em,align=center ] {$\ell$} (A);}
   {\draw[blue,line width=1pt,->-={0.5}{latex}] (F) -- 
   node [text width=0.1cm,midway,right=0.1em,align=center ] {$\ell$} (C);}
   {\draw[line width=1pt,->-={0.5}{latex}] (D) --
   node [text width=2.5cm,midway,below=0.1em,align=center ] {$h(g_1)$} (E);}
   {\draw[line width=1pt,->-={0.5}{latex}] (E) --
   node [text width=2.5cm,midway,below=0.1em,align=center ] {$k(g_1)$} (F);}

   {\draw[line width=1pt,->-={0.5}{latex}] (D) --
   node [text width=2.5cm,midway,above=0.1em,align=center ] {} (B);}
   {\draw[line width=1pt,->-={0.5}{latex}] (E) --
   node [text width=0.3cm,midway,below=0.5em,left=0.1em,align=center ] {} (B);}
   {\draw[line width=1pt,->-={0.5}{latex}] (E) --
   node [text width=2.5cm,midway,above=0.1em,align=center ] {} (C);}
   
   {\draw[line width=1pt,->-={0.5}{latex}] (C) -- 
   node [text width=2.5cm,midway,above=0.1em,align=center ] {} (G);}
   {\draw[line width=1pt,->-={0.5}{latex}] (G) -- 
   node [text width=2.5cm,midway,above=0.1em,align=center ] {} (H);}
   {\draw[green!60!black,line width=1pt,->-={0.5}{latex}] (J) -- 
   node [text width=0.1cm,midway,right=0.1em,align=center ] {$\ell$} (H);}
   {\draw[line width=1pt,->-={0.5}{latex}] (F) --
   node [text width=2.5cm,midway,below=0.1em,align=center ] {$h(g_2)$} (I);}
   {\draw[line width=1pt,->-={0.5}{latex}] (I) --
   node [text width=2.5cm,midway,below=0.1em,align=center ] {$k(g_2)$} (J);}

   {\draw[line width=1pt,->-={0.5}{latex}] (F) --
   node [text width=2.5cm,midway,above=0.1em,align=center ] {} (G);}
   {\draw[line width=1pt,->-={0.5}{latex}] (I) --
   node [text width=0.3cm,midway,below=0.5em,left=0.1em,align=center ] {} (G);}
   {\draw[line width=1pt,->-={0.5}{latex}] (I) --
   node [text width=2.5cm,midway,above=0.1em,align=center ] {} (H);}
   
   {\draw[line width=1pt,->-={0.5}{latex}] (A) --
   node [text width=2.5cm,midway,above=0.1em,align=center ] {$f(g_1 \cdot g_2)$} (K);}
   {\draw[line width=1pt,->-={0.5}{latex}] (K) --
   node [text width=,midway,above=0.1em,align=center ] {$g(g_1 \cdot g_2)$} (H);}
   {\draw[line width=1pt,->-={0.5}{latex}] (B) --
   node [text width=2.5cm,midway,above=0.05em,align=center ] {} (K);}  
   {\draw[line width=1pt,->-={0.5}{latex}] (K) --
   node [text width=2.5cm,midway,above=0.1em,align=center ] {} (G);}
   {\draw[line width=1pt,->-={0.5}{latex}] (B) --
   node [text width=0.3cm,midway,below=0.5em,above=0.1em,align=center ] {} (G);}

   {\draw[dotted,line width=1pt,->-={0.5}{latex}] (L) --
   node [text width=2.5cm,midway,above=0.1em,align=center ] {} (K);}
   {\draw[dotted,line width=1pt,->-={0.5}{latex}] (E) --
   node [text width=0.3cm,midway,below=0.02em,align=center ] {} (L);}
   {\draw[dotted,line width=1pt,->-={0.5}{latex}] (L) --
   node [text width=2.5cm,midway,below=0.02em,align=center ] {} (I);}   
  \end{scope} 
  
\draw [->, semithick] (-7,5) to [out=-5,in=-135] (-4.5,4.4); 
\node [left] at (-7,5) {$f(g_1)$};

\draw [->, semithick] (-7,2) to [out=-5,in=-135] (-1.5,3.4); 
\node [left] at (-7,2) {$g(g_1)$};

\draw [->, semithick] (7,5) to [out=-135,in=-5] (4.5,4.4); 
\node [right] at (7,5) {$g(g_2)$};

\draw [->, semithick] (7,2) to [out=-135,in=-5] (1.5,3.4); 
\node [right] at (7,2) {$f(g_2)$};
  
\end{tikzpicture} 
\end{center}

Based on our model, define, for $g_1, g_2 \in G$, the term $P_2([g_1, g_2])$ as follows:
\begin{align*}
P_2([g_1, g_2])&:= ([\ell, f(g_1), f(g_2)]-[h(g_1), m(g_2), f(g_2)]\\
&+[h(g_1), h(g_2), m(g_1 \cdot g_2)])-([m(g_1), f(g_2), g(g_1)]\\
&-([h(g_2), m(g_1 \cdot g_2), g(g_1)]+[h(g_2), k(g_1), m(g_2)])+([m(g_1), g(g_1), f(g_2)]\\
&-[k(g_1), \ell, f(g_2)]+[k(g_1), h(g_2), m(g_2)])+([m(g_1 \cdot g_2), g(g_1), g(g_2)]\\
&-[k(g_1), m(g_2), g(g_2)]+[k(g_1), k(g_2), \ell]).
\end{align*}
One can easily check that 
\[
(\d P_2 + P_1 \d)([g_1, g_2])=(Ed_{(f,g)}-Ed_{(h,k)})([g_1, g_2])
\]
for any $g_1, g_2 \in G$. 

Thus, $\{P_0, P_1, P_2\}$ is a partial simplicial chain homotopy of dimension $2$ between the chain maps 
\[
Ed_{(f,g)}, Ed_{(h,k)} : \Z BG_\ast \to \Z BH_\ast.
\]
Moreover, $d(0)=0$, $d(1)=4$, and, $d(2)=12$.

\begin{remark}\label{rmk:P2 as simplicial cylinder}
In the similar argument of Remark~\ref{rmk:P1 as simplicial cylinder}, we can describe $P_2([g_1,g_2])$ in terms of a simplicial cylinder between $Ed_{(f,g)}([g_1,g_2])$ and $Ed_{(h,k)}([g_1,g_2])$ related with a system of pillars $T_{Ed_{(h,k)}([g_1,g_2])}^{Ed_{(f,g)}([g_1,g_2])}$. For brevity, we say $T([g_1,g_2])=T_{Ed_{(h,k)}([g_1,g_2])}^{Ed_{(f,g)}([g_1,g_2])}$.

By, the equation~(\ref{eq:order}), notice
\[
Ed_{(f,g)}([g_1],[g_2])=[f(g_1), f(g_2)]-[f(g_2),g(g_1)]+[g(g_1), f(g_2)]+[g(g_1),g(g_2)]
\]
\[
Ed_{(h,k)}([g_1],[g_2])=[h(g_1),h(g_2)]-[h(g_2),k(g_1)]+[k(g_1),h(g_2)]+[k(g_1),g(k_2)].
\]
Define $T([g_1,g_2])$ by $\{T_1,T_2,T_3,T_4\}$ where
\begin{align*}
T_1&=\{\ell,m(g_1),m(g_1\cdot g_2)\},\\
T_2&=\{m(g_1),m(g_1\cdot g_2),m(g_2)\},\\
T_3&=\{m(g_1),\ell,m(g_2)\}, \\
T_4&=\{m(g_1\cdot g_2),m(g_2),\ell\}.
\end{align*}
Then, one can check
\[
P_2([g_1,g_2])=Cyl(Ed_{(f,g)}([g_1,g_2]),Ed_{(h,k)}([g_1,g_2]),T([g_1,g_2])).
\]
\end{remark}
Similarly, we can obtain $P_3$ from a schematic model for $P_3([g_1, g_2, g_3])$ given as follows.

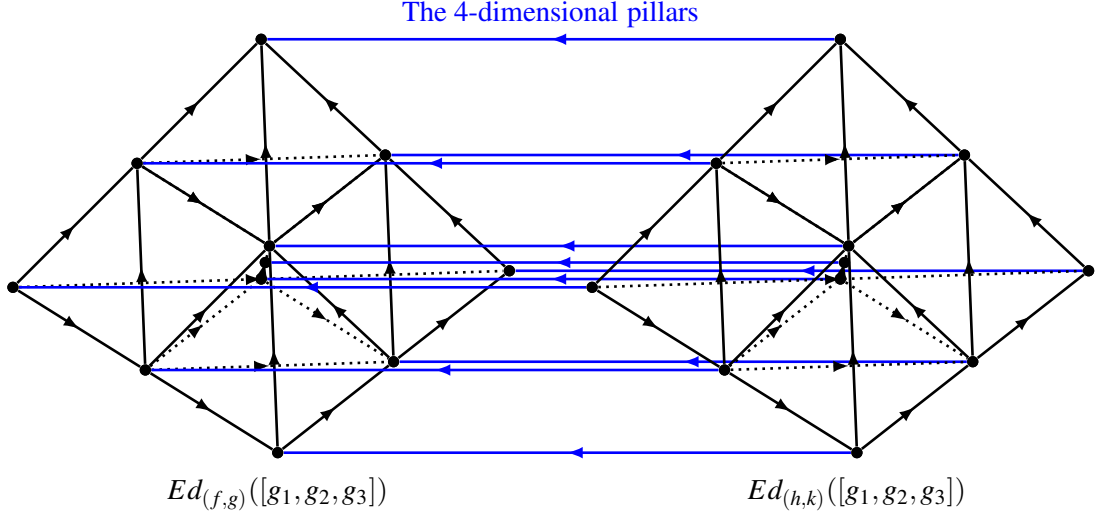
\begin{center}
\tikzset{->-/.style n args={2}{decoration={
  markings,
  mark=at position #1 with {\arrow[line width=1pt]{#2}}},postaction={decorate}}}
\begin{tikzpicture}[scale=1.1, bullet/.style={circle,inner sep=1.5pt,fill}]
 
 \begin{scope}
  \path
   (0,0) node[bullet,label=above:](A){}
   (3.2,-2) node[bullet,label=below:](B){}
   (6,0.2) node[bullet,label=below:](C){}
   (3,3) node[bullet,label=below:](D){}
   
   (1.6,-1) node[bullet,label=above:](E){}
   (4.6,-0.9) node[bullet,label=below:](F){}
   (3,0.1) node[bullet,label=below:](G){} 
   
   (1.5,1.5) node[bullet,label=above:](H){}
   (3.1,0.5) node[bullet,label=below:](I){}
   (4.5,1.6) node[bullet,label=below:](J){}   
   
   (3.05,0.3) node[bullet,label=below:](K){}   
   
   ;

   {\draw[line width=1pt,->-={0.5}{latex}] (A) -- 
   node [text width=2.5cm,midway,above=0.1em,align=center ] {} (E);}
   {\draw[line width=1pt,->-={0.5}{latex}] (E) -- 
   node [text width=2.5cm,midway,above=0.1em,align=center ] {} (B);}
 
   {\draw[line width=1pt,->-={0.5}{latex}] (B) -- 
   node [text width=2.5cm,midway,above=0.1em,align=center ] {} (F);}
   {\draw[line width=1pt,->-={0.5}{latex}] (F) -- 
   node [text width=2.5cm,midway,above=0.1em,align=center ] {} (C);}
 
   {\draw[line width=1pt,->-={0.5}{latex}] (C) -- 
   node [text width=2.5cm,midway,above=0.1em,align=center ] {} (J);}
   {\draw[line width=1pt,->-={0.5}{latex}] (J) -- 
   node [text width=2.5cm,midway,above=0.1em,align=center ] {} (D);}
 
   {\draw[line width=1pt,->-={0.5}{latex}] (A) -- 
   node [text width=2.5cm,midway,above=0.1em,align=center ] {} (H);}
   {\draw[line width=1pt,->-={0.5}{latex}] (H) -- 
   node [text width=2.5cm,midway,above=0.1em,align=center ] {} (D);}

   {\draw[line width=1pt,->-={0.5}{latex}] (B) -- 
   node [text width=2.5cm,midway,above=0.1em,align=center ] {} (I);}
   {\draw[line width=1pt,->-={0.5}{latex}] (I) -- 
   node [text width=2.5cm,midway,above=0.1em,align=center ] {} (D);}
   
   {\draw[line width=1pt,->-={0.5}{latex}] (E) -- 
   node [text width=2.5cm,midway,above=0.1em,align=center ] {} (H);}
   {\draw[line width=1pt,->-={0.5}{latex}] (E) -- 
   node [text width=2.5cm,midway,above=0.1em,align=center ] {} (I);}
   {\draw[line width=1pt,->-={0.5}{latex}] (H) -- 
   node [text width=2.5cm,midway,above=0.1em,align=center ] {} (I);}

   {\draw[line width=1pt,->-={0.5}{latex}] (F) -- 
   node [text width=2.5cm,midway,above=0.1em,align=center ] {} (I);}
   {\draw[line width=1pt,->-={0.5}{latex}] (F) -- 
   node [text width=2.5cm,midway,above=0.1em,align=center ] {} (J);}
   {\draw[line width=1pt,->-={0.5}{latex}] (I) -- 
   node [text width=2.5cm,midway,above=0.1em,align=center ] {} (J);}

   {\draw[dotted,line width=1pt,->-={0.5}{latex}] (A) --
   node [text width=2.5cm,midway,above=0.1em,align=center ] {} (C);}

   {\draw[dotted,line width=1pt,->-={0.5}{latex}] (E) --
   node [text width=2.5cm,midway,above=0.1em,align=center ] {} (F);}
   {\draw[dotted,line width=1pt,->-={0.5}{latex}] (E) --
   node [text width=2.5cm,midway,above=0.1em,align=center ] {} (G);}
   {\draw[dotted,line width=1pt,->-={0.5}{latex}] (G) --
   node [text width=2.5cm,midway,above=0.1em,align=center ] {} (F);}

   {\draw[dotted,line width=1pt,->-={0.5}{latex}] (H) --
   node [text width=2.5cm,midway,above=0.1em,align=center ] {} (I);}
   {\draw[dotted,line width=1pt,->-={0.5}{latex}] (H) --
   node [text width=2.5cm,midway,above=0.1em,align=center ] {} (J);}
   {\draw[dotted,line width=1pt,->-={0.5}{latex}] (I) --
   node [text width=2.5cm,midway,above=0.1em,align=center ] {} (J);}

   {\draw[dotted,line width=1pt,->-={0.5}{latex}] (G) --
   node [text width=2.5cm,midway,above=0.1em,align=center ] {} (I);}

  \end{scope}
  
   \begin{scope}
  \path
   (0,0) node[bullet,label=below:](A){}
   (3.2,-2) node[bullet,label=below:](B){}
   (6,0.2) node[bullet,label=below:](C){}
   (3,3) node[bullet,label=below:](D){}
   
   (1.6,-1) node[bullet,label=above:](E){}
   (4.6,-0.9) node[bullet,label=below:](F){}
   (3,0.1) node[bullet,label=below:](G){} 
   
   (1.5,1.5) node[bullet,label=above:](H){}
   (3.1,0.5) node[bullet,label=below:](I){}
   (4.5,1.6) node[bullet,label=below:](J){}   
   
   (3.05,0.3) node[bullet,label=below:](K){}

   (7,0) node[bullet,label=above:](L){}
   (10.2,-2) node[bullet,label=below:](M){}
   (13,0.2) node[bullet,label=below:](N){}
   (10,3) node[bullet,label=below:](O){}
   
   (8.6,-1) node[bullet,label=above:](P){}
   (11.6,-0.9) node[bullet,label=below:](Q){}
   (10,0.1) node[bullet,label=below:](R){} 
   
   (8.5,1.5) node[bullet,label=above:](S){}
   (10.1,0.5) node[bullet,label=below:](T){}
   (11.5,1.6) node[bullet,label=below:](U){}   
   
   (10.05,0.3) node[bullet,label=below:](W){}   
   
   (6.5,3.3) node[blue,label=below:](Z){The $4$-dimensional pillars}
   
   (3.2,-2.5) node[black,label=below:](Y){$Ed_{(f,g)}([g_1,g_2,g_3])$}

   ;
   
   {\draw[blue,line width=1pt,->-={0.5}{latex}] (L) -- 
   node [text width=2.5cm,midway,above=0.1em,align=center ] {} (A);}
   {\draw[blue,line width=1pt,->-={0.5}{latex}] (M) -- 
   node [text width=2.5cm,midway,above=0.1em,align=center ] {} (B);}
   {\draw[blue,line width=1pt,->-={0.5}{latex}] (N) -- 
   node [text width=2.5cm,midway,above=0.1em,align=center ] {} (C);}
   {\draw[blue,line width=1pt,->-={0.5}{latex}] (O) -- 
   node [text width=2.5cm,midway,above=0.1em,align=center ] {} (D);}
 
   {\draw[blue,line width=1pt,->-={0.5}{latex}] (P) -- 
   node [text width=2.5cm,midway,above=0.1em,align=center ] {} (E);}
   {\draw[blue,line width=1pt,->-={0.5}{latex}] (Q) -- 
   node [text width=2.5cm,midway,above=0.1em,align=center ] {} (F);}
   {\draw[blue,line width=1pt,->-={0.5}{latex}] (R) -- 
   node [text width=2.5cm,midway,above=0.1em,align=center ] {} (G);}
 
   {\draw[blue,line width=1pt,->-={0.5}{latex}] (S) -- 
   node [text width=2.5cm,midway,above=0.1em,align=center ] {} (H);}
   {\draw[blue,line width=1pt,->-={0.5}{latex}] (T) -- 
   node [text width=2.5cm,midway,above=0.1em,align=center ] {} (I);}
   {\draw[blue,line width=1pt,->-={0.5}{latex}] (U) -- 
   node [text width=2.5cm,midway,above=0.1em,align=center ] {} (J);}
   
   {\draw[blue,line width=1pt,->-={0.5}{latex}] (W) -- 
   node [text width=2.5cm,midway,above=0.1em,align=center ] {} (K);}

  \end{scope}
  
   \begin{scope}[xshift=7cm]
  \path
   (0,0) node[bullet,label=above:](A){}
   (3.2,-2) node[bullet,label=below:](B){}
   (6,0.2) node[bullet,label=below:](C){}
   (3,3) node[bullet,label=below:](D){}
   
   (1.6,-1) node[bullet,label=above:](E){}
   (4.6,-0.9) node[bullet,label=below:](F){}
   (3,0.1) node[bullet,label=below:](G){} 
   
   (1.5,1.5) node[bullet,label=above:](H){}
   (3.1,0.5) node[bullet,label=below:](I){}
   (4.5,1.6) node[bullet,label=below:](J){}   
   
   (3.05,0.3) node[bullet,label=below:](K){}   
   
   (3.2,-2.5) node[black,label=below:](Y){$Ed_{(h,k)}([g_1,g_2,g_3])$}
   
   ;

   {\draw[line width=1pt,->-={0.5}{latex}] (A) -- 
   node [text width=2.5cm,midway,above=0.1em,align=center ] {} (E);}
   {\draw[line width=1pt,->-={0.5}{latex}] (E) -- 
   node [text width=2.5cm,midway,above=0.1em,align=center ] {} (B);}
 
   {\draw[line width=1pt,->-={0.5}{latex}] (B) -- 
   node [text width=2.5cm,midway,above=0.1em,align=center ] {} (F);}
   {\draw[line width=1pt,->-={0.5}{latex}] (F) -- 
   node [text width=2.5cm,midway,above=0.1em,align=center ] {} (C);}
 
   {\draw[line width=1pt,->-={0.5}{latex}] (C) -- 
   node [text width=2.5cm,midway,above=0.1em,align=center ] {} (J);}
   {\draw[line width=1pt,->-={0.5}{latex}] (J) -- 
   node [text width=2.5cm,midway,above=0.1em,align=center ] {} (D);}
 
   {\draw[line width=1pt,->-={0.5}{latex}] (A) -- 
   node [text width=2.5cm,midway,above=0.1em,align=center ] {} (H);}
   {\draw[line width=1pt,->-={0.5}{latex}] (H) -- 
   node [text width=2.5cm,midway,above=0.1em,align=center ] {} (D);}

   {\draw[line width=1pt,->-={0.5}{latex}] (B) -- 
   node [text width=2.5cm,midway,above=0.1em,align=center ] {} (I);}
   {\draw[line width=1pt,->-={0.5}{latex}] (I) -- 
   node [text width=2.5cm,midway,above=0.1em,align=center ] {} (D);}
   
   {\draw[line width=1pt,->-={0.5}{latex}] (E) -- 
   node [text width=2.5cm,midway,above=0.1em,align=center ] {} (H);}
   {\draw[line width=1pt,->-={0.5}{latex}] (E) -- 
   node [text width=2.5cm,midway,above=0.1em,align=center ] {} (I);}
   {\draw[line width=1pt,->-={0.5}{latex}] (H) -- 
   node [text width=2.5cm,midway,above=0.1em,align=center ] {} (I);}

   {\draw[line width=1pt,->-={0.5}{latex}] (F) -- 
   node [text width=2.5cm,midway,above=0.1em,align=center ] {} (I);}
   {\draw[line width=1pt,->-={0.5}{latex}] (F) -- 
   node [text width=2.5cm,midway,above=0.1em,align=center ] {} (J);}
   {\draw[line width=1pt,->-={0.5}{latex}] (I) -- 
   node [text width=2.5cm,midway,above=0.1em,align=center ] {} (J);}

   {\draw[dotted,line width=1pt,->-={0.5}{latex}] (A) --
   node [text width=2.5cm,midway,above=0.1em,align=center ] {} (C);}

   {\draw[dotted,line width=1pt,->-={0.5}{latex}] (E) --
   node [text width=2.5cm,midway,above=0.1em,align=center ] {} (F);}
   {\draw[dotted,line width=1pt,->-={0.5}{latex}] (E) --
   node [text width=2.5cm,midway,above=0.1em,align=center ] {} (G);}
   {\draw[dotted,line width=1pt,->-={0.5}{latex}] (G) --
   node [text width=2.5cm,midway,above=0.1em,align=center ] {} (F);}

   {\draw[dotted,line width=1pt,->-={0.5}{latex}] (H) --
   node [text width=2.5cm,midway,above=0.1em,align=center ] {} (I);}
   {\draw[dotted,line width=1pt,->-={0.5}{latex}] (H) --
   node [text width=2.5cm,midway,above=0.1em,align=center ] {} (J);}
   {\draw[dotted,line width=1pt,->-={0.5}{latex}] (I) --
   node [text width=2.5cm,midway,above=0.1em,align=center ] {} (J);}

   {\draw[dotted,line width=1pt,->-={0.5}{latex}] (G) --
   node [text width=2.5cm,midway,above=0.1em,align=center ] {} (I);}

  \end{scope} `

\end{tikzpicture} 
\end{center}

By constructing $P_1, P_2,$ and $P_3$, we obtain a partial simplicial chain homotopy of dimension $3$ between the edgewise subdivisions which are chain maps:
\[
Ed_{(f,g)}, Ed_{(h,k)} \colon  \Z BG_\ast \to \Z BH_\ast
\]
The idea was to split the simplicial cylinder between the edgewise subdivisions into several simplices {\em arrangeable} in some sense.

Based on what we observed in Remark~\ref{rmk:P1 as simplicial cylinder} and Remark~\ref{rmk:P2 as simplicial cylinder}, we construct $P_n$ for any dimension $n \in \N \cup \{0\}$ as a simplicial cylinder between two chains related with a system of pillars in order to obtain a simplicial chain homotopy between $Ed_{(f,g)}$ and $Ed_{(h,k)}$.

As seen in Remark~\ref{rmk:P1 as simplicial cylinder} and Remark~\ref{rmk:P2 as simplicial cylinder}, the top and bottom are $Ed_{(f,g)}(\sigma)$ and $Ed_{(h,k)}(\sigma)$ respectively. So, what we need is a system of pillars $T_{Ed_{(h,k)}(\sigma)}^{Ed_{(f,g)}(\sigma)}$, briefly say $T$.

As seen as in Definition~\ref{def:systemofpillars}, $T$ is a collection of ordered sets of group elements in $H$. Since, for an $n$-simplex $\sigma$, the $n$-chain $Ed_{(f,g)}(\sigma)$ has a diameter of $\Sigma_{i=0}^{n} \Sigma_{j=1}^{\binom{n}{i}}1$ by the equation~(\ref{eq:order}), we need to construct $T$ of $\Sigma_{i=0}^{n} \Sigma_{j=1}^{\binom{n}{i}}1$ ordered sets. Since $Ed_{(f,g)}(\sigma)$ and $Ed_{(h,k)}(\sigma)$ are ordered in terms of the dictionary order of $(p,q)$-shuffles (recall~\ref{def:pq shuffle}), we assign the corresponding order to each ordered set $T_{j}^{p,q} \in T$.

As an ingredient, for $\sigma = [g_1,\cdots,g_n]$ we define the ordered set $T_{j}^{p,q}(\sigma)$ with respect to $f,g$. Corresponding to ${S_{(f,g)}}_{j}^{p,q}(\sigma)$, $T_{j}^{p,q}(\sigma)$ is defined as the ordered set of $n+1$ elements of the group $H$ satisfying the three conditions below:

(1) The first element of $T_{j}^{p,q}(\sigma)$ is $m(g_1 \cdots g_p)$. If $p=0$, then $\ell$ is the first element. (See Remark~\ref{rmk:P2 as simplicial cylinder} as an example.)

(2) If $m(g)$ is the $(s-1)$-th element of $T_{j}^{p,q}(\sigma)$ and $f(g_k)$ is the $s$-th component of ${S_{(f,g)}}_{j}^{p,q}(\sigma)$, then the $s$-th element of $T_{j}^{p,q}(\sigma)$ is $m(g \cdot g_k)$.

(3) If $m(g_k \cdot g_{k+1} \cdots g_q)$ is the $(s-1)$-th element of $T_{j}^{p,q}(\sigma)$ and $g(g_k)$ is the $(s-1)$-th component of ${S_{(f,g)}}_{j}^{p,q}(\sigma)$, then the $s$-th element of $T_{j}^{p,q}(\sigma)$ is $m(g_{k+1} \cdot g_{k+2} \cdots g_q)$.

We provide some examples for the reader. For 
\[
\sigma = [g_1,g_2,g_3],
\]
 we have 
 \[
 {S_{(f,g)}}_{1}^{0,3}(\sigma)=[f(g_1),f(g_2),f(g_3)].
 \]
 Thus, 
\[
T_{1}^{0,3}(\sigma)=\{\ell,m(g_1),m(g_1 g_2),m(g_1 g_2 g_3)\}.
\]
Since 
\[
{S_{(f,g)}}_{2}^{1,2}(\sigma)=[f(g_2),g(g_1),f(g_3)]
\]
we have 
\[
T_{2}^{1,2}(\sigma)=\{m(g_1),m(g_1 g_2),m(g_2),m(g_2 g_3)\}.
\]  

\begin{remark}\label{rmk:swap}
Observe that $d_k T_j^{p,q}(\sigma)$ satisfies the above rule with respect to $d_k {S_{(f,g)}}_j^{p,q}(\sigma)$. Moreover notice that if $T(\sigma)$ and $T'(\sigma)$ satisfy the above rules with respect to a simplex, then $T(\sigma) = T'(\sigma)$.
\end{remark}

Notice 
\[
Cyl( {S_{(f,g)}}_{j}^{i,n-i} \sigma, {S_{(h,k)}}_{j}^{i,n-i} \sigma, T_{j}^{i,n-i}(\sigma))
\]
is well-defined. To show this, we should check below equation 
\[
t_r \cdot \proj_{r}({S_{(f,g)}}_{j}^{i,n-i})=\proj_{r+1}({S_{(h,k)}}_{j}^{i,n-i}) \cdot t_{(r+1)}
\]
for $r \in \{0,1,2,\cdots,n-1\}$ 
9where $T_j^{i,n-i}(\sigma)=[t_0,\ldots, t_n]$. 

\textbf{Case 1 : $\proj_{r}({S_{(f,g)}}_{j}^{i,n-i}) = f(g_k)$}

We may assume $t_r=m(g)$ for some $g \in G$.
\begin{align*}
m(g) \cdot f(g_k)  &=h(\overline{g}) \cdot \ell \cdot f(g) \cdot f(g_k)\\
                   &=h(g_k \cdot \overline{g_k} \cdot \overline{g}) \cdot \ell \cdot f(g \cdot g_k)\\
                   &=h(g_k) \cdot h(\overline{g \cdot g_k}) \cdot \ell \cdot f(g \cdot g_k)\\
                   &=h(g_k) \cdot m(g \cdot g_k).
\end{align*}

\textbf{Case 2 : $\proj_{r}({S_{(f,g)}}_{j}^{i,n-i}) = g(g_k)$}

Since $g(g_1), \cdots, g(g_i)$ are ascending in ${S_{(f,g)}}_{j}^{i,n-i}$, we may assume $t_r = m(g_k \cdot g)$ for some $g \in G$.
\begin{align*}
m(g_k \cdot g) \cdot g(g_k) &=k(g_k \cdot g) \cdot \ell \cdot g(\overline{g_k \cdot g}) \cdot g(g_k)\\
                            &=k(g_k) \cdot k(g) \cdot \ell \cdot g(\overline{g})\\
                            &=k(g_k) \cdot m(g)
\end{align*}
Thus, $Cyl( {S_{(f,g)}}_{j}^{i,n-i} \sigma, {S_{(h,k)}}_{j}^{i,n-i} \sigma, T_{j}^{i,n-i})$ is well-defined.

We define 
\[
T_{h,k}^{f,g}(\sigma) := \{T_j^{i,n-i}(\sigma) : i=0,\cdots,n, j=1,\cdots,\binom{n}{i}\}.
\]
We briefly say $T(\sigma) = T_{h,k}^{f,g}(\sigma)$. Then, by the above well-definedness and Definition~\ref{def:simplicial cylinder with pillars}, $T(\sigma)$ is a system of pillars between $Ed_{(f,g)}(\sigma)$ and $Ed_{(h,k)}(\sigma)$.

We define a simplicial chain homotopy $P_{h,k}^{f,g}$ between $Ed_{(f,g)}$ and $Ed_{(h,k)}$. We briefly say $P = P_{h,k}^{f,g}$. For a simplex $\sigma = [g_1,\cdots,g_n]$, define $P(\sigma)$ by a simplicial cylinder between $Ed_{(f,g)}(\sigma)$ and $Ed_{(h,k)}(\sigma)$. 
\begin{align*}
P(\sigma) &:= Cyl(Ed_{(f,g)}(\sigma), Ed_{(h,k)}(\sigma), T(\sigma))\\
          &\phantom{:}= \Sigma_{i=0}^{n} \Sigma_{j=1}^{\binom{n}{i}} \sign(S_{j}^{i,n-i}) Cyl( {S_{(f,g)}}_{j}^{i,n-i} \sigma, {S_{(h,k)}}_{j}^{i,n-i} \sigma, T_{j}^{i,n-i}(\sigma)).
\end{align*}

Before we show $(\d P + P \d)(\sigma) = (Ed_{(f,g)}-Ed_{(h,k)})(\sigma)$, we need a lemma.

\begin{lemma}\label{lem:cancelingprop}
If 
\[
\sign(S_j^{(i,n-i)})(-1)^{k} d_k ({S_{(f,g)}}_{j}^{i,n-i} \sigma) = -\sign(S_{j'}^{(i',n-i')})(-1)^{k'} d_{k'} ({S_{(f,g)}}_{j'}^{i',n-i'} \sigma)
\]
for any $\sigma=[g_1,\cdots,g_n]$, then
\[
d_{k} T_{j}^{i,n-i}(\sigma) = d_{k'} T_{j'}^{i',n-i'}(\sigma).
\]
\end{lemma}

\begin{proof}
In this proof, we abbreviate notation as follows:
\tabto{1in}$T_{j}^{i,n-i} = T_{j}^{i,n-i}(\sigma)$\\
and\\
\tabto{1in}$d_{k} T_{j}^{i,n-i} = d_{k} T_{j}^{i,n-i}(\sigma)$.\\ 
We first observe two facts below.\\
If\\
\tabto{1in}$
\sign(S_{j}^{i,n-i})(-1)^{k} d_k ({S_{(f,g)}}_{j}^{i,n-i} \sigma) = - \sign(S_{j'}^{i',n-i'})(-1)^{k'} d_{k'} ({S_{(f,g)}}_{j'}^{i',n-i'} \sigma)
,$\\
then\\
\tabto{1in}$
d_k ({S_{(f,g)}}_{j}^{i,n-i} \sigma) = d_{k'} ({S_{(f,g)}}_{j'}^{i',n-i'} \sigma).
$\\
If\\
\tabto{1in}$
d_k ({S_{(f,g)}}_{j}^{i,n-i} \sigma) = d_{k'} ({S_{(f,g)}}_{j'}^{i',n-i'} \sigma)
$\\
and\\
\tabto{1in}the first elements of $d_{k} T_{j}^{i,n-i}$ and $d_{k'} T_{j'}^{i',n-i'}$ are the same,
\\
then\\ 
\tabto{1in}$
d_{k} T_{j}^{i,n-i} = d_{k'} T_{j'}^{i',n-i'}.
$

We may assume $\sigma=[g_1,\cdots,g_n]$ where $g_i$'s are different. Without loss of generality, we may assume $i \leq i'$. Then, there are three cases.

\textbf{Case 1 : $i+2 \leq i'$}

There is no pair of $k, k'$ such that $
d_k ({S_{(f,g)}}_{j}^{i,n-i} \sigma) = d_{k'} ({S_{(f,g)}}_{j'}^{i',n-i'} \sigma)
$
for some $j, j'$ because for any $k'$, there is a component of $d_{k'}({S_{(f,g)}}_{j'}^{i',n-i'} \sigma)$ which is $g(g_{j'-1})$ or $g(g_{j'})$ that both are not in ${S_{(f,g)}}_{j}^{i,n-i} \sigma$.

\textbf{Case 2 : $i+1 = i'$}

If\\
\tabto{1in}$
d_k ({S_{(f,g)}}_{j}^{i,n-i} \sigma) = d_{k'} ({S_{(f,g)}}_{j'}^{i',n-i'} \sigma),
$\\
then\\
\tabto{1in}$
\proj_{1}({S_{(f,g)}}_{j}^{i,n-i} \sigma)=f(g_{i+1})
$\\
and\\
\tabto{1in}$  
\proj_{n}({S_{(f,g)}}_{j'}^{i',n-i'} \sigma)=g(g_{i+1}), \text{ for } k=0 \text{ and } $k'=n$
$\\	
since ${S_{(f,g)}}_{j}^{i,n-i} \sigma$ does not include $g(g_{i+1})$ and ${S_{(f,g)}}_{j'}^{i',n-i'} \sigma$ does not include $f(g_{i+1})$. 

If\\
\tabto{1in}$\sign(S_j^{i,n-i})(-1)^{0} d_0 ({S_{(f,g)}}_{j}^{i,n-i} \sigma) =$ $\sign(S_{j'}^{i',n-i'})(-1)^{n} d_{n} ({S_{(f,g)}}_{j'}^{i',n-i'} \sigma)$\\
for some $j,j',$\\
then\\
\tabto{1in}$d_0 ({S_{(f,g)}}_{j}^{i,n-i} \sigma)$\\
and\\
\tabto{1in}$d_{n} ({S_{(f,g)}}_{j'}^{i',n-i'} \sigma)$\\
have the same components in the same order. Note that the first element of ${T_{(f,g)}}_{j}^{i,n-i}$ is $m(g_1 \cdots g_i)$. In other words, $t_0 = m(g_1 \cdots g_i).$\\
Since\\
\tabto{1in}$\proj_{1}({S_{(f,g)}}_{j}^{i,n-i} \sigma)=f(g_{i+1}),$\\
we obtain\\
\tabto{1in}$t_1 = m(g_1 \cdots g_{i+1})$.\\
Thus, $m(g_1 \cdots g_{i+1})$ is the first elements of $d_0 T_j^{i,n-i}.$\\
Since\\
\tabto{1in}$i' = i+1,$\\
then\\
\tabto{1in}${T_{(f,g)}}_{j'}^{i',n-i'}$ has the first element ${t'}_0 = m(g(g_1) \cdots g(g_{i+1}))$.\\
So, $m(g(g_1) \cdots g(g_{i+1}))$ is the first element of $d_{n} T_{j'}^{i',n-i'}$. By above observation, $d_{k} T_{j}^{i,n-i} = d_{k'} T_{j'}^{i',n-i'}$.

\textbf{Case 3 : $i=i'$}

Since ${S_{(f,g)}}_{j}^{i,n-i} \sigma$ and ${S_{(f,g)}}_{j'}^{i',n-i'} \sigma$ have the same components (For example, we say\\ ${S_{(f,g)}}_{1}^{1,2} \sigma=[g(g_1),f(g_2),f(g_3)]$ and ${S_{(f,g)}}_{2}^{1,2} \sigma=[f(g_2),g(g_1),f(g_3)]$ have the same components of $g(g_1),f(g_2),$ and $f(g_3)$),
Thus, 
\[
d_k ({S_{(f,g)}}_{j}^{i,n-i} \sigma) = d_{k'} ({S_{(f,g)}}_{j'}^{i',n-i'} \sigma)
\]
is possible only if for some $1 \leq k=k'< n,$
\[
\quad {S_{(f,g)}}_{j}^{i,n-i} \sigma = [\textcolor{red}{\cdots}, f(g_s), g(g_t), \textcolor{blue}{\cdots}]
\]
where $f(g_s)$ is the $k$-th component and
\[
{S_{(f,g)}}_{j'}^{i',n-i'} \sigma = [\textcolor{red}{\cdots}, g(g_t), f(g_s), \textcolor{blue}{\cdots}]
\]
where $g(g_t)$ is the $k'(=k)$-th component. Without the loss of generality, we may consider only the above order of $f(g_s)$ and $g(g_t)$. In this case, for $1 \leq k = k' < n$, both $d_{k} T_{j}^{i,n-i}$ and $d_{k'} T_{j'}^{i',n-i'}$ have the same first element $t_0 = {t'}_0 = m(g_1 \cdots g_i)$. Moreover $d_k ({S_{(f,g)}}_{j}^{i,n-i} \sigma)$ and $d_{k'} ({S_{(f,g)}}_{j'}^{i',n-i'} \sigma)$ have the same elements in the same order. Thus, $d_{k} T_{j}^{i,n-i} = d_{k'} T_{j'}^{i',n-i'}$.

Thus, if 
\[
\sign(S_{j}^{i,n-i})(-1)^{k} d_k ({S_{(f,g)}}_{j}^{i,n-i} \sigma) = - \sign(S_{j'}^{i',n-i'})(-1)^{k'} d_{k'} ({S_{(f,g)}}_{j'}^{i',n-i'} \sigma)
\]
for any $\sigma=[g_1,\cdots,g_n]$, then
\[
d_{k} T_{j}^{i,n-i} = d_{k'} T_{j'}^{i',n-i'}.
\]
\end{proof}

Finally, we show $(\d P + P \d)(\sigma) = (Ed_{(f,g)}-Ed_{(h,k)})(\sigma)$.
\begin{align}
\d P(\sigma)   &=\d Cyl(Ed_{(f,g)}(\sigma), Ed_{(h,k)}(\sigma), T(\sigma)) \label{eq:A}\\
               &=Ed_{(f,g)}(\sigma) - Ed_{(h,k)}(\sigma) - Cyl(\d Ed_{(f,g)}(\sigma), \d Ed_{(h,k)}(\sigma), \d T(\sigma)) \label{eq:B}\\
               &=Ed_{(f,g)}(\sigma) - Ed_{(h,k)}(\sigma) - Cyl(Ed_{(f,g)}(\d \sigma), Ed_{(h,k)}(\d \sigma), \d T(\sigma)) \label{eq:C}\\
               &=Ed_{(f,g)}(\sigma) - Ed_{(h,k)}(\sigma) - Cyl(Ed_{(f,g)}(\d \sigma), Ed_{(h,k)}(\d \sigma), T(\d \sigma)) \label{eq:D}\\
               &=Ed_{(f,g)}(\sigma) - Ed_{(h,k)}(\sigma) - P \d(\sigma) \label{eq:E}.
\end{align}
Equation~(\ref{eq:A}) is the definition of $P$. Equation~(\ref{eq:B}) is from Lemma~\ref{lem:general canceling} and Definition~\ref{def:bdsys}. Equation~(\ref{eq:C}) is from that $Ed$ is a chain map and Lemma~\ref{lem:cancelingprop}. Equation~(\ref{eq:D}) is from Remark~\ref{rmk:swap}, and Equation~(\ref{eq:E}) is from the definition of $P$. Thus, $(\d P + P \d)(\sigma) = (Ed_{(f,g)}-Ed_{(h,k)})(\sigma)$. Furthermore, since the sum of number of $(p,q)$-shuffles is $2^n$ for $p+q=n$ and $d(Cyl(\sigma))=n+1$ for an $n$-simplex $\sigma$, $d(P(\sigma)) = 2^n \cdot (n+1)$. This completes the proof of Theorem~\ref{thm:edgewise chain homotopy}.
\end{proof}

\begin{remark}
Concerning the above computation, in step of Equation~(\ref{eq:D}), readers are warned that $\d T(\sigma) \neq T (\d \sigma)$. Equation~(\ref{eq:D}) is from the fact that, for any simplex $\tau$ of the chain $Ed_{(f,g)}(\d \sigma)$, the ordered sets ${\d T(\sigma)}_{\tau}$ and ${T(\d \sigma)}_{\tau}$ associated $\tau$ are the same by Remark~\ref{rmk:swap}.
\end{remark}

\subsection{A controlled chain homotopy into the BDH-acyclic container}

We apply Theorem~\ref{thm:edgewise chain homotopy} to the mitosis embeddings of Definition~\ref{def:mitosis}.

\begin{theorem}\label{thm:mitosis edgewise chain homotopy}
For a group $G$ and a $n \in \N\cup\{0\}$, let $\overline{u_n}, e,$ and $id$ be homomorphisms from $G$ to $\A^n(G)$ such that $\overline{u_n}(g):=g^{\overline{u_n}}$, $e(g):=e$, and $id(g)=g$. Then there is a simplicial chain homotopy $P_{\overline{u_n},e}^{\overline{u_n},id}$ between the chain maps $Ed_{(\overline{u_n},id)}, Ed_{(\overline{u_n},e)} : \Z BG_\ast \to \Z B\A^n(G)_\ast$, whose diameter function is exactly $d(m)$. For $m \leq 7$, the value of $d(m)$ is as follows:

\begin{center}
 \begin{tabular}{c c c c c c c c c} 
 \hline
 m & 0 & 1 & 2 & 3 & 4 & 5 & 6 & 7\\ [0.5ex]
 $d(m)$ & 0 & 4 & 12 & 32 & 80 & 192 & 448 & 1024\\ 
 \hline
 \end{tabular}
\end{center}

Actually, there is the recurrence formula of $d(m)$ as follows:

\begin{center}
    $d(m)=2^{m} \cdot (m+1)$.
\end{center}

\end{theorem}

\begin{proof}
By the group relation of the mitosis group, for $\ell={\overline{t_n}}^{\overline{u_n}} \in \A(G)$, $\ell \cdot \overline{u_n}(g) \cdot id(g) = \overline{u_n}(g) \cdot e(g) \cdot \ell$ for any $g \in G$. By setting $f=h=\overline{u_n}, g=e,$ and $k=id$, trivial by Theorem~\ref{thm:edgewise chain homotopy}.
\end{proof}

\begin{theorem}\label{thm:induction}
Assume there is a partial simplicial chain homotopy $P$ of dimension $(n-1)$ between the chain maps 
\[
i^{n-1}_G, e : \Z BG_\ast \to \Z B\A^{n-1}(G).
\] 
Define $Q([\quad])=0$ and for $\sigma = [g_1,\cdots,g_m]$ where $m \in \{1,2, \cdots, n\}$,
\begin{align}
Q(\sigma):=P_{\overline{u_n},e}^{\overline{u_n},id}(\sigma)-\Sigma_{k=1}^{m-1} T_{\ast} \circ \bigtriangledown (P_k(d_{k+1} \cdots d_m (\sigma)) \otimes \overline{u_n}_{\ast} (d_0^{k}(\sigma))).\label{eq:Q}
\end{align}
Then $Q$ is a partial simplicial chain homotopy of dimension $n$ between chain maps $i^{n}_G, e : \Z BG_\ast \to \Z B\A^{n}(G)$.
\end{theorem}

\noindent {\bf Note:}  Readers are warned that, in the definition~(\ref{eq:Q}) of $Q$, $P_{\overline{u_n},e}^{\overline{u_n},id}$ and $P$ are different. Recall $P_{\overline{u_n},e}^{\overline{u_n},id}$ is a chain homotopy between the chain maps $Ed_{\overline{u_n},id}, Ed_{(\overline{u_n},e)} : \Z BG_\ast \to \Z B\A^n(G)_\ast$. On the other hand, $P$ is the hypothesized given {\em partial} chain homotopy of dimension $(n-1)$ between $i^{n-1}_G, e : \Z BG_\ast \to \Z B\A^{n-1}(G)$.

To prove Theorem~\ref{thm:induction}, we need below Lemma~\ref{lem:cylinder part} and Lemma~\ref{lem:induction part}.

\begin{lemma}\label{lem:cylinder part}
For an $m$-simplex $\sigma=[g_1,\cdots,g_m]$,
\begin{align*}
  \d P_{\overline{u_n},e}^{\overline{u_n},id} + P_{\overline{u_n},e}^{\overline{u_n},id} \d (\sigma) & =\Sigma_{i=1}^{m-1} T_{\ast} \circ \bigtriangledown (([g_1, \cdots,g_i] - [e(g_1), \cdots,e(g_i)]) \otimes [g_{i+1}^{\overline{u_n}},\cdots,g_n^{\overline{u_n}}])\\ & +([g_1,\cdots,g_m]-[e,\cdots,e]).
\end{align*}
\end{lemma}

\begin{proof}
By Definition~\ref{def:edgewise}, for an $m$-simplex $\sigma=[g_1,\cdots,g_m]$,  
\begin{align*}
Ed_{f,g}(\sigma)&= \Sigma_{i=0}^{m} T_{\ast} \circ \bigtriangledown ([g(g_1), \cdots,g(g_i)] \otimes [f(g_{i+1}),\cdots,f(g_m)])\\
                &=T_{\ast} \circ \bigtriangledown ([\quad] \otimes [f(g_{1}),\cdots,f(g_m)]) + T_{\ast} \circ \bigtriangledown ([g(g_1)] \otimes [f(g_{2}),\cdots,f(g_m)])\\
                &+ T_{\ast} \circ \bigtriangledown ([g(g_1), g(g_2)] \otimes [f(g_{3}),\cdots,f(g_m)]) + \cdots\\
                &+ T_{\ast} \circ \bigtriangledown ([g(g_1),\cdots, g(g_{m-1})] \otimes [f(g_m)])+ T_{\ast} \circ \bigtriangledown ([g(g_1),\cdots, g(g_m)] \otimes [\quad])\\
                &=[f(g_{1}),\cdots,f(g_m)] + T_{\ast} \circ \bigtriangledown ([g(g_1)] \otimes [f(g_{2}),\cdots,f(g_m)])\\
                &+ T_{\ast} \circ \bigtriangledown ([g(g_1), g(g_2)] \otimes [f(g_{3}),\cdots,f(g_m)]) + \cdots\\
                &+ T_{\ast} \circ \bigtriangledown ([g(g_1),\cdots, g(g_{m-1})] \otimes [f(g_m)])+ [g(g_1),\cdots, g(g_m)].
\end{align*}
By Theorem~\ref{thm:mitosis edgewise chain homotopy}, $\d P_{\overline{u_n},e}^{\overline{u_n},id} + P_{\overline{u_n},e}^{\overline{u_n},id} \d = Ed_{(\overline{u_n},id)} - Ed_{(\overline{u_n},e)}$. By the above observation,
\begin{align*}
(Ed_{(\overline{u_n},id)}& - Ed_{(\overline{u_n},e)})(\sigma)\\
			&= \Sigma_{i=1}^{m-1} T_{\ast} \circ \bigtriangledown ([g_1, \cdots,g_i] \otimes [g_{i+1}^{\overline{u_n}},\cdots,g_m^{\overline{u_n}}] - [e(g_1), \cdots,e(g_i)] \otimes [g_{i+1}^{\overline{u_n}},\cdots,g_m^{\overline{u_n}}])\\
           &+([g_1,\cdots,g_m]-[e,\cdots,e])\\
                &= \Sigma_{i=1}^{m-1} T_{\ast} \circ \bigtriangledown (([g_1, \cdots,g_i] - [e(g_1), \cdots,e(g_i)]) \otimes [g_{i+1}^{\overline{u_n}},\cdots,g_m^{\overline{u_n}}])\\
                &+([g_1,\cdots,g_m]-[e,\cdots,e]).
\end{align*}
\end{proof}

\begin{lemma}\label{lem:induction part}
For an $m$-simplex $\sigma=[g_1,\cdots,g_m]$,
\begin{align*}
&\d\Sigma_{k=1}^{m-1} T_{\ast} \circ \bigtriangledown (P_k(d_{k+1} \cdots d_m (\sigma)) \otimes \overline{u_n}_{\ast} (d_0^{k}(\sigma)))+\Sigma_{k=1}^{m-2} T_{\ast} \circ \bigtriangledown (P_k(d_{k+1} \cdots d_{m-1} (\d\sigma)) \otimes \overline{u_n}_{\ast} (d_0^{k}(\d\sigma)))\\
&=\Sigma_{i=1}^{m-1} T_{\ast} \circ \bigtriangledown (([g_1, \cdots,g_i] - [e(g_1), \cdots,e(g_i)]) \otimes [g_{i+1}^{\overline{u_n}},\cdots,g_m^{\overline{u_n}}]).
\end{align*}
\end{lemma}

\begin{proof}
We compute each term separately, then note cancellation in the sum.
We introduce the following notation to help with calculating: 
\begin{equation*}
\begin{split}
P_{1,2,\cdots,k}&:=P_k ([g_1,\cdots,g_k])\\
P_{1,23,\cdots,k}&:=P_k ([g_1,g_2 g_3,g_4\cdots,g_k])\\
U_{k+1,\cdots,m}&:=[g_{k+1}^{\overline{u_n}},\cdots,g_{m}^{\overline{u_n}}]\\ 
U_{(k+1)(k+2),\cdots,m}&:=[g_{k+1}^{\overline{u_n}} g_{k+2}^{\overline{u_n}} , g_{k+3}^{\overline{u_n}},\cdots,g_{m}^{\overline{u_n}}].
\end{split}	
\end{equation*}
For the first term,
\begin{align*}
 &\d \Sigma_{k=1}^{m-1} T_{\ast} \circ \bigtriangledown (P_k(d_{k+1} \cdots d_m (\sigma)) \otimes \overline{u_n}_{\ast} (d_0^{k}(\sigma)))\\
 &=\d \Sigma_{k=1}^{m-1} T_{\ast} \circ \bigtriangledown (P_k([g_1,\cdots,g_k] \otimes [g_{k+1}^{\overline{u_n}},\cdots,g_m^{\overline{u_n}}])\\
 &=T_{\ast} \circ \bigtriangledown (\Sigma_{k=1}^{m-1} ( \d ( P_k([g_1,\cdots,g_k]) \otimes [g_{k+1}^{\overline{u_n}},\cdots,g_m^{\overline{u_n}}])))\\
 &=T_{\ast} \circ \bigtriangledown (\Sigma_{k=1}^{m-1} (\d P_k([g_1,\cdots,g_k]) \otimes [g_{k+1}^{\overline{u_n}},\cdots,g_m^{\overline{u_n}}] + (-1)^{k+1} P_k([g_1,\cdots,g_k] \otimes \d [g_{k+1}^{\overline{u_n}},\cdots,g_m^{\overline{u_n}}]))\\
 &=T_{\ast} \circ \bigtriangledown (\Sigma_{k=1}^{m-1} (\d P_k([g_1,\cdots,g_k]) \otimes [g_{k+1}^{\overline{u_n}},\cdots,g_m^{\overline{u_n}}] + (-1)^{k+1} P_k([g_1,\cdots,g_k] \otimes \Sigma_{j=0}^{m-k} d_j [g_{k+1}^{\overline{u_n}},\cdots,g_m^{\overline{u_n}}]))\\
 &=T_{\ast} \circ \bigtriangledown (\textcolor{red!70!black}{\d P_{1} \otimes U_{2,\cdots,m}}\textcolor{orange}{+ P_{1} \otimes U_{3,\cdots,m}}-\Sigma_{s=2}^{n} (-1)^{s} P_{1} \otimes U_{2,3,\cdots,s-1,s(s+1),s+2,\cdots,m-1,m}\\
 &\textcolor{orange}{+\d P_{1,2} \otimes U_{3,\cdots,m}}\textcolor{magenta}{- P_{1,2} \otimes U_{4,\cdots,m}}-\Sigma_{s=3}^{n} (-1)^{s} P_{1,2} \otimes U_{3,4,\cdots,s-1,s(s+1),s+2,\cdots,m-1,m}\\
 &\textcolor{magenta}{+\d P_{1,2,3} \otimes U_{4,\cdots,m}}\textcolor{green!60!black}{+ P_{1,2,3} \otimes U_{5,\cdots,m}}-\Sigma_{s=4}^{n} (-1)^{s} P_{1,2,3} \otimes U_{4,5,\cdots,s-1,s(s+1),s+2,\cdots,m-1,m}\\
 \vdots\\
 &\textcolor{purple}{+ \d P_{1,\cdots,m-2} \otimes U_{m-1,m}}\textcolor{blue!70!black}{+(-1)^{n-2+1} P_{1,\cdots,m-2} \otimes U_{m}}-\Sigma_{s=n-1}^{n} (-1)^{s} P_{1,\cdots,m-2} \otimes U_{m-1,\cdots,s-1,s(s+1),s+2,\cdots,m}\\
 &\textcolor{blue!70!black}{+\d P_{1,\cdots,m-1} \otimes U_{m}})
\end{align*}
where our convention is like below:
\begin{center}
$
U_{2,3,\cdots,s-1,s(s+1),s+2,\cdots,m-1,m}=
    \begin{cases}
      U_{23,4,5,\cdots,m-1,m} & \text{if   } s=2 \\
      U_{2,3,\cdots,s-1,s(s+1),s+2,\cdots,m-1,m} & \text{if   } 3 \leq s \leq m-2 \\
      U_{2,3,\cdots,m-3,m-2,(m-1)m} & \text{if   } s=m-1\\
      U_{2,3,\cdots,m-2,m-1} & \text{if   } s=m.
    \end{cases}
$
\end{center}

For the second term,
\begin{align*}
&\Sigma_{k=1}^{m-2} T_{\ast} \circ \bigtriangledown (P_k(d_{k+1} \cdots d_{m-1} (\d \sigma)) \otimes \overline{u_n}_{\ast} (d_0^{k}(\d \sigma)))\\
&=\Sigma_{k=1}^{m-2} T_{\ast} \circ \bigtriangledown (P_k(d_{k+1} \cdots d_{m-1} (\Sigma_{i=0}^{m-1} (-1)^{i} d_i \sigma)) \otimes \overline{u_n}_{\ast} (d_0^{k}(\Sigma_{i=0}^{m-1} (-1)^{i} d_i \sigma)))\\
&=T_{\ast} \circ \bigtriangledown (\Sigma_{k=1}^{m-2} (P_k(d_{k+1} \cdots d_{m-1} (\Sigma_{i=0}^{m-1} (-1)^{i} d_i \sigma)) \otimes \overline{u_n}_{\ast} (d_0^{k}(\Sigma_{i=0}^{m-1} (-1)^{i} d_i \sigma))))\\
&=T_{\ast} \circ \bigtriangledown (\textcolor{orange}{P_{2}\otimes U_{3,\cdots,m}}\textcolor{orange}{-P_{12}\otimes U_{3,\cdots,m}}+\Sigma_{s=2}^{m} (-1)^{s} P_{1} \otimes U_{2,3,\cdots,s-1,s(s+1),s+2,\cdots,m-1,m}\\
&\textcolor{magenta}{+P_{2,3}\otimes U_{4,\cdots,m}}\textcolor{magenta}{-P_{12,3}\otimes U_{4,\cdots,m}}\textcolor{magenta}{+P_{1,23}\otimes U_{4,\cdots,m}}+\Sigma_{s=3}^{m} (-1)^{s} P_{1,2} \otimes U_{3,4,\cdots,s-1,s(s+1),s+2,\cdots,m-1,m}\\
&\textcolor{green!60!black}{+P_{2,3,4}\otimes U_{5,\cdots,m}}\textcolor{green!60!black}{-P_{12,3,4}\otimes U_{5,\cdots,m}}\textcolor{green!60!black}{+P_{1,23,4}\otimes U_{5,\cdots,m}}\textcolor{green!60!black}{-P_{1,2,34}\otimes U_{5,\cdots,m}}\\
&+\Sigma_{s=4}^{m} (-1)^{s} P_{1,2,3} \otimes U_{4,5,\cdots,s-1,s(s+1),s+2,\cdots,m-1,m}\\
&\vdots\\
&\textcolor{purple}{+P_{2,\cdots,m-2}\otimes U_{m-1,m}}\textcolor{purple}{-P_{12,3,\cdots,m-2}\otimes U_{m-1,m}}\textcolor{purple}{+P_{1,23,4\cdots,m-2}\otimes U_{m-1,m}}\textcolor{purple}{-\cdots}\\
&\textcolor{purple}{+(-1)^{m-1}P_{1,2,\cdots,(m-3)(m-2)}\otimes U_{m-1,m}}\\
&+\Sigma_{s=m-2}^{m} (-1)^{s} P_{1,\cdots,m-2} \otimes U_{m-2,\cdots,s-1,s(s+1),s+2,\cdots,m}\\
&\textcolor{blue!70!black}{+P_{2,\cdots,m-1}\otimes U_{m}}\textcolor{blue!70!black}{-P_{12,3,\cdots,m-1}\otimes U_{m}}\textcolor{blue!70!black}{+P_{1,23,4,\cdots,m-1}\otimes U_{m}}\textcolor{blue!70!black}{-\cdots}\textcolor{blue!70!black}{+(-1)^{m}P_{1,2,\cdots,(m-2)(m-1)}\otimes U_{m}}\\
&+\Sigma_{s=m-1}^{m} (-1)^{s} P_{1,\cdots,m-2} \otimes U_{m-1,\cdots,s-1,s(s+1),s+2,\cdots,m}).
\end{align*}
By adding $\d \Sigma_{k=1}^{m-1} T_{\ast} \circ \bigtriangledown (P_k(d_{k+1} \cdots d_m (\sigma)) \otimes \overline{u_n}_{\ast} (d_0^{k}(\sigma)))$ and $\Sigma_{k=1}^{m-2} T_{\ast} \circ \bigtriangledown (P_k(d_{k+1} \cdots d_{m-1} (\d \sigma)) \otimes \overline{u_n}_{\ast} (d_0^{k}(\d \sigma)))$, the sum of the same color terms turns out to be $(\d P + P \d) \otimes U_{\cdots}$ and terms in the black color turn out to be canceled out. In other words,
\begin{align*}
&\d \Sigma_{k=1}^{m-1} T_{\ast} \circ \bigtriangledown (P_k(d_{k+1} \cdots d_m (\sigma)) \otimes \overline{u_n}_{\ast} (d_0^{k}(\sigma))) + \Sigma_{k=1}^{m-2} T_{\ast} \circ \bigtriangledown (P_k(d_{k+1} \cdots d_{m-1} (\d \sigma)) \otimes \overline{u_n}_{\ast} (d_0^{k}(\d \sigma)))\\
&=T_{\ast} \circ \bigtriangledown (\textcolor{red!70!black}{(\d P_1 + P_0 \d) \otimes U_{2,\cdots,m}}+\textcolor{orange}{(\d P_2 + P_1 \d) \otimes U_{3,\cdots,m}}\\
&+\textcolor{magenta}{(\d P_3 + P_2 \d) \otimes U_{4,\cdots,m}}+\textcolor{green!60!black}{(\d P_4 + P_3 \d) \otimes U_{5,\cdots,m}}+\cdots\\
&+\textcolor{purple}{(\d P_{m-2} + P_{m-3} \d) \otimes U_{m-1,m}}+\textcolor{blue!70!black}{(\d P_{m-1} + P_{m-2} \d) \otimes U_m})\\
&=T_{\ast} \circ \bigtriangledown (\textcolor{red!70!black}{([g_1]-[e(g_1)]) \otimes U_{2,\cdots,m}}+\textcolor{orange}{([g_1,g_2]-[e(g_1),e(g_2)]) \otimes U_{3,\cdots,m}}\\
&+\textcolor{magenta}{([g_1,g_2,g_3]-[e(g_1),e(g_2),e(g_3)]) \otimes U_{4,\cdots,m}}\\
&+\textcolor{green!60!black}{([g_1,g_2,g_3,g_4]-[e(g_1),e(g_2),e(g_3),e(g_4)]) \otimes U_{5,\cdots,m}}\\
&+\cdots\\
&+\textcolor{purple}{([g_1,\cdots,g_{m-2}]-[e(g_1),\cdots,e(g_{m-2})]) \otimes U_{m-1,m}}\\
&+\textcolor{blue!70!black}{([g_1,\cdots,g_{m-1}]-[e(g_1),\cdots,e(g_{m-1})]) \otimes U_m})\\
&=\Sigma_{i=1}^{m-1} T_{\ast} \circ \bigtriangledown (([g_1, \cdots,g_i] - [e(g_1), \cdots,e(g_i)]) \otimes [g_{i+1}^{\overline{u_n}},\cdots,g_m^{\overline{u_n}}]).
\end{align*}
\end{proof}

\begin{proof}[Proof of Theorem~\ref{thm:induction}]
By the definition(~\ref{eq:Q}) of $Q$,
\[
(\d Q + Q \d)([\quad])=0=[\quad]-[\quad].
\]
Applying Lemma~\ref{lem:cylinder part} and Lemma~\ref{lem:induction part} to the definition(~\ref{eq:Q}) of $Q$, for any $m \in \{ 1,2,\cdots,n \}$,
\[
(\d Q + Q \d)([g_1,\cdots,g_m]) = [g_1,\cdots,g_m]-[e,\cdots,e].
\]
\end{proof}

\subsection{Proof of Theorem~\ref{thm:family}, Theorem~\ref{thm:new}, and Theorem~\ref{thm:kill}}

Using Theorem ~\ref{thm:induction}, we prove Theorem~\ref{thm:family}.

\begin{proof}[Proof of Theorem~\ref{thm:family}]
For a group $G$, we inductively construct $\Psi_G^n$. For brevity, we say $\Psi^n=\Psi_G^n$.

For $n=0$, define $\Psi_0^0([\quad]) := 0$. Then $\{\Psi_0^0\}$ is a partial simplicial chain homotopy of dimension $0$ between the chain maps
\[
i^0_G, e \colon \Z BG_\ast \to \Z B\A^0(G)_\ast.
\]
Notice $d(0)=0$.

For $n=1$, define $\Psi_0^1([\quad]) := 0$ and $\Psi_1^1([g]) := [t^{\overline{u}}, {g}^{\overline{u_n}}] - [{g}^{\overline{u_n}}, t^{\overline{u} \cdot \overline{g}}] + [t^{\overline{u} \cdot \overline{g}}, e] - [g, t^{\overline{u}}]$ for each $g\in G$. Then $\{\Psi_0^1, \Psi_1^1\}$ is a partial simplicial chain homotopy of dimension $1$ between the chain maps
\[
i^1_G, e \colon \Z BG_\ast \to \Z B\A(G)_\ast.
\]
Notice $d(0)=0$ and $d(1)=4$.

Applying Theorem~\ref{thm:induction}, we can inductively obtain $\Psi^n$ for each $n$. For instance, $\Psi([\quad]):=0$ and for $\sigma = [g_1,\cdots,g_m]$ where $m \in \{ 1,2, \cdots, n \}$,
\begin{align}
\Psi^n(\sigma):=P_{\overline{u_n},e}^{\overline{u_n},id}(\sigma)-\Sigma_{k=1}^{m-1} T_{\ast} \circ \bigtriangledown (\Psi_k^{n-1}(d_{k+1} \cdots d_m (\sigma)) \otimes \overline{u_n}_{\ast} (d_0^{k}(\sigma))).\label{eq:Psi}
\end{align}
For any $m$-simplex $\sigma$, the diameters of $T_{\ast}(\sigma)$, $d_{i}(\sigma)$, and $\overline{u_{n_{\ast}}}(\sigma)$ is $1$. By Theorem~\ref{thm:edgewise chain homotopy}, the diameter function of $P_{\overline{u_n},e}^{\overline{u_n},id}$ is $2^{m} \cdot (m+1)$ for $m \in \{ 1,2, \cdots, n \}$. For a $p$-simplex $\sigma$ and a $q$-simplex $\tau$, the diameter of $\bigtriangledown(\sigma \otimes \tau)$ is the number of $(p,q)$-shuffle, $\binom{p+q}{q}$. Thus, for an $m$-simplex $\sigma$, the diameter of $\Psi^n(\sigma)$ is given by the recurrence formula of $\gamma(m)$:
\[
\gamma(m) = 2^{m} \cdot (m+1) + \Sigma_{k=1}^{m-1} \gamma (k) \cdot \binom{m+1}{m-k}.
\]
\end{proof}

We now prove Theorem~\ref{thm:new}.

\begin{proof}[Proof of Theorem~\ref{thm:new}]
We keep counting the number of degenerate simplices of $\Psi^n(\sigma)$ defined in the definition~(\ref{eq:Psi}). Recall
 \[P_{\overline{u_n},e}^{\overline{u_n},id}(\sigma)= Cyl(Ed_{(\overline{u_n},id)}(\sigma), Ed_{(\overline{u_n},e)}(\sigma), T(\sigma)).
 \]
 Notice, for an $m$-simplex $\sigma = [g_1,\cdots,g_m]$, $Ed_{(\overline{u_n},e)}(\sigma)$ has $2^m$ $m$-simplices. Since an $m$-simplex of $Ed_{(\overline{u_n},e)}(\sigma)$ is a $(p,q)$-shuffle of $g_1,\cdots,g_p$ and $e,\cdots,e$ ($q$ $e$'s) where $p+q=m$, there are $2^{m-i}$ $m$-simplices of $Ed_{(\overline{u_n},e)}(\sigma)$ which has the first $e$ at the $i$-th component. By Definition~\ref{def:simplicial cylinder}, if the bottom of the cylinder is an $m$-simplex which has the first $e$ at the $i$-th component, then there are $m+1-i$ degenerate $(m+1)$-simplices in the cylinder. Thus, $Cyl(Ed_{(\overline{u_n},id)}(\sigma), Ed_{(\overline{u_n},e)}(\sigma), T(\sigma))$ has $\Sigma_{i=1}^{m} 2^{m-i} (m-i+1)$ degenerate $(m+1)$-simplices. In other words, $P_{\overline{u_n},e}^{\overline{u_n},id}(\sigma)$ has $2^m (m-1)+1$ degenerate $(m+1)$-simplices.
 
For a $p$-simplex $\sigma$ and a $q$-simplex $\tau$, notice a $(p+q)$-simplex in $T_{\ast} \circ \bigtriangledown (\sigma \otimes \tau)$ has the same components of $\sigma$ and $\tau$. If $\sigma$ is degenerate, then every simplex of $T_{\ast} \circ \bigtriangledown (\sigma \otimes \tau)$ is degenerate. For a $p$-simplex $\sigma$ and a $q$-simplex $\tau$ the diameter of $\bigtriangledown(\sigma \otimes \tau)$ is $\binom{p+q}{q}$. Recall the diameter function of $T_{\ast}$ is $1$. 

Thus, for an $m$-simplex $\sigma$, the number of degenerate $(m+1)$-simplices in $\Psi^n(\sigma)$ is given by the recurrence formula of $q(m)$:
\[
q(m) = 2^{m} \cdot (m-1) + 1 + \Sigma_{k=1}^{m-1} q(k) \cdot \binom{m+1}{m-k}.
\]
\end{proof}

And lastly, we are ready to prove Theorem~\ref{thm:kill}.

\begin{proof}[Proof of Theorem~\ref{thm:kill}]
We keep using $\Psi_G^n$ from Theorem~\ref{thm:new}. By Theorem~\ref{thm:family}, for a group $G$ and each $n$, we have a partial chain homotopy $\Psi_G^n$ of dimension $n$ between the chain maps $i^n_G, e \colon  \Z BG_\ast \to \Z B\A^n(G)_\ast$ whose diameter function is given by $\gamma(m)$ for each $m \in \{0,1,2,\cdots,n\}$. By Theorem~\ref{thm:new}, for an $m$-simplex $\sigma$, there are $q(m)$ degenerate $(m+1)$-simplices in $\Psi_G^n(\sigma)$. Recall, for a simplicial set $X$, there is the projection $p \colon  \Z X_{\ast} \to C_{\ast}(X)$, which sends degenerate simplices to zero (see Theorem~\ref{thm:projection from the Moore complex}). Define $\Phi_G^n$ by composition $p \circ \Psi_G^n$. Since the projection $p$ is a chain map, $\Phi_G^n$ is a partial chain homotopy of dimension $n$ between the chain maps $i^n_G, e \colon  \Z BG_\ast \to C_{\ast}(  B\A^n(G) )$. Furthermore, since the projection $p$ sends degenerate simplices to zero, $\Phi_G^n$ is controlled by $c(m) := \gamma (m) - q(m)$ for each $m \in \{0,1,2,\cdots,n\}$.
\end{proof}

\section{Stronger bounds for space forms}

In this chapter, we construct a rationalized $4$-chain for spherical $3$-manifolds and prove Theorem~\ref{thm:space forms} and Theorem~\ref{thm:RelativeRes}.

\subsection{A rationalized \texorpdfstring{$4$}{4}-chain for spherical \texorpdfstring{$3$}{3}-manifolds}

In the proof of Theorem~\ref{thm:general}, we obtain the desired $4$-chain $u$ to apply Theorem~\ref{thm:975} by using the embedding $i_{\pi}^3$ where $\pi=\pi_1(M)$. For spherical $3$-manifolds, we directly construct a rationalized $4$-chain $u$ instead of using the embedding $i_{\pi}^3$.

Since a spherical $3$-manifold $M$ has a finite fundamental group, the universal covering map $p \colon  S^3 \to M$ is a $\lvert \pi_1(M) \rvert$-fold covering map. Since a triangulation of the base can be lifted to a triangulation of the covering space, the universal cover $S^3$ of $M$ is a triangulated $3$-sphere with complexity equal to $\lvert \pi_1 (M) \rvert \cdot d(\zeta_M)$. Since $S^3$ bounds a $4$-ball $B^4$, we can obtain a triangulation of $B^4$ as the cone on the triangulation of $S^3$. Notice that the simplicial complexity of $B^{4}$ is $\lvert \pi_{1}(M)\rvert \cdot n$ where $n=d(\zeta_M)$. We denote by $B_M^4$ the $4$-ball endowed with the triangulation induced by $M$.

We define a {\em rationalized} simplicial-cellular complex of a spherical $3$-manifold.

\begin{definition}
For a spherical $3$-manifold $M$ endowed with a triangulation, the {\em rationalized} simplicial-cellular complex $K_M$ of $M$ is defined by
\[
K_M := B_M^4 \diagup \sim
\]
where $p\colon S^3 \to M$ is the $\lvert \pi_1(M) \rvert$-fold universal covering and $x \sim y$ if $x, y \in S^{3}=\partial B_M^4$ and $p(x)=p(y)$.
\end{definition}

\begin{remark}
Even if $K_M$ is not a manifold, $K_M$ is still a simplicial-cellular complex because $B_M^4$ is triangulated and $p$ sends every $3$-simplex of $S^3=\d B_M^4$ to a $3$-simplex of $M$. In other words, this is because we construct the triangulation of $S^{3}$ by lifting the triangulation of $M$ using the covering map $p$.  We then take the quotient of $S^{3}$ determined by the same map $p$. Moreover there is a natural inclusion $i \colon  M \to K_M$ which is simplicial-cellular. Notice the simplicial complexity of $K_M$ is $\lvert \pi_{1}(M)\rvert \cdot n$.
\end{remark}

\begin{remark}
Notice $\pi_1(K_M)=\pi_1(M)$ because $K_M$ is obtained by attaching a $4$-ball to $M$.
\end{remark}

We need a lemma to proceed.

\begin{lemma}\label{lemma:zero}
For the induced homomorphism of the natural inclusion $i_{\ast} \colon  H_{3}(M) \to H_{3}(K_M)$, 
\[
i_{\ast}(\lvert \pi_{1}(M)\rvert \cdot [M])=0
\]
where $[M]$ is the fundamental class of $M$.
\end{lemma}

\begin{proof}
Since the natural inclusion $i \colon M \to K_M$ is simplicial-cellular (recall Definition~\ref{def:simplicial-cellular}), we consider the long exact homology sequence of pair $(K_M,M)$.
\[
\cdots \to H_{4}(K_M,M) \xrightarrow{\partial} H_{3}(M) \to H_{3}(K_M) \to H_{3}(K_M,M) \to \cdots
\]
Notice $H_{\ast}(K_M,M) = H_{\ast}(S^{4})$ by excision. Also, $H_{3}(M) = \Z$ since $M$ is a orientable closed $3$-manifold. Thus, we have
\[
\cdots \to \Z \xrightarrow{\partial} \Z \to H_{3}(K_M) \to 0 \to \cdots.
\]
By construction, $\partial = p_{\ast}$. Since $p$ is a finite $\lvert \pi_{1}(M) \rvert$-cover of $M$, $p_{\ast}$ is just the multiplication by $\lvert \pi_{1}(M) \rvert$. In other words, $H_{3}(K_M) = \Z_{\lvert \pi_{1}(M) \rvert}$. Thus, $i_{\ast}(\lvert \pi_{1}(M)\rvert \cdot [M])=0$.
\end{proof}

We combine this with Theorem~\ref{thm:975} and use the Definition~\ref{def:rho} of the $\rho$-invariants.

\begin{proof}[Proof of Theorem~\ref{thm:space forms}]
Lemma~\ref{lemma:zero} asserts that there is a $4$-chain $u \in C_{4}(K_M)$ satisfying $\partial u = i_{\sharp}(\lvert \pi_{1}(M)\rvert \cdot \zeta_{M})$. In fact, $u = [ K_M ]$ and $d(u) = \lvert \pi_{1}(M)\rvert \cdot n$. By theorem~\ref{thm:975}, there exists a bordism $W$ over $K_M$ between $\coprod\limits^r M$ and a trivial end, whose $2$-handle complexity is at most $195 \cdot d(\zeta_{\coprod M})+975 \cdot d(u)$ where $r=\lvert \pi_{1}(M) \rvert$. Notice simplicial complexities of $K_M$ and $\coprod\limits^r M$ equate $\lvert \pi_{1}(M)\rvert \cdot n$. Thus, the $2$-handle complexity of $W$ is at most $195 \cdot r \cdot n +975 \cdot r \cdot n$. In other words, we obtain $W$ which makes the below diagram commute for some homomorphism $f$:

\adjustbox{scale=1.1,center}{
\begin{tikzcd}
\pi_1(\coprod\limits^r M)  \arrow[dd, "i_{\ast}"] \arrow[rd, "\coprod id_{\pi}"] &   &  \\
                                   & \pi_1(K_M)               \\
\pi_1(W)  \arrow[ru, "f"]            &         
\end{tikzcd}}
For a given homomorphism $\varphi\colon \pi_1(M) \to G$, we can extend the above diagram indicated below. Notice the left triangle is independent of the given representation of $\varphi$.

\adjustbox{scale=1.1,center}{
\begin{tikzcd}
\pi_1(\coprod\limits^r M) \arrow[rr, "\coprod \varphi"] \arrow[dd, "i_{\ast}"] \arrow[rd, "\coprod id_{\pi}"] &   & G \arrow[dd, hook, "id_{G}"] \\
                                   & \pi_1(K_M) \arrow[ru, "\varphi"] \arrow[rd, "\varphi"] &              \\
\pi_1(W) \arrow[rr, "\varphi \circ f"] \arrow[ru, "f"]            &                         & G
\end{tikzcd}}
As discussed in Remark~\ref{rmk:trivial end}, a trivial end does not affect to the $L^2$ $\rho$-invariant $\rho(M,\varphi)$. In other words, $\rho^{(2)}(M,\varphi)=\frac{1}{r} (\textup{sign}_{\Gamma}^{(2)}(W) - \textup{sign}(W))$. Thus,
\begin{align*}
\lvert \rho^{(2)}(M,\varphi) \rvert&=\frac{1}{r} \lvert \textup{sign}_{\Gamma}^{(2)}(W) - \textup{sign}(W) \rvert\\
                                &\leq \frac{1}{r} (\lvert \textup{sign}_{\Gamma}^{(2)}(W) \rvert + \lvert \textup{sign}(W) \rvert)\\
                                &\leq \frac{1}{r} \cdot 2 \cdot (195 \cdot r \cdot n +975 \cdot r \cdot n)\\
                                &=2340 \cdot n.
\end{align*}
This completes the proof of Theorem~\ref{thm:space forms}.
\end{proof}

By the classical work of Perelman~\cite{Per02}~\cite{Per03a}~\cite{Per03b}, it is known that every closed $3$-manifold with a finite fundamental group is a spherical $3$-manifold and vice versa. Thus, we can rephrase Theorem~\ref{thm:space forms} as follows.

\begin{corollary}\label{thm:finite fundamental}
Let $M$ be a closed, oriented $3$-manifold with simplicial complexity $n$. If $M$ has a finite fundamental group, Then
\begin{center}
$\lvert$ $\rho^{(2)}(M,\varphi)$ $\rvert$ $\leq$ $2340$ $\cdot$ $n$ 
\end{center}
for any homomorphism $\varphi\colon \pi_1(M)\to G$ and any group G.
\end{corollary}

\subsection{Specific bounds for \texorpdfstring{$3$}{3}-manifolds with a map to a spherical \texorpdfstring{$3$}{3}-manifold}

Our methods of proving Theorem~\ref{thm:space forms} extend to prove Theorem~\ref{thm:RelativeRes}.

\begin{theorem}\label{thm:RelativeResGeneral} 
Let $M$ and $N$ be closed, oriented, triangulated $3$-manifolds. Assume $M$ has the simplicial complexity $n$ and $N$ is a spherical space form. Suppose $\varphi \colon M \to N$ is a simplicial-cellular map with nonzero degree. Then,
\begin{center}
$\lvert$ $\rho^{(2)}(M, \varphi_{\ast})$ $\rvert$ $\leq$ $2(195+975 \cdot \lvert deg(\varphi) \rvert)$ $\cdot$ $n$,
\end{center}
where $\varphi_{\ast} \colon \pi_1(M) \to \pi_1(N)$ is the induced homomorphism by $\varphi$.
\end{theorem}

\begin{proof}
Let $\lvert \Pi (N) \rvert$ be $r$. Since $N$ is a space form, there is a universal covering map $p \colon  S^3 \to N$. In the same way to prove Theorem~\ref{thm:space forms}, we define $K = B^{4}/\sim$ where $x \sim y$ if $x, y \in S^{3}=\partial B^{4}$ and $p(x)=p(y)$. Notice there is a natural inclusion $i\colon N \to K$. By Lemma~\ref{lemma:zero},
\[
i_{\ast}(r \cdot [N])=\d B^4 \in C_3(K), \text{ i.e. } i_{\ast}(r \cdot [N])=0 \in H_3(K).
\]
Since $\varphi \colon M \to N$ is a simplicial-cellular map with a nonzero degree, \[
{\id_{G}}_{\ast} \circ {\coprod \varphi}_{\ast} (r \cdot [M]) = {\id_{G}}_{\ast} (r \cdot \lvert deg(\varphi) \rvert \cdot [N]) = \lvert deg(\varphi) \rvert \cdot \d B^4 \in C_3(K)
\] 
where 
\[
G=\pi_1(N)=\pi_1(K),  \text{ i.e. }  {\id_{G}}_{\ast} \circ {\coprod \varphi}_{\ast} (r \cdot [M]) = 0 \in H_3(K).
\]
By Theorem~\ref{thm:975}, we obtain a bordism $W$ between $\coprod\limits ^r M$ and a trivial end and a map $\overline{\varphi}$ which make the diagram below commute.
$$
\adjustbox{scale=1.1,center}{%
\begin{tikzcd}
\pi_1(\coprod\limits^r M) \arrow[r, "{\coprod \varphi}_{\ast}"] \arrow[dd, "i_{\ast}"]
& \pi_1(N)=G \arrow[dd, hook, "id_{G}"] \\
& \\
\pi_1(W) \arrow[r, dashrightarrow, "\overline{\varphi}_\ast"]
& \pi_1(K)=G
\end{tikzcd}
}
$$
Moreover, the $2$-handle complexity of $W$ is given by $195 \cdot r \cdot n+975 \cdot \lvert deg(\varphi) \rvert \cdot r \cdot n$. Thus,
\begin{align*}
\lvert \rho^{(2)}(M,\varphi) \rvert&=\frac{1}{r} \lvert \textup{sign}_{\Gamma}^{(2)}(W) - \textup{sign}(W) \rvert\\
                                &\leq \frac{1}{r} (\lvert \textup{sign}_{\Gamma}^{(2)}(W) \rvert + \lvert \textup{sign}(W) \rvert)\\
                                &\leq \frac{1}{r} \cdot 2 \cdot (195 \cdot r \cdot n+975 \cdot \lvert deg(\varphi) \rvert \cdot r \cdot n)\\
                                &=2(195+975 \cdot \lvert deg(\varphi) \rvert) \cdot n.
\end{align*}
\end{proof}
As a corollary, we show Theorem~\ref{thm:RelativeRes}.

\begin{proof}[Proof of Theorem~\ref{thm:RelativeRes}]
Theorem~\ref{thm:RelativeRes} follows straightforwardly  by applying Theorem~\ref{thm:RelativeResGeneral} since $f \colon M \to N$ has degree 1.
\end{proof}

Theorem~\ref{thm:RelativeRes} is an unexpected extension of Theorem~\ref{thm:space forms} since we make {\em no assumption} that $M$ is spherical.  For example, Luft and Sjerve~\cite{LS} construct homology spheres with infinite fundamental group which satisfy the hypothesis of Theorem~\ref{thm:RelativeRes} from any $2n \times 2n$ matrix $A$ with determinant $1$ and such that $A^2 - I$ is invertible.  They give examples for $n = 2, 3$, but presumably such examples exist for all $n \geq 2$. For details and related discussions we refer the reader to~\cite{LS}.

Theorem~\ref{thm:RelativeRes} does not provide universal bounds since the homomorphism $f_{\ast}$ is not necessarily an inclusion. And we should study whether a map in Luft and Sjerve's construction is homotopic to a simplicial-cellular map from $M$ endowed with a triangulation of the minimal complexity. However, Theorem~\ref{thm:RelativeRes} includes Theorem~\ref{thm:space forms} as a special case, by letting $M$ be spherical and $\varphi_{\ast}$ the identity.

\section{Application}

In this chapter, we introduce various applications for the complexity theory of $3$-manifolds.

\subsection{The complexity of lens spaces}

In this section, we apply Theorem~\ref{thm:space forms} to the complexity of lens spaces and explain how to obtain Theorem~\ref{thm:complexity}.

We first define the pseudo-simplicial triangulation and the pseudo-simplicial complexity.

\begin{definition}\label{def:pseudo-simplicial triangulation}
For a $3$-manifold, a {\em pseudo-simplicial triangulation} is a collection of $3$-simplices whose faces are identified in pairs under affine homeomorphisms to be the $3$-manifold as a quotient space.
\end{definition}

\begin{definition}\label{def:pseudo-simplicial complexity}
For a $3$-manifold $M$, the {\em pseudo-simplicial complexity} $c(M)$ is the minimal number of $3$-simplices in a pseudo-simplicial triangulation of $M$.
\end{definition}

In \cite{Mat}, Matveev defines the complexity of a compact $3$-manifold in the language of simple spines and handle decompositions, which turns out to be equal to the pseudo-simplicial complexity for closed irreducible $3$-manifolds except $S^3, \R P^3,$ and $L(3,1)$. In \cite{JR06}, Jaco and Rubinstein independently investigate the complexity using the notion of layered-triangulations. In \cite{Mat} and \cite{JR06}, Matveev and Jaco-Rubinstein independently conjectures that, for $p > q > 0, p > 3$, $c(L(p,q))=S(p,q) -3$ where $S(p,q)$ is the sum of all partial quotients in the continued fraction expansion of $\frac{p}{q}$. In the case of $q=1$, this specializes to the following conjecture.



\begin{conjecture}\label{conj:Mat}
\textup{(Matveev~\cite{Mat}, Jaco and Rubinstein~\cite{JR06}).} For $n > 3$, $c(L(n,1)) = n-3$.
\end{conjecture}

By Jaco and Rubinstein, it is known that $c(L(n,1)) \leq n-3$~\cite{JR06}. Jaco, Rubinstein, and Tillman proved the conjecture in case of $n$ is even~\cite{JRT09}. Cha provides general lower bounds by using Theorem~\ref{thm:Cha}.

\begin{genthm*}[definition]{Theorem~\#~1.14~from~\cite{Cha16}}
For $n > 3$,
\[
\frac{1}{627419520} \cdot (n-3) \leq c(L(n,1)) \leq n-3.
\]
\end{genthm*}

\begin{remark}\label{rmk:pseudo-simplicial complexity}
We notice that the second barycentric subdivision of a pseudo-simplicial triangulation is an usual triangulation. The second barycentric subdivision of a $3$-simplex turns out to be $(4!)^2=576$ $3$-simplices.
Thus, for a $3$-manifold with the simplicial complexity $n$, $c(M) \leq 576 \cdot n$. As a corollary of Theorem~\ref{thm:Cha}, for any representation $\varphi$ of $M$, one immediately obtains
\[
c(M) \geq \frac{1}{209139840} \cdot \lvert \rho^{(2)}(M,\varphi) \rvert .
\]
Since $L^2$ $\rho$-invariants of lens spaces $L(n,1)$ is known (see ~\cite[Lemma 7.1]{Cha16}), Cha obtains the lower bounds.
\end{remark}

Now, we explain how to obtain Theorem~\ref{thm:complexity}. Cha uses the coefficient of 363090 in Theorem~\ref{thm:Cha} as a factor to obtain $627419520=363090\cdot1728$ in his proof (see~\cite[Theorem 1.14]{Cha16}). Since lens spaces are spherical, we can apply Theorem~\ref{thm:space forms} in place of Theorem~\ref{thm:Cha} and improve on his results. Specifically, we can replace Cha's coefficient, $363090$, with our coefficient, $2340$. Thus, we obtain $4043520=2340\cdot1728$ instead of $627419520$. In other words, for $n > 3$,
\[
\frac{1}{4043520} \cdot (n-3) \leq c(L(n,1)) \leq n-3.
\]

This lower bound is roughly $155$ times larger than the lower bound derived by Cha's Theorem~\#~1.14 \linebreak from~\cite{Cha16} given above. It exploits an essential difference between spherical $3$-manifolds and $3$-manifolds having infinite fundamental groups.


Recently, a more direct approach to prove Matveev and Jaco-Rubinstein's conjecture is developed by Lackenby and Purcell~\cite{LP}. They show that, for a closed orientable hyperbolic $3$-manifold that fibers over the circle, the triangulation complexity is equal to the translation length of the monodromy action on the mapping class group of the fiber. This provides upper and lower bounds on the triangulation complexity.

\subsection{Cha's bounds for complexities}

We conclude this paper by stating improvements on Theorems of Cha~\cite{Cha16} obtained using arguments from~\cite{Cha16}, but replacing bounds from that paper with our improved bounds.

In~\cite{Cha16}, Cha introduced a relation between $\rho$-invariants and Heegaard-Lickorish complexity.

\begin{genthm*}[definition]{Theorem~\#~1.8~from~\cite{Cha16}}
If $M$ is a closed $3$-manifold and the Heegaard-Lickorish complexity is $\ell$, then
\[
\lvert \rho^{(2)}(M,\varphi) \rvert \leq 251258280 \cdot \ell
\]
for any homomorphism $\varphi\colon \pi_1(M)\to G$ to any group G.
\end{genthm*}

Using Theorem~\ref{thm:general}, we can enhance the upper bound.

\begin{theorem}
If $M$ is a closed $3$-manifold with Heegaard-Lickorish complexity $\ell$, then
\[\lvert \rho^{(2)}(M,\varphi) \rvert \leq 191884680 \cdot \ell
\]
for any homomorphism $\varphi\colon \pi_1(M)\to G$ to any group G.
\end{theorem}

We discuss a $3$-manifold obtained from surgery along a framed link. 

\begin{definition}
For a link $L$, $c(L)$ is the crossing number.
\end{definition}

\begin{definition}
For a framed link $L$, $f(L):=\Sigma_i \lvert n_i \rvert$ where $n_i \in \Z$ is the framing on the $i$-th component of $L$.
\end{definition}

\begin{genthm*}[definition]{Theorem~\#~1.9~from~\cite{Cha16}}
Suppose $M$ is a closed $3$-manifold with surgery along a framed link $L$ in $S^{3}$. Then
\[
\lvert \rho^{(2)}(M,\varphi) \rvert \leq 69713280 \cdot c(L) + 34856640 \cdot f(L)
\]
for any homomorphism $\varphi\colon \pi_1(M)\to G$ to any group G.
\end{genthm*}

Using Theorem~\ref{thm:general}, we can have a better bound.

\begin{theorem}
Suppose $M$ is a closed $3$-manifold with surgery along a framed link $L$ in $S^{3}$. Then
\[
\lvert \rho^{(2)}(M,\varphi) \rvert \leq 53239680 \cdot c(L) + 26619840 \cdot f(L)
\]
for any homomorphism $\varphi\colon \pi_1(M)\to G$ to any group G.
\end{theorem}

\begin{genthm*}[definition]{Theorem~\#~6.4~from~\cite{Cha16}} Suppose $D$ is a planar diagram of a link $L$ with $c$ crossings in which each component is involved in a crossing. Let $M$ be the $3$-manifold obtained by surgery on $L$ along the blackboard framing of $D$. Then
\[
\lvert \rho^{(2)}(M,\varphi) \rvert \leq 34856640 \cdot c
\]
for any homomorphism $\varphi\colon \pi_1(M)\to G$ to any group G.
\end{genthm*}

We continue to observe improvements on bounds in~\cite{Cha16} obtained by applying Theorem~\ref{thm:general} in place of bounds from Cha's paper.

\begin{theorem}
Suppose $D$ is a planar diagram of a link $L$ with $c$ crossings in which each component is involved in a crossing. Let $M$ be the $3$-manifold obtained by surgery on $L$ along the blackboard framing of $D$. Then
\[
\lvert \rho^{(2)}(M,\varphi) \rvert \leq 26619840 \cdot c
\]
for any homomorphism $\varphi\colon \pi_1(M)\to G$ to any group G.
\end{theorem}

Cha discussed that the Cheeger-Gromov bounds are asymptotically optimal.

\begin{definition}
$B^{HL}(\ell):=$sup$\{\lvert\rho^{(2)}(M,\varphi)\rvert : M$ has Heegaard-Lickorish complexity $\leq \ell$ and $\varphi$ is a homomorphism of $\pi_1(M)\}$.
\end{definition}

\begin{definition}
$B^{surg}(k):=$sup$\{\lvert\rho^{(2)}(M,\varphi)\rvert : M$ has surgery complexity $\leq k$ and $\varphi$ is a homomorphism of $\pi_1(M)\}$.
\end{definition}

\begin{genthm*}[definition]{Theorem~\#~7.4~from~\cite{Cha16}}
\[
\frac{1}{3} \leq \limsup_{\ell \to \infty} \frac{B^{HL}(\ell)}{\ell} \leq 251258280
\]
and
\[
\frac{1}{3} \leq \limsup_{k \to \infty} \frac{B^{surg}(k)}{k} \leq 34856640.
\]
\end{genthm*}

Using Theorem~\ref{thm:general}, we can have a better bound.

\begin{theorem}
\[
\frac{1}{3} \leq \limsup_{\ell \to \infty} \frac{B^{HL}(\ell)}{\ell} \leq 191884680
\]
and
\[
\frac{1}{3} \leq \limsup_{k \to \infty} \frac{B^{surg}(k)}{k} \leq 26619840.
\]
\end{theorem}

Lastly, Cha claims that the 2-handle complexity 195 $\cdot$ $d(\zeta_M)$+975 $\cdot$ $d(u)$ in Theorem~\ref{thm:975} is asymptotically best possible. Furthermore, he introduces the optimal value $B^{2h}(k)$ and provides upper and lower bounds asymptotically.

\begin{definition}\textup{(Cha~\cite[Definition 7.5]{Cha16}).} Denote by $M(k)$ the collection of pairs $(M, \varphi)$ of a closed triangulated $3$-manifold $M$ and a simplicia-cellular map $\varphi \colon M \to BG$ admitting a $4$-chain $u \in C_4(BG)$ such that $\d u=\varphi_{\#}(\zeta_M)$ and $k=d(\zeta_M)+d(u)$. For a given $(M, \varphi)$, let $B(M,\varphi)$ be the collection of bordisms $W$ over $G$ between $M$ and a trivial end. Define
\[
B^{2h}(k):=\sup_{(M,\varphi)\in M(k)} \min_{W \in B(M,\varphi)} \{2\text{-handle complexity of } W\}.
\]
\end{definition}

\begin{genthm*}[definition]{Theorem~\#~7.6~from~\cite{Cha16}}
\[
\frac{1}{107712} \leq \limsup_{k \to \infty} \frac{B^{2h}(k)}{k} \leq 975\]
\end{genthm*}

Using Theorem~\ref{thm:kill}, we can have a better lower bound which is roughly two times larger.

\begin{theorem}
\[
\frac{1}{56448} \leq \limsup_{k \to \infty} \frac{B^{2h}(k)}{k} \leq 975\]
\end{theorem}

\bibliographystyle{amsalpha-order}
\bibliography{research.bib}

\end{document}